\documentclass[twoside,11pt,reqno]{amsart}
\usepackage{amsmath,amssymb,amscd,mathrsfs,epic,wasysym,latexsym,tikz,mathrsfs,cite,hyperref}
\usepackage{stmaryrd}
\usepackage{pb-diagram}
\usepackage[matrix,arrow]{xy}

\makeatletter

\numberwithin{equation}{section}
\newtheorem{Proposition}[equation]{Proposition}
\newtheorem{Lemma}[equation]{Lemma}
\newtheorem{Theorem}[equation]{Theorem}
\newtheorem{Corollary}[equation]{Corollary}
\theoremstyle{definition}  %% makes all of the theorem environments which follow appear in \rm

\newtheorem{Remark}[equation]{Remark}

\newtheorem{MainTheorem}{Theorem}

\let\<\langle
\let\>\rangle

\newcommand\Comment[2][\relax]{\space\par\medskip\noindent%
   \fbox{\begin{minipage}{\textwidth}\textbf{Comment\ifx\relax#1\else---#1\fi}\newline%
        #2\end{minipage}}\medskip
}

\def\bl{\text{\boldmath$l$}}
\def\bp{\text{\boldmath$p$}}
\def\bq{\text{\boldmath$q$}}
\def\bs{\text{\boldmath$s$}}

\def\bt{\text{\boldmath$t$}}
\def\bc{\text{\boldmath$c$}}
\def\br{\text{\boldmath$r$}}
\def\b1{\text{\boldmath$1$}}

\def\ba{\text{\boldmath$a$}}
\def\bb{\text{\boldmath$b$}}

\def\bu{\text{\boldmath$u$}}
\def\bv{\text{\boldmath$v$}}
\def\bx{\text{\boldmath$x$}}
\def\by{\text{\boldmath$y$}}
\def\bbe{\text{\boldmath$e$}}

\def\bsi{\text{\boldmath$\sigma$}}
\def\btau{\text{\boldmath$\tau$}}

\def\fa{\mathfrak{a}}
\def\a{\mathfrak{a}}
\def\c{\mathfrak{c}}
\def\z{\mathfrak{z}}

\newcommand{\Hom}{\operatorname{Hom}}

\newcommand{\End}{\operatorname{End}}

\newcommand{\id}{\operatorname{id}}

\def\sgn{\mathtt{sgn}}
\newcommand{\res}{\operatorname{res}}

\newcommand{\infl}{\operatorname{infl}}

\newcommand{\Z}{\mathbb{Z}}
\newcommand{\K}{\mathbb{K}}
\newcommand{\F}{\mathbb{F}}

\newcommand{\0}{{\bar 0}}
\renewcommand{\1}{{\bar 1}}
\def\eps{{\varepsilon}}
\def\phi{{\varphi}}

\newcommand{\Triple}{{\mathcal T}}

\newcommand{\ga}{\gamma}

\newcommand{\la}{\lambda}
\newcommand{\La}{\Lambda}
\newcommand{\al}{\alpha}
\newcommand{\be}{\beta}

\def\Si{\mathfrak{S}}
\newcommand{\si}{\sigma}
\newcommand{\om}{\omega}
\newcommand{\Om}{\Omega}

\newcommand{\de}{\delta}

\newcommand{\ka}{\kappa}

\newcommand{\Conv}{\mathop{\scalebox{1.5}{\raisebox{-0.2ex}{$\ast$}}}}

\def\id{\mathop{\mathrm {id}}\nolimits}

\def\rank{\mathop{\mathrm{ rank}}\nolimits}

\newcommand{\tr}{{\mathsf {t}}}

\newcommand{\ZZ}{{\mathbb Z}}

\newcommand{\D}{{\mathscr D}}

\renewcommand{\mod}{\bmod \,}

\def\K{\mathbb K}

\newcommand{\EZig}{Z}

\def\ttb{{\mathtt b}}

\def\col{{\operatorname{col}}}

\def\b{\mathfrak{b}}
\def\k{\Bbbk}

\def\spa{\operatorname{span}}

\def\mod#1{#1\!\operatorname{-mod}}

\def\iso{\stackrel{\sim}{\longrightarrow}}

\def\ch{\operatorname{ch}}
\def\lan{\langle}
\def\ran{\rangle}

\def\Seq{\operatorname{Tri}}

\def\ttb{{\mathtt b}}

\def\TE{{\operatorname{E}}}

\newcommand{\Star}{\operatorname{\mathtt{Star}}}

\def\bla{\text{\boldmath$\lambda$}}
\def\bmu{\text{\boldmath$\mu$}}
\def\bnu{\text{\boldmath$\nu$}}

{\catcode`\|=\active
  \gdef\set#1{\mathinner{\lbrace\,{\mathcode`\|"8000%
  \let|\midvert #1}\,\rbrace}}
}
\def\midvert{\egroup\mid\bgroup}

\colorlet{darkgreen}{green!50!black}
\tikzset{dots/.style={very thick,loosely dotted},
         greendot/.style={fill,circle,color=darkgreen,inner sep=1.5pt,outer sep=0},
         blackdot/.style={fill,circle,color=black,inner sep=1.5pt,outer sep=0},
         graydot/.style={fill,circle,color=gray,inner sep=1.1pt,outer sep=0}
}
\def\greendot(#1,#2){\node[greendot] at(#1,#2){}}
\def\blackdot(#1,#2){\node[blackdot] at(#1,#2){}}
\def\graydot(#1,#2){\node[graydot] at(#1,#2){}}

\newenvironment{braid}{% sets defaults for the braid diagrams
  \begin{tikzpicture}[baseline=6mm,black,line width=1pt, scale=0.32,
                      draw/.append style={rounded corners},
                      every node/.append style={font=\fontsize{5}{5}\selectfont}]%
  }{\end{tikzpicture}
}

\def\Grid(#1,#2){%  draws a coordinate grid inside a braid diagram
  \draw[very thin,gray,step=2mm] (0,0)grid(#1,#2);
  \draw[very thin,darkgreen,step=10mm] (0,0)grid(#1,#2);
}

\newcommand\Tableau[2][\relax]{
  \begin{tikzpicture}[scale=0.5,draw/.append style={thick,black}]
    \ifx\relax#1\relax%
    \else % shade the boxes in #1
      \foreach\box in {#1} { \filldraw[blue!30]\box+(-.5,-.5)rectangle++(.5,.5); }
    \fi
    \newcount\row\newcount\col
    \row=0
    \foreach \Row in {#2} {
       \col=1
       \foreach\k in \Row {
          \draw(\the\col,\the\row)+(-.5,-.5)rectangle++(.5,.5);
          \draw(\the\col,\the\row)node{\k};
          \global\advance\col by 1
       }
       \global\advance\row by -1
    }
  \end{tikzpicture}
}

\newcommand\YoungDiagram[2][\relax]{
  \begin{tikzpicture}[scale=0.5,draw/.append style={thick,black}]
    \ifx\relax#1\relax%
    \else % shade the boxes in #1
    \foreach\box in {#1} {
      \filldraw[blue!30]\box rectangle ++(1,1);
    }
    \fi
    \newcount\row
    \row=0
    \foreach \col in {#2} {
       \draw(1,\the\row)grid ++(\col,1);
       \global\advance\row by -1
    }
  \end{tikzpicture}
}

\begin{document}

\title[Generalized Schur algebras]{{\bf Generalized Schur algebras}}

\author{\sc Alexander Kleshchev}
\address{Department of Mathematics\\ University of Oregon\\
Eugene\\ OR 97403, USA}
\email{klesh@uoregon.edu}

\author{\sc Robert Muth}
\address{Department of Mathematics\\ Washington \& Jefferson College
\\
Washington\\ PA 15301, USA}
\email{rmuth@washjeff.edu}


\thanks{The first author was supported by the NSF grant No. DMS-1700905 and the DFG Mercator program through the University of Stuttgart. This work was also supported by the NSF under grant No. DMS-1440140 while both authors were in residence at the MSRI during the Spring 2018 semester.}

\begin{abstract}
We define and study a new class of bialgebras, which generalize certain Turner double algebras related to generic blocks of symmetric groups. Bases and generators of these algebras are given. We investigate when the algebras are symmetric, which is relevant to block theory of finite groups. We then establish a double centralizer property related to blocks of Schur algebras.  
\end{abstract}

\maketitle

\section{Introduction}
Let $\k$ be a commutative domain of characteristic $0$ and 
$A$ be a unital $\k$-superalgebra, which is free as a $\k$-supermodule. Let $\a$ be a unital subalgebra of the even part $A_\0$, which is a direct summand of $A_\0$ as a $\k$-module. Some of the unitality conditions will be relaxed in the main body of the paper but in this introduction we will consider a special case.

We define and study {\em generalized Schur (super)algebras} $$T^A_\a(n,d)\subseteq S^A(n,d).$$ 
The algebra $S^A(n,d)$ is defined as the algebra of invariants $(M_n(A)^{\otimes d})^{\Si_d}$, and so in the case $A=\k$ we get that $S^\k(n,d)$ is the classical Schur algebra. If $\a=A_\0$, then $T^A_\a(n,d)= S^A(n,d)$, but in general the subalgebra $T^A_\a(n,d)\subseteq S^A(n,d)$ is proper, although it is always a full sublattice in $S^A(n,d)$. Thus extending scalars to a field $\K$ of characteristic $0$ produces the same algebras: $T^A_\a(n,d)_\K= S^A(n,d)_\K$. However, importantly, extending scalars to a field $\F$ of positive characteristic will in general yield {\em non-isomorphic} algebras $T^A_\a(n,d)_\F$ and $S^A(n,d)_\F$ of the same dimension. It turns out that in many situations it is the more subtly defined  algebra $T^A_\a(n,d)_\F$ that plays an important role. 

%While it is convenient to study the algebras $T^A_\a(n,d)$ and $S^A(n,d)$ together, we believe that it is the algebra $T^A_\a(n,d)$ that plays a more important role. For example, a
As a special case of our construction, we recover the {\em Turner double algebras} $D^A(n,d)$  studied in \cite{T, T2, T3,EK1}. In fact, we show in \S\ref{SSTE} that 
$$D^A(n,d)\cong T^{\TE(A)}_{A_\0}(n,d),$$ 
where $\TE(A)$ is the trivial extension algebra of $A$. Turner double algebras are important because of their connection to generic blocks of symmetric groups via Turner's conjecture, recently proved in \cite{EK2}. To be more precise, for an appropriate zigzag algebra  $\bar Z$ and a subalgebra $\bar \z\subseteq \bar Z$, the generalized Schur algebra $T^{\bar Z}_{\bar \z}(n,d)$ is Morita equivalent to weight $d$ RoCK blocks of  symmetric groups. In this way, $T^{\bar Z}_{\bar \z}(n,d)$ can be considered as a `local'  object replacing wreath products of Brauer tree algebras in the context of the Brou{\'e} abelian defect group conjecture for blocks of symmetric groups with {\em non-abelian}  defect groups. 
 
 However, it is known that Turner doubles cannot provide a similar `local' description for blocks of classical Schur algebras because the former are always symmetric algebras while the latter in general are not. We believe that our more general construction of $T^A_\a(n,d)$ fixes the problem. In \cite{KM}, we formulate an explicit conjecture for RoCK blocks of classical Schur algebras in terms of the generalized Schur algebras $T^Z_\z(n,d)$, where $Z$ is an  {\em extended}\, zigzag algebra. %, see Conjecture~\ref{Conj}. 

Furthermore, we will prove in \cite{KM} that, under reasonable  additional assumptions on $\a$, the algebra  $T^A_\a(n,d)$ is quasi-hereditary if $A$ is quasi-hereditary. 
This provides us with a method to produce new interesting quasi-hereditary algebras from old. In particular, the algebra $T^Z_\z(n,d)$ from the previous paragraph is quasi-hereditary, as should be expected if it is to be Morita equivalent to a block of the Schur algebra. 
 
We now describe the contents of the paper in more detail.  Section~\ref{SPrelim} is preliminary.
In Section~\ref{SSA}, given a basis \(B\) for \(A\) which extends a basis for \(\a\), we describe a natural basis for \(S^A(n,d)\) in terms of certain elements \(\xi^{\bb}_{\br, \bs}\), where \(\bb \in B^d\), \(\br, \bs \in [1,n]^d\). This is an analogue of Schur's basis of the classical Schur algebra. 
By rescaling this natural basis using certain products of factorals defined in Section~\ref{SPrelim}, we define the full sublattice \(T^A_\a(n,d) \subseteq S^A(n,d)\). Our first main result is:

\begin{MainTheorem}\label{T1}
{\em We have that $T^A_\a(n,d) \subseteq S^A(n,d)$ is a unital subalgebra. 
}
\end{MainTheorem}

There is another description of $T^A_\a(n,d)$ as a subalgebra of $S^A(n,d)$, which shows in particular that \(T^A_\a(n,d)\) is independent of the choice of basis \(B\) above:

\begin{MainTheorem}\label{T3}
{\em
We have that \(T^A_\a(n,d)\) is the subalgebra of $S^A(n,d)$ generated by \(S^\a(n,d)\) and the elements of the form
$$
\sum_{e=0}^{d-1}
1^{\otimes d-1-e} \otimes \xi\otimes 1^{\otimes e},
$$
where $\xi\in M_n(A)$ and $1:=1_{M_n(A)}$. 
}
\end{MainTheorem}

A slightly stronger result appears as Theorem~\ref{TGen}. In order to prove this result, we first investigate some coproducts and $*$-products. Recall that $\bigoplus_{d \geq 0}M_n(A)^{\otimes d}$ has a natural coproduct $\nabla$, see \S\ref{SSCoproduct}. We then prove 

\begin{MainTheorem}\label{T2}
{\em The coproduct \(\nabla\) restricts to coproducts on 
$$S^A(n):= \bigoplus_{d \geq 0}S^A(n,d)\quad \text{and} \quad T^A_\a(n):=\bigoplus_{d \geq 0}T^A_\a(n,d).
$$
}
\end{MainTheorem}

In Section~\ref{Ssuperbi}, we show that the \(*\)-product (or shuffle product) on $$\bigoplus_{d \geq 0}M_n(A)^{\otimes d}$$ restricts to a product on \(S^A(n)\) and \(T^A_\a(n)\), which, together with \(\nabla\), gives these objects a superbialgebra structure. 
We then prove that \(T^A_\a(n)\) is generated under the \(*\)-product by \(S^\a(n)\) and \(M_n(A)\). This allows us to prove Theorem 2.

In Section~\ref{SExamples} we first discuss some properties of idempotents and idempotent truncations in \(T^A_\a(n,d)\). Given an idempotent $e\in\a$, we define an idempotent $\xi^e\in T^A_\a(n,d)$ and prove in Lemma~\ref{LTruncation} that
$$
\xi^e T^A_\a(n,d)\xi^e\cong T^{eAe}_{e\a e}(n,d).
$$
 Section~\ref{SExamples} is completed with some important examples of generalized Schur algebras. We discuss how \(T^A_\a(n,d)\) generalizes the Turner double construction and look at the case where \(A\) is the extended zigzag algebra. %, and end with a conjecture that relates the generalized Schur algebra of \(A\) to RoCK blocks of the \(q\)-Schur algebra.
 
In Section~\ref{Ssym} we study the symmetricity of $T^A_\a(n,d)$.  This is important since blocks of finite groups are symmetric algebras and, inspired by \cite{EK2}, we hope that in some situations $T^A_\a(n,d)$ could provide a local description of some interesting blocks. As the example $A=\k$ shows, it is certainly not enough to assume that $A$ is symmetric to guarantee that so is $T^A_\a(n,d)$. A natural assumption we have to make is that the symmetrizing form $\tr$ is {\em $(A,\a)$-symmetrizing}, i.e. $(\a,\a)_\tr=0$ and the $\k$-complement $\c$ of $\a$ in $A_\0$ can be chosen so that the restriction of $(\cdot,\cdot)_\tr$ to $\a\times\c$ is a perfect pairing. Then we construct an explicit symmetrizing form $\tr^T$ on $T^A_\a(n,d)$ and prove in Corollary~\ref{Tsym}:

\begin{MainTheorem}\label{T3.5}
{\em 
If $\tr$ is an $(A,\a)$-symmetrizing form on $A$, then the algebra $T^A_\a(n,d)$ is symmetric, with symmetrizing form $\tr^T$.
}
\end{MainTheorem}

In Section~\ref{SDCP} we investigate double centralizer properties. Let $S$ be a $\k$-algebra and  $e\in S$ be an idempotent. We say that $e$ is a {\em double centralizer idempotent for $S$} if the natural map $
S \to \End_{eSe}(Se)
$ is an isomorphism. Given $e\in\a$, which is a double centralizer idempotent for $A$, it is not in general true that $\xi^e$ is a double centralizer idempotent for $T^A_\a(n,d)$, see Remark~\ref{RCounterExample}. However, in Theorem~\ref{Tk}, we prove the following positive result:

\begin{MainTheorem}\label{T4}
{\em 
Let $e \in A$ be a double centralizer idempotent for $A$ and $d\leq n$. Then $\xi^e$ is a double centralizer idempotent for $S^A(n,d)$. In particular, if $\K$ is the quotient field of $\k$, then $\xi^e$  is a double centralizer idempotent for $S^A(n,d)_\K=T^A_\a(n,d)_\K$. 
}
\end{MainTheorem}

Finally, in Theorem~\ref{TDCPZ}, we deal with the all-important zigzag case over the arbitrary $\k$:   

\begin{MainTheorem}\label{T5}
{\em 
Let $Z$ be the extended zigzag algebra with the standard idempotents $e_0,e_1,\dots, e_\ell$. We set $e:=e_0+\dots+e_{\ell-1}$, so that $eZe$ is the zigzag algebra.  Then $e$ is a double centralizer idempotent for $Z$, and $\xi^e$ is a double centralizer idempotent for $T^Z_\z(n,d)$ provided $d\leq n$. 
}
\end{MainTheorem}

\section{Preliminaries}\label{SPrelim}
Throughout the paper $\k$ is always a commutative domain of characteristic~$0$. 

\subsection{Superalgebras and supermodules}
Let $V$ be a {\em $\k$-supermodule}, i.e. $V$ is endowed with a $\k$-module decomposition  
$V=V_\0\oplus V_\1
$ 
(the superstructure could be trivial, i.e. we could have $V=V_\0$).  
If $\eps\in\Z/2$ and $v\in V_\eps$, we call $v$ {\em homogeneous} and write $\bar v:=\eps$. 
%We denote by $V_{\rm hom}$ the set of all non-zero homogeneous elements of $V$. 
For a set $S$ of homogeneous elements of $V$ and $\eps\in\Z/2$ 
%and $m\in\Z$, 
we denote 
\begin{equation}\label{EH}
S_\eps:=S\cap  V_\eps.
\end{equation}
A map $f:V\to W$ of $\k$-supermodules is called {\em homogeneous} if $f(V_\eps)\subseteq W_\eps$ for all $\eps$. 
%For a free $\k$-module $W$ of finite rank $d$, we write $d=\dim W$. 
A $\k$-supermodule $V$ is {\em free} 
%of finite rank 
if so is each $V_\eps$. Let $V$ be a free  $\k$-supermodule. 
%of finite rank. 
A {\em homogeneous basis} of $V$ is a $\k$-basis all of whose elements are homogeneous. 
%The {\em graded dimension} of $V$ is  $$\dim^q_\pi V:=\sum_{n\in\Z,\, \eps\in\Z/2}(\dim V^n_\eps)q^n\pi^\eps\in R.$$
A (not necessarily unital) $\k$-algebra $A$ is called a {\em $\k$-superalgebra}, if $A$ is a $\k$-supermodule and $A_{\eps}A_{\de}\subseteq A_{\eps+\de}$ for all $\eps,\de$.  
%We will only consider unital $\k$-algebras, which are free of finite rank as graded $\k$-supermodules. 

Throughout the paper we will work with a fixed superalgebra $A$ which is free %of finite rank 
as a $\k$-supermodule (not necessarily of finite rank). Moreover, we fix a $\k$-subalgebra $\a\subseteq A_\0$ such that \(\a\) and \(A/\a\) are both free as \(\k\)-modules.
 Such a pair $(A,\a)$ will be called a {\em good pair}. It is called a {\em unital good pair} if both $A$ and $\a$ are unital and $1_\a=1_A$. 

For our fixed good pair $(A,\a)$, we pick a $\k$-module complement $\c$ for $\a$ in $A_\0$ and \(\k\)-bases  $B_{\fa}$, $B_\c$, $B_\1$  for $\a$, $\c$, $A_\1$, respectively, so that  
\begin{align}\label{AaBasis}
B= B_{\fa}\sqcup B_{\c} \sqcup B_{\bar 1}
\end{align}
is a homogeneous basis for \(A\). We call such a basis an {\em \((A, \fa)\)-basis}.

Define the structure constants $\kappa^b_{a,c}$ of $A$ from
\begin{equation}\label{EStrConst}
ac=\sum_{b\in B}\kappa^b_{a,c} b
\qquad(a,c\in A).
\end{equation}
More generally, for 
$$\bb=(b_1,\dots, b_d)\in B^d\quad \text{and}\quad \ba=(a_1,\dots, a_d),\,\bc=(c_1,\dots, c_d)\in A^d,$$ 
we define 
\begin{equation}\label{ESCB}
\ka^\bb_{\ba,\bc}:=\ka^{b_1}_{a_1,c_1}\dots \ka^{b_d}_{a_d,c_d}.
\end{equation}
Finally, we denote by $H$ the set of all non-zero homogeneous elements of $A$. 

The matrix algebra $M_n(A)$ is naturally a superalgebra. For $1\leq r,s\leq n$ and $a\in A$, we denote
\begin{equation}\label{EXirs}
\xi_{r,s}^a:=a E_{r,s}\in M_n(A). 
\end{equation}
Then 
\begin{equation}\label{EBasisM_n(A)}
\{\xi_{r,s}^b\mid 1\leq r,s\leq n,\ b\in B\}
\end{equation}
 is a homogeneous basis of $M_n(A)$, and by (\ref{EStrConst}) we have 
\begin{equation}\label{EXiProduct}
\xi_{r,s}^{a}\xi_{t,u}^{c}=\de_{s,t}\sum_{b\in B}\kappa^b_{a,c}\xi_{r,u}^{b}
\qquad(a,c\in A,\ 1\leq r,s,t,u\leq n).
\end{equation}

\subsection{Combinatorics}\label{SComb}
For $r,s\in\Z$ we denote $[r,s]:=\{t\in\Z\mid r\leq t\leq s\}$. 
We fix $n\in\Z_{>0}$ and $d\in\Z_{\geq 0}$. 
For a set $X$, the elements of $X^d$ are referred to as {\em words} (of length $d$) with letters in the alphabet $X$. 
The words are usually written as 
$
x_1x_2\cdots x_d\in X^d.
$
For $\bx\in X^d$ and $\bx'\in X^{d'}$ we denote by $\bx\bx'\in X^{d+d'}$ the concatenation of $\bx$ and $\bx'$. For 
$x\in X$, we denote $x^d:=x\cdots x\in X^d$.

The symmetric group $\Si_d$ acts on the right on $X^d$ by place permutations:
$$
(x_1\cdots x_d)\si=x_{\si 1}\cdots x_{\si d}.
$$
For $\bx,\bx'\in X^d$, we write $\bx\sim\bx'$ if $\bx\si=\bx'$ for some $\si\in \Si_d$. If $X_1,\dots,X_N$ are sets, then $\Si_d$ acts on $X_1^d\times\dots\times X_N^d$ diagonally:
$$
(\bx^1,\dots,\bx^N)\si=(\bx^1\si,\dots,\bx^N\si).
%\qquad(\bx^1\in X_1^d,\dots,\bx^N\in X_N^d,\ \si\in\Si_d). 
$$
The set of the corresponding orbits is denoted 
$
(X_1^d\times\dots\times X_N^d)/\Si_d,
$ 
and the orbit of $(\bx^1,\dots,\bx^N)$ is denoted $[\bx^1,\dots,\bx^N]$. 
We write 
$$(\bx^1,\dots,\bx^N)\sim (\by^1,\dots,\by^N)$$ 
if $[\bx^1,\dots,\bx^N]= [\by^1,\dots,\by^N]$.  

Let $P$ be a set of homogeneous elements of $A$. Our main examples will be $P=B$ and $P= H$ (the set of all non-zero homogeneous elements of $A$). We have $P=P_\0\sqcup P_\1$. Define $\Seq^P (n,d)$ to be the set of all triples 
$$
(\bp,\br, \bs) = ( p_1\cdots p_d,\, r_1\cdots r_d,\, s_1\cdots s_d ) \in  P^d\times[1,n]^d\times[1,n]^d
$$
such that for any $1\leq k\neq l\leq d$ we have 
$(p_k,r_k,s_k)=(p_l,r_l, s_l)$ 
only if $p_k\in P_\0$. Then $\Seq^P (n,d)\subseteq P^d\times[1,n]^d\times[1,n]^d$ is $\Si_d$-invariant and so we have the orbit set $\Seq^P (n,d)/\Si_d$. 

For $(\bp,\br,\bs)\in\Seq^P(n,d)$, we consider the stabilizer 
$$\Si_{\bp,\br,\bs}:=\{\si\in \Si_d\mid (\bp,\br,\bs)\si=(\bp,\br,\bs)\},
$$
and denote by\, ${}^{\bp,\br,\bs}\D$ a set of the shortest coset representatives for $\Si_{\bp,\br,\bs}\backslash\Si_n$. Then $
\{(\bp,\br,\bs)\si\mid \si\in {}^{\bp,\br,\bs}\D\} 
$ 
is the set of distinct elements in the orbit $[\bp,\br,\bs]$.

We fix a total order `$<$' on 
%$P$. Then we have a lexicographic order on 
$P\times[1,n]\times[1,n]$. 
%: $(p,r,s)< (p',r', s')$ if and only if $p<p'$, or $p=p'$ and $r<r'$, or $p=p',r=r'$ and $s<s'$. 
Then we also have a total order on $\Seq^P(n,d)$ defined as follows: $(\bp,\br,  \bs)< (\bp',\br',  \bs')$ if and only if there exists $l\in[1,d]$ such that $(p_k,r_k,s_k)=(p_k',r_k',s_k')$ for all $k<l$ and $(p_l,r_l,s_l)<(p_l',r_l',s_l')$. Denote
\begin{equation}\label{ESeq0}
\Seq^P_0(n,d)=\{(\bp, \br, \bs) \in \Seq^P(n,d)\mid (\bp, \br, \bs)\leq (\bp, \br, \bs) \sigma\ \text{for all}\ \sigma \in \mathfrak{S}_d\}.
\end{equation}
We have a bijection
$$
\Seq^P_0(n,d)\iso \Seq^P(n,d)/\Si_d,\ (\bp, \br, \bs)\mapsto [\bp, \br, \bs].
$$

For $(\bp,\br, \bs) \in \Seq^P(n,d)$, $\bp' \in P^d$ and $\si\in\Si_d$, 
we define
\begin{align*}
\lan\bp, \br,  \bs\ran
&:=\sharp\{(k,l)\in[1,d]^2\mid k<l,\ p_k,p_l\in P_\1,\ (p_k,r_k,s_k)> (p_l,r_l, s_l)\},
\\
\lan \bp, \bp'\ran
&:=\sharp\{(k,l)\in[1,d]^2\mid k>l,\  p_k,p_l'\in P_\1\}.
\\
\lan \si;\bp\ran&:=\sharp\{(k,l)\in[1,d]^2\mid k<l,\  \si^{-1}k>\si^{-1}l,\ p_k,p_l\in P_\1\}.
\end{align*}
Note that
\begin{equation}\label{E171116}
(-1)^{\lan\bp,\br,  \bs\ran+\lan\bp\si,\br\si,  \bs\si\ran}=(-1)^{\lan \si;\bp\ran}.
\end{equation}

Let us now specialize to the case $P=B$.

\begin{Lemma} \label{Lsigneq} %{\rm \cite{}}%{\bf ()}
Let \((\ba, \br, \bt), (\bc, \bt, \bu) \in \Seq^B(n,d)\). Assume that, for some \(1 \leq k < d\), either \(\bar a_k = \bar c_k\) or \(\bar a_{k+1} = \bar c_{k+1}\). Then 
 \begin{align*}
 (-1)^{\langle \ba, \br, \bt \rangle + \langle \bc, \bt, \bu \rangle + \langle \ba, \bc \rangle}
 =
  (-1)^{\langle \ba s_k, \br s_k, \bt s_k \rangle + \langle \bc s_k, \bt s_k, \bu s_k \rangle + \langle \ba s_k, \bc s_k \rangle}. %\label{signeq}
 \end{align*}
\end{Lemma}
\begin{proof}
We consider three cases:
 
 {\em Case 1: at least two of \(a_k, c_k, a_{k+1}, c_{k+1}\) are even.} In this case  \(s_k\) does not exchange the positions of two odd elements in \(\ba\) or \(\bc\), so 
$\langle \ba, \br, \bt \rangle =\langle \ba s_k, \br s_k, \bt s_k \rangle
$ 
and
$
\langle \bc, \bt, \bu \rangle =\langle \bc s_k, \bt s_k, \bu s_k \rangle. 
$ 
We also note that \(a_{k+1}\) and \(c_k\) cannot both be odd, and \(a_k\) and \(c_{k+1}\) cannot both be odd, so 
$\langle \ba, \bc \rangle =\langle \ba s_k, \bc s_k \rangle.$ 

 {\em Case 2: Exactly one of \(a_k, c_k, a_{k+1}, c_{k+1}\) is even.} By symmetry we may assume that \(a_k, a_{k+1}\) are odd and 
 %since \((\ba, \br, \bt) \in \Seq(n,d)\), it cannot be the case \(a_k = a_{k+1}, r_k = r_{k+1}, t_k = t_{k+1}\). Thus 
one of \(c_k\), \(c_{k+1}\) is even. Then   we have 
\begin{align*}
 (-1)^{\langle \ba, \br, \bt \rangle} &= -(-1)^{\langle \ba s_k, \br s_k, \bt s_k \rangle},
\\
(-1)^{\langle \bc, \bt, \bu \rangle}
& =
  (-1)^{\langle \bc s_k, \bt s_k, \bu s_k \rangle},
\\
(-1)^{\langle \ba, \bc \rangle} &=-(-1)^{\langle \ba s_k, \bc s_k \rangle}. 
\end{align*}

{\em Case 3: \(a_k, c_k, a_{k+1}, c_{k+1}\) are all odd.} Then  we have 
\begin{align*}
(-1)^{\langle \ba, \br, \bt \rangle} &= -(-1)^{\langle \ba s_k, \br s_k, \bt s_k \rangle}, 
\\
   (-1)^{\langle \bc, \bt, \bu \rangle}
 &=
  -(-1)^{\langle \bc s_k, \bt s_k, \bu s_k \rangle},
\\
(-1)^{\langle \ba, \bc \rangle} &=(-1)^{\langle \ba s_k, \bc s_k \rangle}.
\end{align*} 
\end{proof}

Let $(\bb,\br, \bs)\in\Seq^B(n,d)$. 
For $b\in B$ and $r,s\in [1,n]$, we denote 
\begin{equation}\label{E190617}
[\bb,\br, \bs]^b_{r,s}:=\sharp\{k\in[1,d]\mid  (b_k,r_k,s_k)=(b,r,s)\},
\end{equation}
and define
\begin{align}
[\bb,\br, \bs]^! &:=\prod_{b\in B,\, r,s\in [1,n]} [\bb,\br, \bs]^b_{r,s}!=\prod_{b\in B_\0,\, r,s\in [1,n]} [\bb,\br, \bs]^b_{r,s}!
\end{align}
(if $B$ is infinite, these are infinite products but all but finitely many factors are $1$). 
Note that
\begin{equation}
\label{subalg1}
|\Si_{\bb, \br, \bs}| = [\bb, \br, \bs]^!. 
\end{equation}
Moreover, we define
\begin{align}
[\bb,\br, \bs]^!_{\a} &:=\prod_{ b\in B_\a,\ r,s\in [1,n]}[\bb,\br, \bs]^b_{r,s}!,\\
\label{subalg2}
[\bb,\br, \bs]^!_{\c} &:=\prod_{ b\in B_\c,\ r,s\in [1,n]}[\bb,\br, \bs]^b_{r,s}!.
\end{align}

\section{Generalized Schur algebras}\label{SSA}
Throughout the section, $(A,\a)$ is a fixed good pair with an $(A,\a)$-basis $B= B_{\fa}\sqcup B_{\c} \sqcup B_{\bar 1}$ as in (\ref{AaBasis}). Recall that $H$ denotes the set of all non-zero homogeneous elements of $A$. We also fix $n\in\Z_{>0}$ and $d\in\Z_{\geq 0}$.

In this section, we will construct generalized Schur algebras $T^A_\a(n,d)\subseteq S^A(n,d)$. The definition of the algebra $S^A(n,d)$ is straightforward, while $T^A_\fa(n,d)$ is obtained by making a subtle choice of a full-rank sublattice in $S^A(n,d)$ which depends on $\a$. If $\a=A_\0$, then $T^A_\fa(n,d)= S^A(n,d)$, but in general the algebras are different. In \S\ref{SSCoproduct}, we 
investigate a natural coproduct on $S^A(n):=\bigoplus_{d\in\Z_{\geq 0}}S^A(n,d)$ and show that $$T^A_\a(n):=\bigoplus_{d\in\Z_{\geq 0}}T^A_\a(n,d)\subseteq S^A(n)$$ is a subcoalgebra.

%Crucially, if \(A\) has an \(\fa\)-conforming heredity data, then $T^A_\fa(n,d)$ turns out to be quasi-hereditary, while \(S^A(n,d)\) does not in general, see \cite{greenThree}.

\subsection{The algebra\, $S^A(n,d)$}
%Its definition depends on the superalgebra structures on $A$. We now recap some facts. 
Let $M_n(A)$ be the $\k$-(super)algebra of $n\times n$ matrices with entries in $A$ and recall the notation (\ref{EXirs}). 
%The symmetric group $\Si_d$ acts on the superalgebra $M_n(A)^{\otimes d}$ by superpermutations  of tensors. 
%To be more precise, 
%let $P=A_{\rm hom}$. T
There is a right action of $\Si_d$ on $M_n(A)^{\otimes d}$ with (super)algebra automorphisms, such that for all $a_1,\dots,a_d\in H$, $r_1,s_1,\dots,r_d,s_d\in [1,n]$ and $\si\in \Si_d$, we have 
$$(\xi_{r_1,s_1}^{a_1}\otimes\dots\otimes \xi_{r_d,s_d}^{a_d})^\si=
(-1)^{\lan\si;\ba\ran} \xi_{r_{\si1},s_{\si1}}^{a_{\si1}}\otimes\dots\otimes \xi_{r_{\si d},s_{\si d}}^{a_{\si d}}.
$$
The algebra $S^A(n,d)$ is defined as the corresponding algebra of invariants 
$$
S^A(n,d):=(M_n(A)^{\otimes d})^{\Si_d}.
$$
Note that $S^A(n,d)$ is unital if and only if so is $A$. 
%where the symmetric group $\Si_d$ acts on the superalgebra $M_n(A)^{\otimes d}$ by superpermutations  of tensors. 

For $(\ba,\br,\bs)\in\Seq^H(n,d)$, we define elements 
\begin{equation}\label{EXiDef}
\begin{split}
\xi_{\br,\bs}^\ba&:= \sum_{\si\in{}^{\ba,\br,\bs}\D} 
(\xi_{r_1,s_1}^{a_1} \otimes \cdots \otimes \xi_{r_d,s_d}^{a_d})^\si
\\
&= \sum_{(\ba',\br',\bs')\sim (\ba,\br,\bs)} 
(-1)^{\lan\ba,\br,  \bs\ran+\lan\ba',\br',  \bs'\ran}\,
\xi_{r_1',s_1'}^{a_1'} \otimes \cdots \otimes \xi_{r_d',s_d'}^{a_d'}.
\end{split}
\end{equation}
in $S^A(n,d)$, 
where we have used  (\ref{E171116}) to obtain the last equality. 
The following is clear (as noted in \cite[Lemma 6.10]{EK1}):

\begin{Lemma}   \label{LBasis'} 
We have that 
$
\{\xi_{\br,\bs}^\bb\mid [\bb,\br,\bs]\in\Seq^B(n,d)/\Si_d\}
$ is a basis of $S^A(n,d)$. 
\end{Lemma}

\begin{Lemma} \label{LXiZero} %{\rm \cite{}}%{\bf ()}
If $(\ba',\br',\bs') \sim (\ba,\br,\bs)$ are elements of  $\Seq^H(n,d)$, then $$\xi_{\br',\bs'}^{\ba'}=(-1)^{\lan\ba,\br,  \bs\ran+\lan\ba',\br',  \bs'\ran}\xi_{\br,\bs}^\ba. 
$$
\end{Lemma}
\begin{proof}
This follows from (\ref{EXiDef}).
\end{proof}

For $(\ba,\bp,  \bq),\, (\bc,\bu,  \bv) \in \Seq^H(n,d)$ and $(\bb,\br,  \bs) \in \Seq^B(n,d)$, define the structure constants $f_{\ba,\bp, \bq;\bc,  \bu,\bv}^{\bb,\br,\bs}$ from
\begin{equation}\label{EFBABBBCNew}
 \xi^\ba_{\bp,  \bq}\, \xi^\bc_{\bu,  \bv}=\sum_{[\bb,\br,\bs]\in\Seq^B(n,d)/\Si_d} f_{\ba,\bp, \bq;\bc,  \bu,\bv}^{\bb,\br,\bs}\,  \xi^\bb_{\br,  \bs}.
%\qquad(\bC,\bD\in\Mat^\ttB(n,d)). 
\end{equation}
Note by Lemma~\ref{LXiZero} that if $(\bb',\br',\bs') \sim (\bb,\br,\bs)$ then 
$$
f_{\ba,\bp, \bq;\bc,  \bu,\bv}^{\bb,\br,\bs}=(-1)^{\lan\bb,\br,  \bs\ran+\lan\bb',\br',  \bs'\ran}f_{\ba,\bp, \bq;\bc,  \bu,\bv}^{\bb',\br',\bs'}.
$$

Recalling the notation (\ref{ESCB}), the following generalization of Green's product rule \cite[(2.3b)]{Green} follows from  \cite[(6.14)]{EK1}. 

\begin{Proposition} \label{CPR} %{\rm \cite{}}%{\bf ()}
Let $(\ba,\bp,  \bq),\, (\bc,\bu,  \bv) \in \Seq^H(n,d)$ and $(\bb,\br,  \bs) \in \Seq^B(n,d)$. Then 
$$
f_{\ba,\bp, \bq;\bc,  \bu,\bv}^{\bb,\br,\bs}= \sum_{\ba', \bc',\bt} 
(-1)^{\lan\ba,\bp,  \bq\ran 
+\lan\bc,\bu,  \bv\ran
+ \lan\ba',\br,  \bt\ran
+\lan\bc',\bt, \bs\ran+\lan\ba',\bc'\ran} \, 
\kappa_{\ba',\bc'}^{\bb},
$$
where the sum is over all 
$\ba', \bc'\in H^d$ and $\bt\in[1,n]$  
such that $(\ba',\br,  \bt) \sim (\ba,\bp,  \bq)$ and $(\bc',\bt,  \bs)\sim (\bc,\bu,\bv)$. 
\end{Proposition}

We can collect some of the equal terms in the formula above to rewrite it in the following form:

\begin{Corollary} \label{CSigns} %{\rm \cite{}}%{\bf ()}
Let $(\ba,\bp,  \bq),\, (\bc,\bu,  \bv) \in \Seq^H(n,d)$, $(\bb,\br,  \bs) \in \Seq^B(n,d)$, and let $X$ be the set of all \((\ba',\bc',\bt) \in H^d \times H^d \times [1,n]\) such that $(\ba',\br,  \bt) \sim (\ba,\bp,  \bq)$, $(\bc',\bt,  \bs)\sim (\bc,\bu,\bv)$, and \(\bar a_k' + \bar c_k' = \bar b_k\) for all \(k \in [1,d]\). We have:
\begin{enumerate}
\item[{\rm (i)}] If \((\ba', \bc', \bt) \in X\) then $(\ba', \bc', \bt)\si\in X$ for any $\si\in \Si_{\bb, \br, \bs}$. Let $\llbracket\ba', \bc', \bt\rrbracket:=\{(\ba', \bc', \bt)\si\mid \si\in \Si_{\bb, \br, \bs}\}\subseteq X$ denote the corresponding $\Si_{\bb, \br, \bs}$-orbit.
\item[{\rm (ii)}] $\kappa^{\bb}_{\ba', \bc'}$ and the parity of $\lan\ba',\br,  \bt\ran
+\lan\bc',\bt, \bs\ran+\lan\ba',\bc'\ran$ depend only on the orbit $\llbracket\ba', \bc', \bt\rrbracket$.
\item[{\rm (iii)}] The structure constant $f_{\ba,\bp, \bq;\bc,  \bu,\bv}^{\bb,\br,\bs}$ equals 
$$
\sum_{\llbracket \ba', \bc',\bt\rrbracket \in X/\mathfrak{S}_{\bb, \br, \bs}} 
\hspace{-7mm}
(-1)^{\lan\ba,\bp,  \bq\ran 
+\lan\bc,\bu,  \bv\ran
+ \lan\ba',\br,  \bt\ran
+\lan\bc',\bt, \bs\ran+\lan\ba',\bc'\ran}\,
[\mathfrak{S}_{\bb, \br, \bs} : \mathfrak{S}_{\bb, \br, \bs} \cap \mathfrak{S}_{\ba', \bc', \bt}]
 \, 
\kappa^{\bb}_{\ba', \bc'}.
$$
\end{enumerate}
\end{Corollary}
\begin{proof}
Let $\si\in \Si_{\bb, \br, \bs}$ and  \((\ba', \bc', \bt) \in X\). 

(i) To show that $(\ba'\si, \bc'\si,\bt\si) \in X$, note that 
$$(\ba'\si,\br,\bt\si)=(\ba'\si,\br\si,\bt\si)\sim(\ba,\bp,\bq)$$ and similarly $(\bc'\si,\bt\si,\bs)\sim(\bc,\bu,\bv)$. Finally, we have $\bar a'_{\si k}+\bar c'_{\si k}
=\bar b_{\si k}=\bar b_k$ for all $k$. 

(ii) We have $\kappa^{\bb}_{\ba'\si, \bc'\si}=\kappa^{\bb\si}_{\ba'\si, \bc'\si}=\kappa^{\bb}_{\ba', \bc'}$, giving the first statement of (ii). To complete the proof of (ii), we now show that 
  \begin{align*}
 (-1)^{\langle \ba', \br, \bt\rangle + \langle \bc', \bt, \bs \rangle + \langle \ba', \bc' \rangle}
 =
 (-1)^{\langle \ba'\si, \br, \bt\si \rangle + \langle \bc'\si, \bt\si, \bs \rangle + \langle \ba'\si, \bc'\si \rangle} 
 %\label{signsame}
 \end{align*}

Write \(\sigma\) as a reduced product of simple transpositions \(\sigma = s_{l_1} \cdots s_{l_m}\) (it is not in general true that \(s_{l_1}, \ldots, s_{l_m} \in \mathfrak{S}_{\bb, \br, \bs}\)).
Since \((\bb, \br, \bs) \in \Seq(n,d)\), we have \(\sigma k= k\) for all \(k\) such that \(b_k\) is odd. Therefore for all \(1 \leq j \leq m\),  at least one of \((\bb s_{l_1} \cdots s_{l_{j-1}})_{l_j}, (\bb s_{l_1} \cdots s_{l_{j-1}})_{l_j+1}\) is even---i.e., no two odd elements are ever exchanged by the simple transpositions that comprise \(\sigma\).
 
For \(1 \leq j \leq m\), either \((\ba's_{l_1} \cdots s_{l_{j-1}})_{l_j}\) and \((\bc's_{l_1} \cdots s_{l_{j-1}})_{l_j}\) are of the same parity, or  \((\ba's_{l_1} \cdots s_{l_{j-1}})_{l_j+1}\) and \((\bc's_{l_1} \cdots s_{l_{j-1}})_{l_j+1}\) are of the same parity, by the above paragraph and the fact that \(\overline{a}'_k + \overline{c}'_k = \overline{b}_k\) for all \(k\). Therefore we may repeatedly apply Lemma~\ref{Lsigneq} to get:
 \begin{align*}
 &(-1)^{\langle \ba', \br, \bt \rangle + \langle \bc', \bt, \bs \rangle + \langle \ba', \bc' \rangle}
 \\
 =\,
 & (-1)^{\langle \ba's_{l_1}, \br s_{l_1}, \bt s_{l_1} \rangle + \langle \bc' s_{l_1}, \bt s_{l_1}, \bs s_{l_1} \rangle + \langle \ba' s_{l_1}, \bc' s_{l_1} \rangle}\\
  =\,&
  (-1)^{\langle \ba's_{l_1}s_{l_2}, \br s_{l_1}s_{l_2}, \bt s_{l_1}s_{l_2} \rangle + \langle \bc' s_{l_1}s_{l_2}, \bt s_{l_1}s_{l_2}, \bs s_{l_1}s_{l_2} \rangle + \langle \ba' s_{l_1}s_{l_2}, \bc' s_{l_1}s_{l_2} \rangle}\\
  =\,& \cdots \\
  =\,& (-1)^{\langle \ba' \sigma, \br \sigma, \bt \sigma \rangle + \langle \bc' \sigma, \bt \sigma, \bs \sigma \rangle + \langle \ba' \sigma, \bc' \sigma \rangle}\\
   =\,& (-1)^{\langle \ba' \sigma, \br, \bt \sigma \rangle + \langle \bc' \sigma, \bt \sigma, \bs  \rangle + \langle \ba' \sigma, \bc' \sigma \rangle},
 \end{align*}
  completing the proof of (ii).
 
 (iii) As $\kappa^{b_k}_{a_k',c_k'}=0$ unless $\bar a_k' + \bar c_k' = \bar b_k$, we may assume that the summation in Proposition~\ref{CPR} is over all $(\ba',\bc',\bt)\in X$. 
 By (i), (ii) and Proposition~\ref{CPR}, we have
 \begin{align*}
f_{\ba,\bp, \bq;\bc,  \bu,\bv}^{\bb,\br,\bs}= 
\sum_{\llbracket\ba', \bc',\bt\rrbracket \in X/\mathfrak{S}_{\bb, \br, \bs}} 
\hspace{-8mm}
\#\llbracket \ba', \bc', \bt \rrbracket\,(-1)^{\lan\ba,\bp,  \bq\ran 
+\lan\bc,\bu,  \bv\ran
+ \lan\ba',\br,  \bt\ran
+\lan\bc',\bt, \bs\ran+\lan\ba',\bc'\ran} \, 
\kappa^{\bb}_{\ba', \bc'}.
 \end{align*}
It remains to note that
$
\#\llbracket\ba', \bc', \bt\rrbracket = |\mathfrak{S}_{\bb, \br, \bs} / \mathfrak{S}_{\bb, \br, \bs} \cap \mathfrak{S}_{\ba', \bc', \bt}|
$. 
\end{proof}

Let $\tau$ be a homogeneous anti-involution on $A$.   
Then $\tau$ induces a homogeneous anti-involution 
$$\tau_n:M_n(A)\to M_n(A),\ \xi^a_{r,s}\mapsto \xi^{\tau(a)}_{s,r},
$$
which in turn induces an anti-involution
\begin{equation}\label{TauND}
\tau_{n,d}:S^A(n,d)\to S^A(n,d),\ \xi^{\ba}_{\br,\bs}\mapsto \xi^{\ba^\tau}_{\bs,\br},
\end{equation}
%for all admissible $(\ba,\br,\bs)$, and 
where for $\ba=a_1\cdots a_d\in H^d$, we have denoted
$
\ba^\tau:=\tau(a_1)\cdots \tau(a_d)\in H^d.
$

\subsection{The algebra $T^A_{\fa}(n,d)$} Recalling the notation (\ref{subalg2}), for $(\bb,\br,  \bs) \in \Seq^B(n,d)$, we set 
\begin{equation}\label{E080717}
\eta^\bb_{\br,\bs}:=[\bb,\br,  \bs]^!_{\c}\, \xi^\bb_{\br,  \bs}. 
%= \frac{[\bb, \br, \bs]^!}{[\bb, \br, \bs]^!_{\a}}\xi_{\br, \bs}^{\bb}.
\end{equation}
Define the \(\k\)-submodule \(T^A_\fa(n,d) \subseteq {} S^A(n,d)\) to be
$$
T^A_\fa(n,d):=\spa\big(\,\eta^\bb_{\br,\bs}\mid (\bb,\br,\bs)\in\Seq^B(n,d)\,\big).
$$ 
It will turn out that $T^A_\fa(n,d)$ depends only on $\a$ but not on 
$\c$ or $B$, see Proposition~\ref{AaInd}.

\begin{Lemma} \label{LBasis} %{\rm \cite{}}%{\bf ()}
We have that $\big\{\,\eta^\bb_{\br,\bs}\mid (\bb,\br,\bs)\in\Seq^B(n,d)/\Si_d\,\big\}$ is a basis of $T^A_\fa(n,d)$. 
\end{Lemma}
\begin{proof}
Follows from the definition and Lemma~\ref{LBasis'}. 
\end{proof}

\begin{Lemma} \label{L140117} %{\rm \cite{}}%{\bf ()}
Let $a_1,\dots,a_d\in \fa\cup A_\1$ and $\br,\bs\in[1,n]^d$. Then $\xi^\ba_{\br,\bs}\in T^A_\fa(n,d)$. 
\end{Lemma}
\begin{proof}
By assumption, for $1\leq l\leq d$, either $a_l=\sum_{b\in \fa}c_{l,b}b$ or $a_l=\sum_{b\in B_{\1}}c_{l,b}b$, with $c_{l,b}\in\k$. It follows that $\xi^\ba_{\br,\bs}$ is a linear combination of the elements $\xi^\bb_{\br,\bs}$ such that  $\bb$ is of the form $b_1\cdots b_d$ with $b_l\in B_{\fa}\cup B_\1$ for all $l=1,\dots,d$. But for such $\bb$, we  have $\xi^\bb_{\br,\bs}=\eta^\bb_{\br,\bs}\in T^A_\fa(n,d)$. 
\end{proof}

\begin{Proposition} \label{PSubalgebra} %{\rm \cite{}}%{\bf ()}
We have that $T^A_\fa(n,d)\subseteq S^A(n,d)$ is a $\k$-subalgebra. It is a unital subalgebra if \((A,\a)\) is a unital good pair.
\end{Proposition}
\begin{proof}
By Lemma~\ref{L140117}, if \(1_A \in \fa\), then the identity $1_A\otimes\dots\otimes 1_A\in S^A(n,d)$ belongs to $T^A_\fa(n,d)$, so we only have to prove the first statement of the lemma. 

We now fix \((\ba,\bp,\bq),(\bc,\bu,\bv),(\bb, \br, \bs) \in \Seq^B(n,d)\) and apply Corollary~\ref{CSigns}. Using the notation as in the corollary, assume that \((\ba',\bc',\bt) \in X\) is such that \(\kappa^{\bb}_{\ba', \bc'} \neq 0\). In view of Corollary~\ref{CSigns}(iii), 
it suffices to prove that the integer 
\begin{align*}
M:=[\ba, \bp, \bq]^!_{\c} \cdot [\bc, \bu, \bv]^!_{\c} \cdot |\mathfrak{S}_{\bb, \br, \bs} / \mathfrak{S}_{\bb, \br, \bs} \cap \mathfrak{S}_{\ba', \bc', \bt}|
\end{align*}
is divisible by \([\bb, \br, \bs]^!_{\c}\).
For $b,a',c'\in B$ and $r,s,t\in [1,n]$, define 
\begin{align*}
m^{a', b,c'}_{r, s,t} := \#\{k \in [1,d] \mid a'_k = a', b_k = b, c'_k = c',  r_k = r, s_k=s, t_k = t\}.
\end{align*}
Then, using that $(\ba',\br,  \bt) \sim (\ba,\bp,  \bq)$, $(\bc',\bt,  \bs)\sim (\bc,\bu,\bv)$, we obtain: 
\begin{align}
|\Si_{\bb, \br, \bs} \cap \Si_{\ba', \bc', \bt}| &= \displaystyle \prod_{ a', b,c' \in B,\,\, r,s,t \in [1,n]} m^{a',b, c'}_{r, s,t}! \label{subalg3}
\\
[\bb, \br, \bs]^b_{r,s} &= \displaystyle \sum_{a', c' \in B, \,\, t \in [1,n]} m_{r, s,t}^{a', b, c'}
\qquad(b\in B,\ r,s\in[1,n]), 
\label{subalg4a}
\\
[\ba, \bp, \bq]^a_{p, q} &= \displaystyle \sum_{ b, c' \in B,\,\, t \in [1,n]} m^{a, b, c'}_{p, t,q}
\qquad(a\in B,\ p,q\in[1,n]), 
\label{subalg4b}
\\ 
[\bc, \bu, \bv]^c_{u, v} &= \displaystyle \sum_{a',b \in B,\,\, t \in [1,n]} m^{a', b, c}_{t, v,u},
\qquad(c\in B,\ u,v\in[1,n]),  \label{subalg4c}
\end{align}

By (\ref{subalg4a}), for every \(b \in B\) and \(r,s \in [1,n]\), we have that 
$$
y^b_{r,s}:=\frac{[\bb, \br, \bs]^b_{r, s}!}{\prod_{a', c' \in B,\,\, t \in [1,n]     } m^{a',b,c'}_{r, s,t}!}\in\Z.
$$
So 
$$
C := \prod_{b \in B_\fa \cup B_{\1}, \,\, r,s \in [1,n]}
y^b_{r,s}
 \in \Z,
$$
and by  (\ref{subalg3}), we have 
\begin{align*}
|\Si_{\bb, \br, \bs} / \Si_{\bb, \br, \bs} \cap \Si_{\ba', \bc', \bt}|\,=\,
\prod_{b \in B,\,\, r,s \in [1,n] }
y^b_{r,s}
\,=\,
C\cdot
\prod_{b \in B_{\c},\,\, r,s \in [1,n]}
y^b_{r,s}.
\end{align*}

Now we claim that \(b \in B_{\c}\) and \(m^{a',b,c'}_{r, s,t} >1\) imply \(a' \in B_{\c}\) or \(c' \in B_{\c}\). Indeed, if \(a' \in B_{\1}\), then using (\ref{subalg4b}), we get \(m^{a',b,c'}_{r, s,t}\leq [\ba, \bp, \bq]^{a'}_{r, t} \leq 1$ as \((\ba, \bp, \bq) \in  \Seq^B(n,d)\),  which is a contradiction. Thus \(a' \in B_{\0}\). Similarly, \(c' \in B_{\0}\). If \(a', c' \in B_{\fa}\), then \(\kappa^{b}_{a', c'} = 0\) since \(\fa\) is closed under multiplication. Since \(m^{a',b,c'}_{r, s,t} >0\), this implies \(\kappa^{\bb}_{\ba', \bc'} = 0\), 
which contradicts our choice of \((\ba', \bc', \bt)\), proving the claim. 

By the claim, for $b\in B_{\c}$ and $r,s\in[1,n]$, we may write
\begin{align*}
y^b_{r,s}=
\frac{
[\bb, \br, \bs]^b_{r,s}!
}{
\Big(\prod_{a' \in B_{\c},\, c' \in B,\, t \in [1,n]} m^{a',b,c'}_{r, s,t}!
\Big)
\Big(
\prod_{a' \in B_\fa \cup B_{\bar 1},\, c' \in B_{\c},\, t \in [1,n]} m^{a',b,c'}_{r, s,t}!
\Big)}.
\end{align*}
So $M$ equals 
\begin{align*}
&\Bigg(\prod_{\substack{  a' \in B_{\c}  \\ r,t \in [1,n]  }} \hspace{-2mm}[\ba, \bp, \bq]^{a'}_{r, t}!
\Bigg)\cdot 
\Bigg(
\prod_{\substack{  c' \in B_{\c}  \\ t,s \in [1,n]  }} \hspace{-2mm}[\bc, \bu, \bv]^{c'}_{t, s}!
\Bigg)
\cdot C\cdot 
\prod_{b \in B_{\c},\, r,s \in [1,n]}
y^b_{r,s}
\\
=\,\,&
\Bigg(\prod_{\substack{a' \in B_{\c}\\ r,t \in [1,n]}}
\frac{
[\ba, \bp, \bq]^{a'}_{r, t}!
}
{
\prod_{c' \in B,\, b \in B_{\c}, \, s \in [1,n]} m^{a', b,c'}_{r, s,t}! 
}
\Bigg)
\Bigg(\prod_{\substack{c' \in B_{\c}\\ t,s \in [1,n]}}
\frac{
[\bc, \bu, \bv]^{c'}_{t, s}!
}
{
\prod_{a' \in B_\fa \cup B_{\1},\, b \in B_{\c}, \, r \in [1,n]} m^{a', b,c'}_{r, s,t}! 
}
\Bigg)
\\
&\times C\cdot
\prod_{b \in B_{\c}, \, r,s \in [1,n]} [\bb, \br, \bs]^b_{r,s}!.
\end{align*}
Note that the first factor is an integer by  (\ref{subalg4b}), and the second factor is an integer by  (\ref{subalg4c}). We have thus proved that $M$ 
is divisible by  
$$\prod_{b \in B_{\c}, \, r,s \in [1,n]} [\bb, \br, \bs]^b_{r,s}!=[\bb, \br, \bs]^!_{\c},$$ 
completing the proof.
\end{proof}

\subsection{Coproduct on generalized Schur algebras}\label{SSCoproduct}
 In this subsection, it will be convenient to use the following notation. Let $\Triple=(\bb,\br,\bs)\in\Seq^B(n,d)$. We write 
$$
\xi_\Triple:=\xi^\bb_{\br,\bs},\ \,\,\eta_\Triple:=\eta^\bb_{\br,\bs}, \,\,{}^\Triple\D:={}^{\bb,\br,\bs}\D,\ \,\,[\Triple]^!_{\c}:=[\bb,\br,\bs]^!_{\c},\ \,\,\Triple\si:=(\bb,\br,\bs)\si,\ \,\,\text{etc.}
$$
If $d=d_1+d_2$, $\Triple^1=(\bb^1, \br^1, \bs^1)\in\Seq^B(n,d_1)$ and $\Triple^2=(\bb^2, \br^2, \bs^2)\in\Seq^B(n,d_2)$, we denote 
$$
\Triple^1\Triple^2:=(\bb^1\bb^2, \br^1\br^2, \bs^1\bs^2)\in B^d\times[1,n]^d\times [1,n]^d.
$$
In general $\Triple^1\Triple^2$ does not need to be an element of $\Seq^B(n,d)$. 

Recall the notation (\ref{ESeq0}). For \(\Triple \in \Seq^B_0(n,d)\) and \(0 \leq l \leq d\), define the sets of {\em $l$-splits of $\Triple$} and {\em splits of $\Triple$} as
\begin{align*}%\label{}
\textup{Spl}_l(\Triple)&:=\big\{(\Triple^1, \Triple^2)\in \Seq^B_0(n,l) \times \Seq^B_0(n,d-l) \mid 
\Triple^1\Triple^2\sim \Triple\big\},
\\
\textup{Spl}(\Triple)&:=\bigsqcup_{0\leq l\leq d}\textup{Spl}_l(\Triple).
\end{align*}
For $(\Triple^1,\Triple^2) \in \textup{Spl}_l(\Triple)$, let \(\si^{\Triple}_{\Triple^1,\Triple^2}\) be the unique element of ${}^{\Triple}\mathscr{D}$ such that 
$$
\Triple\si^{\Triple}_{\Triple^1,\Triple^2}  = \Triple^1\Triple^2.
$$ 

Let \(\iota_l : \mathfrak{S}_l \times \mathfrak{S}_{d-l} \to \mathfrak{S}_d\) be the standard inclusion. Let \(\Triple \in \Seq^B_0(n,d)\). 
Note that for every \(\sigma \in {}^{\Triple}\mathscr{D}\) there exist unique $(\Triple^1,\Triple^2) \in \textup{Spl}_l(\Triple)$,  \(\sigma_1 \in {}^{\Triple^1}\mathscr{D}\) and \(\sigma_2 \in {}^{\Triple^2}\mathscr{D}\) such that 
 \(\si=\si^{\Triple}_{\Triple^1,\Triple^2}\iota_l(\sigma_1, \sigma_2)\). In other words, the map 
\begin{equation}\label{E110217}
\bigsqcup_{\substack{ (\Triple^1,\Triple^2)\in \textup{Spl}_l(\Triple)}} 
{}^{\Triple^1}\mathscr{D} \times
{}^{\Triple^2}\mathscr{D}
\to 
 {}^{\Triple}\mathscr{D}
\end{equation}
sending \((\sigma_1, \sigma_2) \in {}^{\Triple^1}\mathscr{D} \times
{}^{\Triple^2}\mathscr{D} \) to \(\si^{\Triple}_{\Triple^1,\Triple^2}\iota_l(\sigma_1, \sigma_2)\), is a bijection.

Note that for $(\sigma_1, \sigma_2) \in {}^{\Triple^1}\mathscr{D} \times
{}^{\Triple^2}\mathscr{D}$ we have 
\begin{equation}\label{E110217_2}
\langle \si^{\Triple}_{\Triple^1,\Triple^2}\iota_l(\sigma_1, \sigma_2) ; \bb\rangle = \langle \si^{\Triple}_{\Triple^1,\Triple^2}; \bb \rangle + \langle \sigma_1; \bb^1 \rangle + \langle \sigma_2 ; \bb^2 \rangle.
\end{equation}

Recall from \cite[\S3.3]{EK1}, that \(\bigoplus_{d\geq 0} M_n(A)^{\otimes d}\) is a supercoalgebra with the coproduct \(\nabla\) defined by
\begin{align*}
\nabla\,:\,\, M_n(A)^{\otimes d}\,\, &\to \,\,\bigoplus_{l=0}^d M_n(A)^{\otimes l} \otimes M_n(A)^{\otimes (d-l)}\\
\xi_1 \otimes \cdots \otimes \xi_d\,\,&\mapsto\,\, \sum_{l=0}^d (\xi_1 \otimes \cdots \otimes \xi_l) \otimes (\xi_{l+1} \otimes \cdots \otimes \xi_d).
\end{align*}
Let 
\begin{equation}\label{ES(n)}
S^A(n) := \bigoplus_{d\geq 0} S^A(n,d)\quad \text{and}\quad T^A_\fa(n) := \bigoplus_{d\geq 0} T^A_\fa(n,d).
\end{equation}
We next prove that 
these are sub-supercoalgebras of \(\bigoplus_{d\geq 0} M_n(A)^{\otimes d}\).  The following result is actually contained in \cite{EK1}, but we give a proof using our current notation for reader's convenience.

\begin{Lemma}\label{coprodxi} {\rm \cite[(6.12)]{EK1}} 
If \(\Triple=(\bb,\br,\bs) \in \Seq^B_0(n,d)\) then 
\begin{align*}
\nabla(\xi_\Triple) =
%\sum_{l = 0}^d
 \sum_{(\Triple^1, \Triple^2) \in \textup{Spl}(\Triple)} 
(-1)^{\langle \si^{\Triple}_{\Triple^1,\Triple^2}; \bb \rangle}
\xi_{\Triple^1} \otimes \xi_{\Triple^2}.
\end{align*}
In particular, $S^A(n)$ is a sub-supercoalgebra of \(\bigoplus_{d\geq 0} M_n(A)^{\otimes d}\). 
\end{Lemma}
\begin{proof}
Writing $\displaystyle\sum_{\small\textup{Spl}_l(\Triple)}$ for the sum over all $(\Triple^1,\Triple^2) \in \textup{Spl}_l(\Triple)$ with $\Triple^1=(\bb^1, \br^1, \bs^1)$ and $\Triple^2=(\bb^2, \br^2, \bs^2)$, we have that $\nabla(\xi^\bb_{\br,\bs})$ equals
\begin{align*}
&
\sum_{\sigma \in {}^{\Triple}\mathscr{D}} (-1)^{\langle \sigma ; \bb\rangle}\nabla(\xi_{r_{\sigma 1}, s_{\sigma 1}}^{b_{\sigma 1}} \otimes \cdots \otimes \xi_{r_{\sigma d}, s_{\sigma d}}^{b_{\sigma d}})\\
=& 
\sum_{\sigma \in {}^{\Triple}\mathscr{D}} (-1)^{\langle \sigma ; \bb\rangle}
\sum_{l=0}^d
(\xi_{r_{\sigma 1}, s_{\sigma 1}}^{b_{\sigma 1}} \otimes \cdots \otimes \xi_{r_{\sigma l}, s_{\sigma l}}^{b_{\sigma l}}) \otimes ( \xi_{r_{\sigma (l+1)}, s_{\sigma (l+1)}}^{b_{\sigma (l+1)}}\otimes \cdots \otimes \xi_{r_{\sigma d}, s_{\sigma d}}^{b_{\sigma d}})\\
=&
\sum_{l=0}^d
 \sum_{\small\textup{Spl}_l(\Triple)} 
\sum_{\substack{ \sigma_1 \in {}^{\Triple^1}\mathscr{D} \\ \sigma_2 \in {}^{\Triple^2}\mathscr{D}}}
(-1)^{ \langle \si^{\Triple}_{\Triple^1,\Triple^2}; \bb \rangle}
(\xi^{b^1_1}_{r^1_1, s^1_1}  \otimes \cdots \otimes \xi^{b^1_l}_{r^1_l, s^1_l})^{\sigma_1} 
\\
&\hspace{7.5cm}\otimes 
(\xi^{b^2_1}_{r^2_1, s^2_1}  \otimes \cdots \otimes \xi^{b^2_{d-l}}_{r^2_{d-l}, s^2_{d-l}})^{\sigma_2}
\\
= &\sum_{l = 0}^d
\sum_{\small\textup{Spl}_l(\bb, \br, \bs)}
(-1)^{\langle \si^{\Triple}_{\Triple^1,\Triple^2}; \bb \rangle}
\xi_{\Triple^1} \otimes \xi_{\Triple^2},
\end{align*}
where we have used the bijection (\ref{E110217}) and the sign identity (\ref{E110217_2}) for the second equality above. 
\end{proof}

\begin{Corollary}\label{coprodeta}
If \(\Triple=(\bb,\br,\bs) \in \Seq^B_0(n,d)\) then 
\begin{align*}
%\sum_{l = 0}^d
\nabla(\eta_\Triple)=\sum_{(\Triple^1,\Triple^2) \in \textup{Spl}(\Triple)} 
(-1)^{\langle \si^{\Triple}_{\Triple^1,\Triple^2}; \bb \rangle}
{\small \frac{[\Triple]^!_{\c}}{[\Triple^1]^!_{\c}[\Triple^2]^!_{\c}}}
\eta_{\Triple^1} \otimes \eta_{\Triple^2},
\end{align*}
with $\displaystyle{\small \frac{[\Triple]^!_{\c}}{[\Triple^1]^!_{\c}[\Triple^2]^!_{\c}}}\in\Z$. In particular, $T^A_\fa(n)$ is a sub-supercoalgebra of \(\bigoplus_{d\geq 0} M_n(A)^{\otimes d}\). 
\end{Corollary}
\begin{proof}
By Lemma~\ref{coprodxi}, we just have to check that \([\Triple^1]^!_{\c}[\Triple^2]^!_{\c}  \) divides \([\Triple]^!_{\c} \) whenever \((\Triple^1, \Triple^2) \in \textup{Spl}_l(\Triple)\). But in this situation we have that \([\Triple]^b_{r,s} = [\Triple^1]^b_{r,s} + [\Triple^2]^b_{r,s}\) for all $b\in B$ and $1\leq r,s\leq n$, which implies the required divisibility.
\end{proof}

\section{Superbialgebra structure}\label{Ssuperbi}
Recall the definition of $S^A(n)$ and $T^A_\a(n)$ from (\ref{ES(n)}). 
In this section we study the star-product on $S^A(n)$ and $T^A_\a(n)$ which together with the coproduct $\nabla$ from \S\ref{SSCoproduct} makes them into superbialgbras. For $S^A(n)$ this is well-known, see for example \cite[Lemma 3.12]{EK1}.

\subsection{Star-product}
For $d,e\in Z_{\geq 0}$, let ${}^{(d,e)}\D$ be the set of the shortest coset representatives for $(\Si_d\times \Si_e)\backslash\Si_d$. Given $\xi_1\in M_n(A)^{\otimes d}$ and $\xi_2\in M_n(A)^{\otimes e}$, we define 
\begin{equation}\label{EStarNotation}
\xi_1* \xi_2:=\sum_{\si\in{}^{(d,e)}\D}(\xi_1\otimes \xi_2)^\si.
\end{equation}
%\begin{equation}\label{EStarNotation} \xi_1*\dots* \xi_q:=\sum_{\si\in{}^\de\D}(\xi_1\otimes\dots\otimes \xi_q)^\si. \end{equation}
It is well-known that this $*$-product  makes  $\bigoplus_{d\geq 0}M_n(A)^{\otimes d}$ into an associative supercommutative superalgebra.

\begin{Lemma}\label{xistarproducts} For \((\bb, \br, \bs)\in\Seq^B(n,d)\) and \((\bc, \bt, \bu) \in \Seq^B(n,e)\), we have
\begin{enumerate}
\item \(\xi_{r_1,s_1}^{b_1} * \cdots * \xi_{r_d,s_d}^{b_d} = [\bb, \br, \bs]^!\, \xi_{\br, \bs}^{\bb}\).
\item \(\displaystyle\xi^{\bb}_{\br, \bs} * \xi^{\bc}_{\bt, \bu} = \frac{[\bb\bc, \br\bt, \bs\bu]^!}{[\bb, \br, \bs]^! [\bc, \bt, \bu]^!}\xi_{\br\bt, \bs\bu}^{\bb\bc}\).
\item \(\displaystyle\eta^{\bb}_{\br, \bs} * \eta^{\bc}_{\bt, \bu} = \frac{[\bb\bc, \br\bt, \bs\bu]_\a^!}{[\bb, \br, \bs]_\a^! [\bc, \bt, \bu]_\a^!}\eta_{\br\bt, \bs\bu}^{\bb\bc},\)
\end{enumerate}
where \(\frac{[\bb\bc, \br\bt, \bs\bu]^!}{[\bb, \br, \bs]^! [\bc, \bt, \bu]^!}\) and \(\frac{[\bb\bc, \br\bt, \bs\bu]_\a^!}{[\bb, \br, \bs]_\a^! [\bc, \bt, \bu]_\a^!}\) are integers, and the right hand sides of \textup{(ii)} and \textup{(iii)} are taken to be zero when \((\bb\bc, \br\bt, \bs\bu) \notin \Seq^B(n,d+e)\).
\end{Lemma}
\begin{proof}
We have that \(\xi_{r_1,s_1}^{b_1} * \cdots * \xi_{r_d,s_d}^{b_d}\) is equal to 
\begin{align*}
\sum_{ \sigma \in \mathfrak{S}_d}( \xi_{r_1, s_1}^{b_1} \otimes \cdots \otimes \xi_{r_d, s_d}^{b_d})^\sigma
&= \sum_{ \sigma \in {}^{\bb, \br, \bs} \mathscr{D}}  \sum_{\si' \in \mathfrak{S}_{\bb, \br, \bs}}( \xi_{r_1, s_1}^{b_1} \otimes \cdots \otimes \xi_{r_d, s_d}^{b_d})^{\si'\sigma}\\
&= [\bb, \br, \bs]^!\sum_{ \sigma \in {}^{\bb, \br, \bs} \mathscr{D}}( \xi_{r_1, s_1}^{b_1} \otimes \cdots \otimes \xi_{r_d, s_d}^{b_d})^{\sigma}
\\
&= [\bb, \br, \bs]^!\, \xi_{\br, \bs}^{\bb},
\end{align*}
proving (i). Thus 
\begin{align*}
[\bb, \br, \bs]^! [\bc, \bt, \bu]^!\,\xi^{\bb}_{\br, \bs} * \xi^{\bc}_{\bt, \bu} &= (\xi_{r_1, s_1}^{b_1} * \cdots * \xi_{r_d,s_d}^{b_d})*(\xi_{t_1, u_1}^{c_1} * \cdots * \xi_{t_e,u_e}^{c_e})\\
&= \xi_{r_1, s_1}^{b_1} * \cdots * \xi_{r_d,s_d}^{b_d}*\xi_{t_1, u_1}^{c_1} * \cdots * \xi_{t_e,u_e}^{c_e}\\
&=[\bb\bc, \br\bt, \bs\bu]^! \,\xi_{\br\bt, \bs\bu}^{\bb\bc},
\end{align*}
where the last line is interpreted as $0$ if \((\bb\bc, \br\bt, \bs\bu) \notin \Seq^B(n,d+e)\). 
Therefore
\begin{align*}
[\bb, \br, \bs]^!_\a [\bc, \bt, \bu]^!_\a\,\eta^{\bb}_{\br, \bs} * \eta^{\bc}_{\bt, \bu} &= [\bb, \br, \bs]^! [\bc, \bt, \bu]^!\, \xi^{\bb}_{\br, \bs} * \xi^{\bc}_{\bt, \bu}\\
&=[\bb\bc, \br\bt, \bs\bu]^!\, \xi_{\br\bt, \bs\bu}^{\bb\bc}\\
&=[\bb\bc, \br\bt, \bs\bu]^!_\a\,\eta_{\br\bt, \bs\bu}^{\bb\bc}.
\end{align*}
Now (ii) and (iii) follow by noting that 
$$[\bb\bc, \br\bt, \bs\bu]^b_{r,s} = [\bb, \br, \bs]^b_{r,s}+ [\bc, \bt, \bu]^b_{r,s}
$$ 
for all \(b,r,s\).
\end{proof}

%Recalling (\ref{ES(n)}), the following corollary shows that $S^A(n)$ and $T^A_\a(n)$ are closed under the $*$-operation. 

\begin{Corollary} \label{L250217_3} %{\rm \cite{}}%{\bf ()}
$S^A(n)$ and $T^A_\a(n)$ are subsuperalgebras of $\bigoplus_{d\geq 0}M_n(A)^{\otimes d}$ with respect to the $*$-product. 
\end{Corollary}

Corollary~\ref{L250217_3} together with \cite[Lemma 3.12]{EK1} now imply

\begin{Corollary} \label{subsuperbi}
With respect to the coproduct $\nabla$ and the product $*$, $S^A(n)$ and $T^A_\a(n)$ are superbialgebras. 
\end{Corollary}

We will also need the following result, where the Sweedler notation $\nabla(x)=\sum x_{(1)}\otimes x_{(2)}$ is used:

\begin{Lemma}\label{Lstfo} \cite[Lemma 4.2]{EK1}
Let $x,y,z,u\in S^A(n,d).$ Then 
\[
(x*y) (z*u) = \sum (-1)^{s}
  (x_{(1)} z_{(1)}) *  (y_{(1)}z_{(2)}) * (x_{(2)} u_{(1)}) * (y_{(2)} u_{(2)}),
\]
where $s=(\bar x_{(2)}+\bar y_{(2)}) \bar z + \bar y_{(1)} (\bar x_{(2)} +  \bar z_{(1)}) 
+ \bar y_{(2)} \bar u_{(1)}$. 
\end{Lemma}

\subsection{Separation}
Let $q\in\Z_{>0}$ and $\de=(d_1,\dots,d_q)\in\Z_{\geq 0}^q$ with $d_1+\dots+d_q=d$. Then $\Si_\de:=\Si_{d_1}\times\dots\times\Si_{d_q}\leq \Si_d$. 
Suppose that for each $m=1,\dots,q$, we are given 
$$
(\ba^{(m)},\br^{(m)}, \bs^{(m)}),\,(\bc^{(m)},\bt^{(m)}, \bu^{(m)})
\in\Seq^H(n,d_m).
$$ 
We write $$\ba^{(m)}=a^{(m)}_1\cdots a^{(m)}_{d_m},\ \br^{(m)}=r^{(m)}_1\cdots r^{(m)}_{d_m},\ \text{etc.}
$$ 
Let 
$$%\begin{align*}
\ba=\ba^{(1)}\dots\ba^{(q)},\ \br=\br^{(1)}\dots\br^{(q)},
%\ \bs=\bs^{(1)}\dots\bs^{(q)},\ \bc=\bc^{(1)}\dots\bc^{(q)},
\ \text{etc.}
%\bt=\bt^{(1)}\dots\bt^{(q)},\ \bu=\bu^{(1)}\dots\bu^{(q)},
$$ %\end{align*}
We also write 
$$\ba=a_1\cdots a_{d},\ \br=r_1\cdots r_d,\ \text{etc.}
$$ 
The triple $(\ba,\br, \bs)$ is called {\em $\de$-separated} if 
$1\leq m\neq l\leq q$ implies 
$$(a^{(m)}_t,r^{(m)}_t,s^{(m)}_t)\neq (a^{(l)}_u,r^{(l)}_u,s^{(l)}_u)$$ for all $1\leq t\leq d_m$ and $1\leq u\leq d_l$. Note that we then automatically have  $(\ba,\br, \bs)
\in\Seq^H(n,d)$.

%, whenever $m\neq n$.

%The triples $(\ba,\br, \bs)$ and $(\bc,\bt, \bu)$  are called {\em $\de$-immiscible} modulo $I$ if for any $\si\in\Si_d\setminus\Si_\de$ there exists $1\leq k\leq d$ such that $\xi^{a_k}_{r_k,s_k}\xi^{c_{\si k}}_{t_{\si k},u_{\si k}}=0.$

\begin{Lemma} \label{LSingleSep} %{\rm \cite{}}%{\bf ()}
If $(\ba,\br, \bs)$ is $\de$-separated then  
$$
\xi^\ba_{\br,\bs}=\xi_{\br^{(1)},\bs^{(1)}}^{\ba^{(1)}}*\dots*\xi_{\br^{(q)},\bs^{(q)}}^{\ba^{(q)}}
\quad
\text{and} 
\quad
\eta^\ba_{\br,\bs}=\eta_{\br^{(1)},\bs^{(1)}}^{\ba^{(1)}}*\dots*\eta_{\br^{(q)},\bs^{(q)}}^{\ba^{(q)}}.
$$
\end{Lemma}
\begin{proof}
Recalling the notation (\ref{ESeq0}), for each \(1 \leq t \leq q\), there exists \((\hat{\ba}^{(t)}, \hat{\br}^{(t)}, \hat{\bs}^{(t)}) \in \Seq^H_0(n,d_t)\) such that \((\hat{\ba}^{(t)}, \hat{\br}^{(t)}, \hat{\bs}^{(t)}) \sim (\ba^{(t)}, \br^{(t)}, \bs^{(t)})\).
Write \(\hat{\ba}:= \hat{\ba}^{(1)} \cdots \hat{\ba}^{(q)}, \hat{\br}:= \hat{\br}^{(1)} \cdots \hat{\br}^{(q)}, \hat{\bs}:= \hat{\bs}^{(1)} \cdots \hat{\bs}^{(q)}\). Then \((\hat{\ba},\hat{\br},\hat{ \bs}) \sim (\ba, \br, \bs)\), and \((\hat{\ba},\hat{\br},\hat{ \bs})\) is \(\delta\)-separated. Moreover we have that $\Si_{\hat{\ba}, \hat{\br}, \hat{\bs}} \leq \Si_\de$, and both groups are standard parabolic subgroups of \(\Si_d\). Using (\ref{EXiDef}), we get
\begin{align*}
\xi^{\hat{\ba}}_{\hat\br,\hat\bs}
&=\sum_\si (\xi_{\hat r_{1}, \hat s_{1}}^{\hat a_{1}} \otimes \cdots \otimes \xi_{\hat \hat r_{d},\hat s_{d}}^{ \hat a_{d}})^\si
\\
&=\sum_{\si',\si''} (\xi_{\hat r_{1},\hat s_{1}}^{\hat a_{1}} \otimes \cdots \otimes \xi_{\hat r_{d},\hat s_{d}}^{\hat a_{d}})^{\si'\si''}
\\
&=\sum_{\si''} (\xi_{\hat \br^{(1)},\hat \bs^{(1)}}^{\hat \ba^{(1)}}\otimes \dots \otimes \xi_{\hat \br^{(q)}, \hat \bs^{(q)}}^{\hat \ba^{(q)}})^{\si''}
\\
&=\xi_{\hat \br^{(1)},\hat \bs^{(1)}}^{\hat \ba^{(1)}}*\dots*\xi_{\hat \br^{(q)},\hat \bs^{(q)}}^{\hat \ba^{(q)}},
\end{align*}
where $\si$ runs over ${}^{\hat \ba,\hat \br,\hat \bs}\D$, $\si'$ runs over all shortest coset representatives for $\Si_{\hat \ba,\hat \br,\hat \bs}\backslash\Si_\de$ and $\si''$ runs over 
all shortest coset representatives for $\Si_{\de}\backslash\Si_d$. 

Since \((\hat{\ba}^{(t)}, \hat{\br}^{(t)}, \hat{\bs}^{(t)}) \sim (\ba^{(t)}, \br^{(t)}, \bs^{(t)})\) for all \(1 \leq t \leq q\), we have
\begin{align*}
(-1)^{\langle \ba, \br, \bs \rangle + \langle \hat \ba, \hat \br, \hat \bs \rangle} = (-1)^{\langle \ba^{(1)}, \br^{(1)}, \bs^{(1)} \rangle + \cdots + \langle \ba^{(q)}, \br^{(q)}, \bs^{(q)} \rangle + \langle \hat \ba^{(1)}, \hat \br^{(1)}, \hat \bs^{(1)} \rangle + \cdots +  \langle \hat \ba^{(q)}, \hat \br^{(q)}, \hat \bs^{(q)}\rangle}.
\end{align*}

Then, using Lemma \ref{LXiZero}, we have
\begin{align*}
\xi^{\ba}_{\br, \bs} &= (-1)^{\langle \ba, \br, \bs \rangle + \langle \hat \ba, \hat \br, \hat \bs \rangle} \xi^{\hat \ba}_{\hat \br, \hat \bs} = (-1)^{\langle \ba, \br, \bs \rangle + \langle \hat \ba, \hat \br, \hat \bs \rangle} \xi_{\hat \br^{(1)},\hat \bs^{(1)}}^{\hat \ba^{(1)}}*\dots*\xi_{\hat \br^{(q)},\hat \bs^{(q)}}^{\hat \ba^{(q)}}\\
&=(-1)^{\langle \ba^{(1)}, \br^{(1)}, \bs^{(1)} \rangle +   \langle \hat \ba^{(1)}, \hat \br^{(1)}, \hat \bs^{(1)} \rangle}\xi^{\hat{\ba}^{(1)}}_{\hat{\br}^{(1)}, \hat{\bs}^{(1)}}
\\&\hspace{5 cm} * \cdots *
(-1)^{\langle \ba^{(q)}, \br^{(q)}, \bs^{(q)} \rangle +   \langle \hat \ba^{(q)}, \hat \br^{(q)}, \hat \bs^{(q)} \rangle}\xi^{\hat{\ba}^{(q)}}_{\hat{\br}^{(q)}, \hat{\bs}^{(q)}}\\
&= \xi_{\br^{(1)},\bs^{(1)}}^{\ba^{(1)}}*\dots*\xi_{ \br^{(q)},\bs^{(q)}}^{ \ba^{(q)}},
\end{align*}
as desired. 
The result for $\eta$'s follows from the result on $\xi$'s. 
\end{proof}

\begin{Lemma} \label{LSep} %{\rm \cite{}}%{\bf ()}
Let $(\ba,\br, \bs)$ and $(\bc,\bt, \bu)$ be $\de$-separated and suppose that 
$$
(\xi_{\br^{(1)},\bs^{(1)}}^{\ba^{(1)}}\otimes\dots\otimes\xi_{\br^{(q)},\bs^{(q)}}^{\ba^{(q)}})^\si(\xi_{\bt^{(1)},\bu^{(1)}}^{\bc^{(1)}}\otimes\dots\otimes\xi_{\bt^{(q)},\bu^{(q)}}^{\bc^{(q)}})^{\si'}=0
$$
whenever $\si$ and $\si'$ are distinct elements of ${}^\de\D$. Then 
$$
\xi^\ba_{\br,\bs}\xi^\bc_{\bt,\bu}=\pm(\xi_{\br^{(1)},\bs^{(1)}}^{\ba^{(1)}}\xi_{\bt^{(1)},\bu^{(1)}}^{\bc^{(1)}})*\dots*(\xi_{\br^{(q)},\bs^{(q)}}^{\ba^{(q)}}\xi_{\bt^{(q)},\bu^{(q)}}^{\bc^{(q)}}).
$$
Moreover, if $a_1,\dots,a_d$ or $c_1,\dots, c_d$ are all even, then the sign in the right hand side is $+$. 

\end{Lemma}
\begin{proof}
By Lemma~\ref{LSingleSep}, $\xi^\ba_{\br,\bs}\xi^\bc_{\bt,\bu}$ equals
$$
\left(\sum_{\si\in{}^\de\D}(\xi_{\br^{(1)},\bs^{(1)}}^{\ba^{(1)}}\otimes\dots\otimes\xi_{\br^{(q)},\bs^{(q)}}^{\ba^{(q)}})^\si\right)
\left(\sum_{\si'\in{}^\de\D}(\xi_{\bt^{(1)},\bu^{(1)}}^{\bc^{(1)}}\otimes\dots\otimes\xi_{\bt^{(q)},\bu^{(q)}}^{\bc^{(q)}})^{\si'}\right),
$$
and the result follows. 
\end{proof}

\iffalse{
%\subsection{Blocks of Schurifications}
\begin{Lemma}
Let \((A', \fa')\) and \((A'', \fa'')\) be good pairs. Write \(A := A' \oplus A''\) and \(\fa := \fa' \oplus \fa''\). Then we have
\begin{align*}
S^A(n,d)& \cong \bigoplus_{d' + d'' = d}S^{A'}(n,d') \otimes S^{A''}(n,d'').\\
T^A_\fa(n,d)& \cong \bigoplus_{d' + d'' = d}T^{A'}_{\fa'}(n,d') \otimes T^{A'}_{\fa''}(n,d'').
\end{align*}
as superalgebras.
\end{Lemma}
\begin{proof}
We prove the claim for \(S^A(n,d)\); the proof for \(T^A_\fa(n,d)\) is similar.
It follows from Lemma \ref{LSingleSep} that, for any \((\bb, \br, \bs) \in \Seq^B(n,d)\), we have \(\xi^{\bb}_{\br, \bs} = \pm \xi^{\bb'}_{\br', \bs'} * \xi^{\bb''}_{\br'',\bs''}\), for some \(d' + d'' =d\), \((\bb', \br', \bs') \in \Seq^{B'}(n,d')\) and \((\bb'', \br'', \bs'') \in \Seq^{B''}(n,d'')\). Then
\begin{align*}
S^A(n,d)&= \bigoplus_{d'+ d'' = d}S^{A'}(n,d') * S^{A''}(n,d'')
\end{align*}
as \(\k\)-supermodules. Moreover, it follows from Lemma \ref{Lstfo} that 
\begin{align*}
S^{A'}(n,d') * S^{A''}(n,d'') \cong S^{A'}(n,d') \otimes S^{A''}(n,d'')
\end{align*}
as \(\k\)-superalgebras for all \(d',d''\). So, again using Lemma \ref{Lstfo}, we have 
\begin{align*}
S^A(n,d)&\cong \bigoplus_{d' + d'' = d}S^{A'}(n,d') \otimes S^{A''}(n,d'')
\end{align*}
as \(\k\)-superalgebras.
\end{proof}
}\fi

The following result allows one to reduce the study of $S^A(n,d)$ to the blocks of $A$, and similarly for $T^A_\a(n,d)$. 

\begin{Lemma}
Let \(m \in \ZZ_{>0}\). For \(t \in [1,m]\) assume that \((A_t, \fa_t)\) is a good pair. Write \(A:= \bigoplus_{t=1}^m A_t\) and \(\fa := \bigoplus_{t=1}^m \fa_t\). Then we have
\begin{align*}
S^A(n,d) \cong \bigoplus_{\nu \in \La(m,d)} \bigotimes_{t=1}^m S^{A_t}(n,\nu_t)
\qquad
\textup{and}
\qquad
T^A_\fa(n,d) \cong \bigoplus_{\nu \in \La(m,d)} \bigotimes_{t=1}^m T^{A_t}_{\fa_t}(n,\nu_t)
\end{align*}
as \(\k\)-superalgebras.
\end{Lemma}
\begin{proof}
For \(t \in [1,m]\), let \(B_t\) be the designated \((A_t, \fa_t)\)-basis, and set \(B= \sqcup_{t=1}^m B_t\) as the designated \((A,\fa)\)-basis. It follows from Lemma \ref{LSingleSep} that, for any \((\bb, \br, \bs) \in \Seq^B(n,d)\), we have \(\xi^{\bb}_{\br, \bs} = \pm \xi^{\bb^{(1)}}_{\br^{(1)}, \bs^{(1)}} * \cdots * \xi^{\bb^{(m)}}_{\br^{(m)},\bs^{(m)}}\) for some \(\nu \in \La(m,d)\) and \((\bb^{(t)}, \br^{(t)}, \bs^{(t)}) \in \Seq^{B_t}(n,\nu_t)\) for \(t \in [1,m]\). So we may write
\begin{align*}
S^A(n,d)&= \bigoplus_{\nu \in \La(m,d)}S^{A_1}(n,\nu_1) * \cdots * S^{A_m}(n,\nu_m)
\end{align*}
Inductive application of Lemma \ref{Lstfo} shows that this is a decomposition of \(S^A(n,d)\) into subalgebras. Moreover, it follows as well from Lemma \ref{Lstfo} that for all \(\nu \in \La(m,d)\) we have
\begin{align*}
S^{A_1}(n,\nu_1) * \cdots * S^{A_m}(n,\nu_m) \cong S^{A_1}(n,\nu_1) \otimes \cdots \otimes S^{A_m}(n,\nu_m)
\end{align*}
as \(\k\)-superalgebras, proving the claim for \(S^A(n,d)\). The proof of the claim for \(T^A_\fa(n,d)\) proceeds exactly as above, since Lemma \ref{LSingleSep} provides an analogous result for \(\eta\)'s.
\end{proof}

\subsection{Generation}

We define 
\begin{align*}
&Y:=\spa(\xi_{r,s}^b\mid r,s\in[1,n],\ b\in B_\c\sqcup B_\1)\subseteq M_n(A),
\\
&\Star^d Y :=\underbrace{Y* \dots * Y}_{d\ \text{times}}\subseteq T^A_\a(n,d),
\end{align*}
where the second inclusion comes from Corollary~\ref{L250217_3}.
Note also that $S^\a(n,d)\subseteq T^A_\a(n,d)$ since by definition, for $\bb\in B_\a^d$, we have $\xi^\bb_{\br,\bs}=\eta^\bb_{\br,\bs}$. The following is a generalization of \cite[Lemma 4.30]{EK1}. 

\begin{Lemma} \label{L020218} %{\rm \cite{}}%{\bf ()}
We have 
\[
T^A_\a(n,d) = \bigoplus_{e=0}^d S^\a(n,d-e) * \Star^e Y.
\]
\end{Lemma}
\begin{proof}
As $S^\a(n,d-e)\subseteq T^A_\a(n,d-e)$ and $Y\subseteq T^A_\a(n,1)$, the right hand side is contained in the left hand side thanks to Corollary~\ref{L250217_3}. For the converse containment, we only need to prove that every $\eta^\bb_{\br,\bs}$ with $(\bb,\br,\bs)\in\Seq^B(n,d)$ is contained in the right hand side. For any $b\in B$ and $r,s\in[1,n]$, denote $m^b_{r,s}:=[\bb,\br,\bs]^b_{r,s}$ and set 
$$e:=\sum_{\ttb\in B_\c\sqcup B_\1,\,r,s\in[1,n]}m^\ttb_{r,s}.
$$ 
Using the fact that $m^b_{r,s}\in\{0,1\}$ for all $b\in B_\1$, Lemma~\ref{LSingleSep} and the definition of $\eta^\bb_{\br,\bs}$, we see that 
\begin{align}\label{etastar}
\eta^\bb_{\br,\bs}=\pm\left(\Conv_{\ttb\in B_\a,\,r,s\in[1,n]}\, \big( (\xi^\ttb_{r,s})^{\otimes m^\ttb_{r,s}} \big)\right)*\left(\Conv_{\ttb\in B_\c\sqcup B_\1,\,r,s\in[1,n]}\, \big( (\xi^\ttb_{r,s})^{* m^\ttb_{r,s}} \big)\right),
\end{align}
with the first term in $S^\a(n,d-e)$ and the second term in $\Star^e Y$. 
\end{proof}

\begin{Proposition}\label{AaInd}
The algebra \(T^A_\fa(n,d)\) depends only on the subalgebra \(\fa\), and not on the choice of the $(A,\a)$-basis \(B\).
\end{Proposition}
\begin{proof}
Let \(B= B_{\a}\sqcup B_{\c} \sqcup B_{\bar 1}\) and \(B'= B'_{\a}\sqcup B'_{\c} \sqcup B'_{\bar 1}\) be distinct choices of \((A,\fa)\)-bases, $Y=\spa(\xi_{r,s}^b\mid r,s\in[1,n],\ b\in B_\c\sqcup B_\1)$, and $Y'=\spa(\xi_{r,s}^b\mid r,s\in[1,n],\ b\in B'_\c\sqcup B'_\1)$. 
As $B'_\c\subseteq \spa(B_\c\sqcup B_\a)$, we deduce that 
$$
\Star^e Y'\subseteq\bigoplus_{f=0}^e\Star^{e-f}Y*\Star^f\a.
$$ 
Therefore by Lemma~\ref{L020218}, the algebra ${}'T^A_\a(n,d)$ defined using the basis $B'$ is contained in the algebra $T^A_\a(n,d)$ defined using the basis $B$. Similarly, $T^A_\a(n,d)\subseteq {}'T^A_\a(n,d)$.  
\end{proof}

Let $\tau$ be an anti-involution on $A$, such that $\tau(\a)=\a$. 
Then it is easy to see using Proposition~\ref{AaInd} that the involution 
$
\tau_{n,d}$ on $S^A(n,d)$ defined in (\ref{TauND}) restricts to the involution of $T^A_\a(n,d)$. Moreover, if $\tau(B_\a)=B_\a$, 
$\tau(B_\c)=B_\c$ and $\tau(B_\1)=B_\1$, then we have 
\begin{equation}\label{EAntiinv}
\tau_{n,d}:T^A_\a(n,d)\to T^A_\a(n,d),\ \eta^{\bb}_{\br,\bs}\mapsto \eta^{\bb^\tau}_{\bs,\br}.
\end{equation}

The following theorem generalizes \cite[Theorem 4.31]{EK1}.

\begin{Theorem} \label{TGen} %{\rm \cite{}}%{\bf ()}
Suppose that $(A,\a)$ is a unital good pair and let $1:= 1_{M_n(A)}$. Then $T^A_\a(n,d)$ is the subalgebra of $S^A(n,d)$ generated by $S^\a(n,d)$ and $1^{\otimes d-1}*Y:=\{1^{\otimes d-1}* y\mid y\in Y\}$. 
\end{Theorem}

\begin{proof}
Let $U$ be the subalgebra of $T^A_\a(n,d)$ generated by $S^\a(n,d)$ and $1^{\otimes d-1}*Y$. We show by induction on \(e =0,\dots,d\) that \(U\) contains every element of the form \(\eta^{\bb}_{\br, \bs} * 1^{\otimes(d-e)}\), where \((\bb, \br, \bs) \in \Seq^B(n,e)\). This proves the theorem in the case \(d=e\). 

The base case \(e=0\) is clear. Let \(0< e\leq d\). Let \((\bb', \br', \bs') \in \Seq^B(n,e)\). We will show using the inductive assumption that \(\eta^{\bb'}_{\br', \bs'} * 1^{\otimes (d-e)} \in U\). If \(\bb' \in B_\a^{e}\), then \(\eta^{\bb'}_{\br', \bs'} *1^{\otimes(d-e)} \in S^\a(n,d) \subseteq U\), and we are done. So we may assume  that \((\bb', \br', \bs')=(\bb b, \br r, \bs s)\), for some \((\bb, \br, \bs) \in \Seq^B(n,e-1)\), \(b \in B_{\c} \cup B_{\1}\) and \(r,s \in [1,n]\).

By the induction assumption, \(\eta^{\bb}_{\br, \bs} * 1^{\otimes(d-e+1)} \in U\), and we also have \(1^{\otimes (d-1)} * \xi^{b}_{r,s}  \in  1^{\otimes (d-1)} *Y\subseteq U\). Thus the following product is contained in \(U\):
\begin{align*}
&(\eta^{\bb}_{\br, \bs} * 1^{\otimes(d-e+1)})(1^{\otimes (d-1)} * \xi^b_{r,s})
\\
=&\pm(\eta^{\bb}_{\br, \bs})_{(1)}
* 1^{\otimes(d-e+1)}
*(\eta^{\bb}_{\br, \bs})_{(2)}\xi^b_{r,s}  \pm \eta^{\bb}_{\br, \bs} * 1^{\otimes (d-e)} * \xi^b_{r,s}
\\
=&\pm(\eta^{\bb}_{\br, \bs})_{(1)}
*(\eta^{\bb}_{\br, \bs})_{(2)}\xi^b_{r,s} * 1^{\otimes(d-e+1)} \pm \eta^{\bb}_{\br, \bs}  * \xi^b_{r,s}* 1^{\otimes (d-e)},
\end{align*}
where the equalities come from Lemma \ref{Lstfo} and the supercommutativity of $*$. Note that by Corollary \ref{subsuperbi}, we have that 
\( (\eta^{\bb}_{\br, \bs})_{(1)}*(\eta^{\bb}_{\br, \bs})_{(2)}\xi^b_{r,s}\)
belongs to \(T^A_\a(n,e-1)\), and thus may be written as a linear combination of elements of the form \(\eta^{\bb''}_{\br'', \bs''}\), where \((\bb'', \br'', \bs'') \in \Seq^B(n,e-1)\). Thus the induction assumption implies that the term
\(
(\eta^{\bb}_{\br, \bs})_{(1)}*(\eta^{\bb}_{\br, \bs})_{(2)}\xi^b_{r,s} * 1^{\otimes(d-e+1)}
\)
belongs to \(U\), which in turn implies that 
\(
\eta^{\bb}_{\br, \bs}  * \xi^b_{r,s}* 1^{\otimes (d-e)} \in U\). But since \(b \in B_\c \cup B_\1\), we have as in (\ref{etastar}) that
\(
\eta^{\bb'}_{\br',\bs'} = \pm \eta^{\bb}_{\br,\bs} * \xi^b_{r,s},
\)
so
\begin{align*}
\eta^{\bb'}_{\br', \bs'} * 1^{\otimes(d-e)} = \pm \eta^{\bb}_{\br,\bs}  *  \xi^b_{r,s} *1^{\otimes(d-e)}\in U,
\end{align*}
completing the induction step, and the proof.
\end{proof}

\section{Miscellaneous properties and examples}\label{SExamples}
Throughout the section, $(A,\a)$ is a fixed good pair with an $(A,\a)$-basis $B= B_{\fa}\sqcup B_{\c} \sqcup B_{\bar 1}$ as in (\ref{AaBasis}). %We establish various miscellaneous properties of the generalized Schur algebras $T^A_\a(n,d)$ and $S^A(n,d)$ which will be used in \cite{greenThree}. 
Much of this section deals with various idempotent truncations. If $e\in A$ is an idempotent, we say that $B$ is {\em $e$-admissible} if $ebe=b$ or $ebe=0$ for all $b\in B$. We say that we say that $B$ is {\em right $e$-admissible} if $be=b$ or $be=0$ for all $b\in B$. 

\subsection{Idempotents and characters}\label{SSIdCh}
Throughout the section, let \(e_0, \ldots, e_\ell \in \a\) be a set of orthogonal idempotents. 
We do not assume that $\sum_{i=0}^\ell e_i=1$, and usually we do not make any admissibility assumptions on $B$. 
Set \(I=[0,\ell]\). % and $e:=e_0+\dots+e_\ell$. 

Let \(\La(n):=\Z_{\geq 0}^n\) and $\La^I(n):=\La(n)^I$. We think of the elements of $\La(n)$ as compositions $\la=(\la_1,\dots,\la_n)$ and the elements of $\La^I(n)$ as tuples \(\bla = (\la^{(0)}, \ldots, \la^{(\ell)})\) of compositions. For such $\la\in \La(n)$ and $\bla \in\La^I(n)$, we set \(|\la| := \sum_{r=1}^n \la_r\), $|\bla|:=\sum_{i\in I}|\la^{(i)}|$, and, for any $d\in\Z_{\geq0}$, we define 
$$
\La(n,d):=\{\la\in\La(n)\mid |\la|=d\},\quad \La^I(n,d):=\{\bla\in\La(n)\mid |\bla|=d\}.
$$
The group $\Si_n$ acts on $\La(n)$ via  
$$
\si\la:=(\la_{\si^{-1}1},\dots , \la_{\si^{-1}n}).
$$
The group $\Si_n^I:= \prod_{i \in I} \Si_n$ acts on \(\La^I(n)\) via 
$$
\bsi\bla:=(\si^{(0)}\la^{(0)},\dots , \si^{(\ell)}\la^{(\ell)}),
$$
for $\bsi=(\si^{(0)},\dots,\si^{(\ell)})\in\Si_n^I$ and $\bla=(\la^{(0)},\dots,\la^{(\ell)})\in\La^I(n)$. 

\iffalse{
Let \(\La^I_+(n,d) \subseteq \La^I(n,d)\) be the subset of {\em multipartitions}; i.e. \(\bla \in \La^I_+(n,d)\) if \(\la^{(i)}_j \geq \la^{(i)}_k\) for all \(i \in I\) and \(1 \leq j \leq k \leq n\). The group \(\Si_n^I:= \prod_{i \in I} \Si_n\) acts naturally on \(\La^I(n,d)\) by permuting composition entries, and the orbit of any \(\bmu \in \La^I(n,d)\) contains a unique element of \(\La^I_+(n,d)\) which we denote with \(\bmu_+\).
}\fi

To $\la\in\La(n,d)$ we associate the word $\bl^\la=1^{\la_1}\cdots n^{\la_n}\in [1,n]^d$. For any idempotent $f\in A$ we have an idempotent  
$$
\xi_\la^f:=\xi^{f^d}_{\bl^\la,\bl^\la}\in S^A(n,d).
$$
Note using Lemma~\ref{L140117} that $\xi_\la^f\in T^A_\a(n,d)$ if $f\in\a$. Define 
\begin{equation}\label{E070318_2}
\xi^f:= \sum_{\la \in \La(n,d)} \xi^f_\la.
\end{equation}
If, for any $a\in A$ we define 
\begin{equation}\label{Ea}
E^a:=\sum_{r=1}^n\xi_{r,r}^a\in M_n(A),
\end{equation} 
then 
\begin{equation}\label{E240318}
\xi^f=E^f\otimes\dots\otimes E^f.
\end{equation} 
If $A$ is unital, we denote 
\begin{equation}\label{E070318_3}
\xi_\la:=\xi_\la^{1_A}. 
\end{equation}
Then
$
1_{S^A(n,d)}=\sum_{\la\in\La(n,d)}\xi_\la
$
is an orthogonal idempotent decomposition. If the pair $(A,\a)$ is unital, then $\xi_\la\in T^A_\a(n,d)$ for all $\la\in\La(n,d)$. 
For \(\bmu \in \La^I(n,d)\), define:
\begin{equation}\label{EBla}
e_{\bla}:=\xi_{\la^{(0)}}^{e_0} * \cdots * \xi_{\la^{(\ell)}}^{e_\ell} \in T^A_\a(n,d).
\end{equation}

For $a\in A$ and \(\sigma \in \Si_n\), let \(\xi_\sigma^a:=\sum_{r=1}^n\xi_{\si(r),r}^a\in M_n(A)\) be the permutation matrix corresponding to $\si$ multiplied by $a$. 
For \(\bsi=(\si^{(0)},\dots,\si^{(\ell)}) \in \Si^I_n\), we set
\begin{align*}
\xi_\bsi:=\sum_{(\de_0,\dots,\de_\ell) \in \La^I(d)} (\xi_{\sigma^{(0)}}^{e_0})^{\otimes \de_0} * \cdots * (\xi_{\sigma^{(\ell)}}^{e_\ell})^{\otimes \de_\ell}\in T^A_\a(n,d).
\end{align*}

%Writing \(e=\sum_{i \in I}e_i\), we have \(\xi_{\textup{id}} = 1_{\xi^e T^A_\a(n,d)\xi^e}\), and i
%The following lemma implies that \(\spa\{\xi_\bsi \mid \bsi \in \Si_n^I\}\cong \k\Si^I_n\). % in \(T^A_\a(n,d)\).  

\begin{Lemma}\label{L070318}
For all \(\bsi, \btau \in \Si^I_n\), we have \(\xi_\bsi \xi_\btau = \xi_{\bsi \btau}\).
\end{Lemma}
\begin{proof}
This follows easily from Lemma~\ref{LSep}. 
\end{proof}

\begin{Lemma}\label{symchar}
If \(\bla \in \La^I(n,d)\) and \(\bsi \in \Si^I_n\), then we have \(\xi_{\bsi}e_{\bla}\xi_{\bsi^{-1}} = e_{\bsi \bla}\).
\end{Lemma}
\begin{proof}
Let $d_i=|\la^{(i)}|$ for all $i\in I$. 
Using Lemma \ref{LSep}, we get
\begin{align*}
\xi_{\bsi}e_{\bla}\xi_{\bsi^{-1}} &=
\left((\xi_{\sigma^{(0)}}^{e_0})^{\otimes d_0} \xi^{e_0}_{\la^{(0)}}(\xi_{(\sigma^{(0)})^{-1}}^{e_0})^{\otimes d_0}\right)
* 
\cdots 
*
\left((\xi_{\sigma^{(\ell)}}^{e_\ell})^{\otimes d_\ell} \xi^{e_\ell}_{\la^{(\ell)}}(\xi_{(\sigma^{(\ell)})^{-1}}^{e_\ell})^{\otimes d_\ell}\right).
\end{align*}
For all \(i \in I\), we have
\begin{align*}
(\xi_{\sigma^{(i)}}^{e_i})^{\otimes d_i} \xi^{e_i}_{\la^{(i)}}(\xi_{(\sigma^{(i)})^{-1}}^{e_i})^{\otimes d_i}
&=(\xi^{e_i}_{\sigma^{(i)}})^{\otimes d_i}\left( (\xi^{e_i}_{1,1})^{\otimes \la^{(i)}_1} * \cdots * (\xi^{e_i}_{n,n})^{\otimes \la^{(i)}_n} \right) (\xi^{e_i}_{(\sigma^{(i)})^{-1}})^{\otimes d_i}\\
 &=(\xi^{e_i}_{\sigma^{(i)}}\xi^{e_i}_{1,1}\xi^{e_i}_{(\sigma^{(i)})^{-1}})^{\otimes \la^{(i)}_1} * \cdots * (\xi^{e_i}_{\sigma^{(i)}}\xi^{e_i}_{n,n}\xi^{e_i}_{(\sigma^{(i)})^{-1}})^{\otimes \la^{(i)}_n}\\
 &=(\xi^{e_i}_{\sigma^{(i)}1,\sigma^{(i)}1})^{\otimes \la^{(i)}_1} * \cdots * (\xi^{e_i}_{\sigma^{(i)}n,\sigma^{(i)}n})^{\otimes \la^{(i)}_n}\\
 &=(\xi^{e_i}_{1,1})^{\otimes \la^{(i)}_{(\sigma^{(i)})^{-1} 1}} * \cdots * (\xi^{e_i}_{n,n})^{\otimes \la^{(i)}_{(\sigma^{(i)})^{-1} n}}\\
 &= \xi^{e_i}_{\sigma^{(i)}\la^{(i)}},
\end{align*}
where we have used the commutativity of $*$-product on even elements for the penultimate equality. 
So the result follows.
\end{proof}

We consider \(\La^I(n)\) as an abelian monoid, where \(\bla=\bmu + \bnu\) when \(\la^{(i)}_r = \mu^{(i)}_r + \nu^{(i)}_r\) for all \(i \in I\) and \(r \in [1,n]\).

\begin{Lemma}\label{coprodchar}
For \(\bla \in \La^I(n,d)\), we have 
\begin{align*}
\nabla(e_{\bla}) = \sum_{\substack{\bmu,\, \bnu \in \La^I(n) \\ \bmu + \bnu = \bla}} e_{\bmu} \otimes e_{\bnu}.
\end{align*}
\begin{proof}
The result follows by Corollary \ref{coprodeta}.
Indeed, recalling the notation of \S\ref{SSCoproduct}, note that \(e_{\bla} = \xi_{\mathcal{T}}\), where \(\mathcal{T}=(\bb, \br, \bs)\), with
$\bb = e_0^{|\la^{(0)}|} \cdots e_\ell^{|\la^{(\ell)}|}$, 
$\br= \bs= \bl^{\la^{(0)}} \cdots \bl^{\la^{(\ell)}}$. 
Then \((\mathcal{T}^1,\mathcal{T}^2) \in \textup{Spl}(\mathcal{T})\) if and only if \(\mathcal{T}^1 = (\bb^1, \br^1, \bs^1)\) with
$\bb^1 = e_0^{|\mu^{(0)}|} \cdots e_\ell^{|\mu^{(\ell)}|}$, 
$\br^1= \bs^1= \bl^{\mu^{(0)}} \cdots \bl^{\mu^{(\ell)}}$ 
and 
\(\mathcal{T}^2 = (\bb^2, \br^2, \bs^2)\), with
$
\bb^2 = e_0^{|\nu^{(0)}|} \cdots e_\ell^{|\nu^{(\ell)}|}$, 
$\br^2= \bs^2= \bl^{\nu^{(0)}} \cdots \bl^{\nu^{(\ell)}}$ 
such that  \(\bmu + \bnu = \bla\).
\end{proof}
\end{Lemma}

Define 
$$
R:=\Z[t]/(t^2-1), 
$$  
and denote the image of $t$ in the quotient ring by $\pi$, so that $\pi^\eps$ makes sense for $\eps\in\Z/2$. Writing the operation in the monoid $\La^I(n,d)$ multiplicatively, denote by $R \La^I(n,d)$ the corresponding $R$-monoid algebra. This algebra inherits the $\Si^I_n$-action from that on $\La^I(n,d)$. Since this action is by algebra automorphisms, we have the invariant algebra \((R\La^I(n,d))^{\mathfrak{S}_n^I}\).

If $V$ is a free $\k$-module of finite rank, we denote its rank by $\dim V$. 
If $V$ be a free $\k$-supermodule of finite rank, its super-rank is defined to be $\dim_\pi V:= \dim V_\0+(\dim V_\1)\pi\in R$. Let $W$ be a $T^A_\a(n,d)$-supermodule.  If $e_\bla W$  is 
free of finite rank as a $\k$-supermodule for all $\bla\in\La^I(n,d)$, we say that $W$ is a {\em supermodule with free weight spaces}.  In this case, the (formal) {\em character} of $W$ is defined to be 
\begin{align*}
\ch_\pi  W:=\sum_{\bla\in \La^I(n,d)}(\dim_\pi e_\bla W)\, \bla \in R \La^I(n,d).
\end{align*}

\begin{Lemma} %\label{}%{\rm \cite{}}%{\bf ()}
If $W$ is a\, $T^A_\a(n,d)$-supermodule with free weight spaces then $\ch_\pi  W \in (R\La^I(n,d))^{\mathfrak{S}_n^I}$.\end{Lemma}
\begin{proof}
By Lemma \ref{symchar}, we have that \(e_\bmu W \cong e_\bla W\) as $\k$-supermodules whenever \(\bmu\) and \(\bla\) are in the same \(\Si^I_n\)-orbit. 
%It follows that characters of \(T^A_\a(n,d)\)-modules are {\em symmetric functions}---they belong naturally to the \(\Si^I_n\)-invariant subspace \((R\La^I(n,d))^{\mathfrak{S}_n^I}\) of \(R\La^I(n,d)\). 
\end{proof}

Finally, Lemma \ref{coprodchar} gives us:

\begin{Lemma}
Let $W_1$ be a $T^A_\a(n,d_1)$-supermodule with free weight spaces and $W_2$ be a $T^A_\a(n,d_2)$-supermodule with free weight spaces. We consider $W_1\otimes W_2$ as a $T^A_\a(n,d_1+d_2)$-supermodule via the coproduct $\nabla$. Then $W_1\otimes W_2$ is a supermodule with free weight spaces, and 
\begin{align*}
\ch_\pi(W_1 \otimes W_2) = \ch_\pi(W_1) \, \ch_\pi(W_2).
\end{align*}
\end{Lemma}

Let $\bla,\bmu\in\La^I(n)$. We call $\bla,\bmu$ {\em non-overlapping} if for every $i\in I$ and $r\in [1,n]$ we have that $\la_r^{(i)}\neq 0$ implies $\mu_r^{(i)}=0$. 

\begin{Proposition} %\label{}%{\rm \cite{}}%{\bf ()}
Suppose that $B$ is right $e_i$-admissible for all $i\in I$. 
Let $\bla\in\La^I(n,c)$, $\bmu\in\La^I(n,d)$ and suppose that $\bla$ and $\bsi\bmu$ are non-overlapping for some $\bsi\in\Si_n^I$. Then we have isomorphisms
\begin{align*}
S^A(n,c) e_\bla\otimes S^A(n,d)e_\bmu&\cong S^A(n,c+d)e_{\bla+\bsi\bmu},
\\
T^A_\a(n,c) e_\bla\otimes T^A_\a(n,d)e_\bmu&\cong T^A_\a(n,c+d)e_{\bla+\bsi\bmu}.
\end{align*}
of \(S^A(n,c+d)\)- and \(T^A_\a(n,c+d)\)-modules, respectively.
\end{Proposition}
\begin{proof}
We prove the result for $T^A_\a$; the proof for $S^A$ is similar. Since $T^A_\a(n,d)e_\bmu\cong T^A_\a(n,d)e_{\bsi\bmu}$ by Lemma~\ref{symchar}, we may assume that $\bla$ and $\bmu$ are non-overlapping and prove that 
$$T^A_\a(n,c) e_\bla\otimes T^A_\a(n,d)e_\bmu\cong T^A_\a(n,c+d)e_{\bla+\bmu}.$$ 

Set \(B(i):= \{b \in B \mid be_i = b\}\) for all \(i \in I\).
For \(\bnu \in \La^I(n,f)\), let \(\Seq^B_\bnu(n,f)\) be the set of all \((\bb,\br, \bs) \in \Seq^B(n,f)\) such that
\begin{align*}
\#\{k \mid b_k \in B(i), s_k = t\} = \nu^{(i)}_{t},\, \forall i \in I, t \in [1,n].
\end{align*}
Then for all \((\bb,\br,\bs) \in \Seq^B(n,f)\) we have
\begin{align*}
\eta^\bb_{\br, \bs} e_\bnu \neq 0
\iff
\eta^\bb_{\br, \bs} e_\bnu = \eta^\bb_{\br, \bs}
\iff
(\bb, \br, \bs) \in \Seq_\bnu^B(n,f),
\end{align*}
so 
$$\{ \eta^\bb_{\br, \bs} \mid [\bb, \br, \bs] \in \Seq^B_\bnu(n,f)/\Si_f\}$$ 
is a basis for \(T^A_\a(n,f)e_\bnu\). 
By the non-overlapping condition, we may choose a total order on \(B \times [1,n] \times [1,n]\) such that \((b,r,s) > (b', r', s')\) whenever \(b \in B(i)\) and \(b' \in B(j)\) for some \(i,j \in I\) with \(\la^{(i)}_s >0\) and \(\mu^{(j)}_{s'} >0\). Let \(\Seq^B_\bnu(n,f)_0 \subseteq \Seq^B_\bnu(n,f)\) be the subset of triples which are lexicographically maximal under this total order. The set \(\Seq^B_\bnu(n,f)/\Si_f\) is in bijection with \(\Seq^B_\bnu(n,f)_0\), so 
$$\{ \eta^\bb_{\br, \bs} \mid (\bb, \br, \bs) \in \Seq^B_\bnu(n,f)_0\}$$ 
is a basis for \(T^A_\a(n,f)e_\bnu\).

We have a one-to-one correspondence
\[
\Seq^B_\bla(n,c)_0 \times \Seq^B_\bmu(n,d)_0 \leftrightarrow \Seq^B_{\bla + \bmu}(n,c+d)_0, 
\]
given by
\[
((\bb,\br,\bs),(\bb',\br',\bs')) \mapsto (\bb\bb', \br \br', \bs \bs').
\]
Thus we have a \(\k\)-linear isomorphism
\begin{align*}
\phi:T^A_\a(n,c) e_\bla\otimes T^A_\a(n,d)e_\bmu \xrightarrow{\sim} T^A(n,c+d)e_{\bla+\bmu}
\end{align*}
defined via
\begin{align*}
\eta^\bb_{\br, \bs} \otimes \eta^{\bb'}_{\br', \bs'} \mapsto \eta^{\bb \bb'}_{\br \br', \bs \bs'}
\end{align*}
for all \((\bb, \br, \bs) \in \Seq^B_\bla(n,c)_0\) and \((\bb', \br', \bs') \in \Seq^B_\bmu(n,d)_0\). Moreover, in this situation \((\bb\bb', \br \br', \bs \bs')\) is \((c,d)\)-separated by the non-overlapping condition, so \(\eta^{\bb \bb'}_{\br \br', \bs \bs'} = \eta^\bb_{\br, \bs} * \eta^{\bb'}_{\br', \bs'}\) by Lemma \ref{LSingleSep}. Thus we may describe the isomorphism more generally via the star map:
\begin{align*}
\phi:T^A_\a(n,c) e_\bla\otimes T^A_\a(n,d)e_\bmu \xrightarrow{\sim} T^A_\a(n,c+d)e_{\bla+\bmu},\qquad
x \otimes y \mapsto x*y.
\end{align*}
Finally, $\phi$ is an isomorphism of $T^A_\a(n,c+d)$-modules thanks to Lemma~\ref{Lstfo}.
\end{proof}

\subsection{Idempotent truncation}
Let \(e \in \a\) be an idempotent and $\xi^e\in T^A_\a(n,d)$ be the idempotent of (\ref{E070318_2}). Set 
$$\bar{A}:=eAe\quad \text{and}\quad \bar{\a}:=e \a e.
$$
By definition, $\bar A$ is a subalgebra of $A$ and $\bar \a$ is a subalgebra of $\a$. So we can consider $S^{\bar A}(n,d)$ and hence $T^{\bar A}_{\bar\a}(n,d)$ as subalgebras of $S^A(n,d)$.

\begin{Lemma} \label{LTruncation} %{\rm \cite{}}%{\bf ()}
Let \(e \in \a\) be an idempotent. 
Suppose that $B$ is $e$-admissible. 
Then:
%Considered as subalgebras of \(S^A(n,d)\), we have 
\begin{enumerate}
\item[{\rm (i)}] \(S^{\bar A}(n,d) = \xi^e{}S^A(n,d)\xi^e\).
\item[{\rm (ii)}] \(T^{\bar A}_{\bar\a}(n,d) = \xi^e{}T^A_\a(n,d)\xi^e\).
\end{enumerate}
\end{Lemma}
\begin{proof}
By assumption, we have an 
%By Proposition \ref{AaInd}, we may assume we have chosen our 
\((A,\a)\)-basis \(B=B_\a \sqcup B_\c \sqcup B_\1\) such that \(ebe = b\) or \(ebe=0\) for all \(b \in B\). 
Defining
\begin{align*}
\bar{B}_\a&:= \{b \in B_\a \mid ebe =b\},
\\
\bar{B}_\c&:= \{b \in B_\c \mid ebe =b\},
\\
\bar{B}_\1&:=\{b \in B_\1 \mid ebe =b\},
\end{align*}
we have that \(\bar{B}:=\bar{B}_\a \sqcup \bar{B}_\c \sqcup \bar{B}_\1\) is an \((\bar{A},\bar{\a})\)-basis for \(\bar{A}\). Then, for all \((\bb, \br, \bs) \in \Seq^B(n,d)\), we have
\begin{align*}
\xi^e \xi^{\bb}_{\br, \bs} \xi^e &= 
\xi^{e b_1 e, \ldots, e b_d e}_{\br, \bs}=
\begin{cases}
\xi^{\bb}_{\br, \bs}
&
\textup{if }
\bb \in \bar{B}^d\\
0 &
\textup{otherwise},
\end{cases}
\end{align*}
which implies the result.
\end{proof}

For 
$\br\in[1,n]^d$ we define
$$
\om^\br=(\om_1,\dots,\om_n)\in \La(n,d)
$$
via $\om_r:=\{k\in[1,d]\mid r_k=r\}$ for all $r\in [1,n]$. Recall the idempotent $\xi_\la$ from (\ref{E070318_3}).

\begin{Lemma} \label{L2118} %{\rm \cite{}}%{\bf ()}
Let $A$ be unital. If $\la\in \La(n,d)$ and $(\ba,\br,\bs)\in\Seq^H(n,d)$ then 
$$
\xi_\la\xi^\ba_{\br,\bs}=\de_{\la,\om^\br}\xi^\ba_{\br,\bs}\quad \text{and}\quad 
\xi^\ba_{\br,\bs}\xi_\la=\de_{\la,\om^\bs}\xi^\ba_{\br,\bs}.
$$
\end{Lemma}
\begin{proof}
Immediate from Proposition~\ref{CPR}. 
\end{proof}

%\subsection{Idempotent truncations}\label{SSSHI}

Let $N\geq n$. Set
$$
\La^N_n(d):=\{\la\in\La(N,d)\mid \la_{n+1}=\dots=\la_N=0\}\subseteq \La(N,d),
$$
and define the idempotent
\begin{equation}\label{E130818}
\xi^N_n(d):=\sum_{\la\in\La^N_n(d)}\xi_\la\in S^A(N,d). 
\end{equation}
If the $(A,\a)$ is unital, then $\xi^N_n(d)\in T^A_\a(N,d)$.

\begin{Lemma} \label{LIdEasy} %{\rm \cite{}}%{\bf ()}
Let $A$ be unital, $N\geq n$ and $(\bb,\br,\bs)\in\Seq^B(N,d)$. 
\begin{enumerate}
\item[{\rm (i)}] We have 
$$
\xi^{N}_n(d)\xi^\bb_{\br,\bs}=\de_{\om^\br\in\La^N_n(d)}\xi^\bb_{\br,\bs}\quad \text{and}\quad  
\xi^\bb_{\br,\bs}\xi^{N}_n(d)=\de_{\om^\bs\in\La^N_n(d)}\xi^\bb_{\br,\bs}.
$$
In particular, the map 
$$
S^A(n,d)\to S^A(N,d),\ \xi^\bb_{\br,\bs}\mapsto \xi^\bb_{\br,\bs}\qquad\big((\bb,\br,\bs)\in\Seq^B(n,d)\big)
$$
is a (unital) algebra isomorphism 
$$S^A(n,d)\iso \xi^{N}_n(d)S^A(N,d)\xi^{N}_n(d).
$$
\item[{\rm (ii)}] 
If $(A,\a)$ is a unital good pair then $\xi^{N}_n(d)\in T^A_\a(N,d)$, 
$$
\xi^{N}_n(d)\eta^\bb_{\br,\bs}=\de_{\om^\br\in\La^N_n(d)}\eta^\bb_{\br,\bs}\quad \text{and}\quad  
\eta^\bb_{\br,\bs}\xi^{N}_n(d)=\de_{\om^\bs\in\La^N_n(d)}\eta^\bb_{\br,\bs}.
$$
In particular, the map 
$$
T^A_\a(n,d)\to T^A_\a(N,d),\ \eta^\bb_{\br,\bs}\mapsto \eta^\bb_{\br,\bs}\qquad\big((\bb,\br,\bs)\in\Seq^B(n,d)\big)
$$
is a (unital) algebra isomorphism 
$$T^A_\a(n,d)\iso \xi^{N}_n(d)T^A_\a(N,d)\xi^{N}_n(d).
$$
\end{enumerate} 
\end{Lemma}
\begin{proof}
Follows from Lemma~\ref{L2118}.  
\end{proof}

\begin{Corollary} \label{CEquivalence} %{\rm \cite{}}%{\bf ()}
If $d\leq n\leq N$, then $V\mapsto \xi^N_n(d) V$ defines equivalences of categories $$\mod{S^A(N,d)}\iso \mod{S^A(n,d)}\quad  \text{and} \quad \mod{T^A_\a(N,d)}\iso \mod{T^A_\a(n,d)}.$$ 
\end{Corollary}
\begin{proof}
To prove the result for $S^A$, in view of Lemma~\ref{LIdEasy}, we just have to prove that $$S^A(N,d)\xi^N_n(d)S^A(N,d)=S^A(N,d).$$ The last equality will follow if we can show that each $\xi_\la$ with $\la\in\La(N,d)$ is in the left hand side. By the assumption that $d\leq n$, there is $\si\in\Si_n$ such that all non-zero parts of $\si\la$ are among its first $n$ parts, and so 
$$
\xi_{\si\la}=\xi_{\si\la}\xi^N_n(d)\in S^A(N,d)\xi^N_n(d)S^A(N,d).$$ 
By Lemma~\ref{symchar}, we have that $\xi_\si\xi_\la\xi_\si^{-1}=\xi_{\si\la}$, or 
$$\xi_\la=\xi_\si^{-1}\xi_{\si\la}\xi_\si\in S^A(N,d)\xi^N_n(d)S^A(N,d),$$ and we are done. 
The proof for $T^A_\a$ is the same, using the fact that $\xi_\si\in T^A_\a(n,d)$. 
\end{proof}

\begin{Remark} %\label{}%{\rm \cite{}}%{\bf ()}
{\rm 
Let $d\leq n$ and $\om:=(1,\dots,1,0,\dots,0)\in \La(n,d)$. It is proved in \cite[Lemma 5.15]{EK1} that the idempotent truncation $\xi_\om S^A(n,d)\xi_\om$ is naturally isomorphic to the wreath product superalgebra $A\wr \Si_d$. If the pair $(A,\a)$ is unital, we have $\xi_\om\in T^A_\a(n,d)$ and it is easy to see that $\xi_\om T^A_\a(n,d)\xi_\om=\xi_\om S^A(n,d)\xi_\om$. 
}
\end{Remark}

\subsection{Tensor product, truncation and induction}
In this subsection, we drop indices and write $T(n,d)$ for $T^A_\a(n,d)$ and $S(n,d)$ for $S^A(n,d)$. 
Throughout the subsection, we fix $a\in\Z_{\geq 1}$, a composition $\de=(d_1,\dots,d_a)\in\La(a,d)$, and a composition $\nu=(n_1,\dots,n_a)\in\La(a,n)$ with $n_1,\dots,n_a>0$.  We denote 
$$
T(n,\de):=T(n,d_1)\otimes\dots\otimes T(n,d_a)\quad\text{and}\quad
T(\nu,\de):=T(n_1,d_1)\otimes\dots\otimes T(n_a,d_a).
$$
Let 
$$\nabla^{(a)}:=(\id^{\otimes a-2}\otimes \nabla)\circ \dots \circ(\id\otimes \nabla)\circ\nabla\,:\,T(n,d)\to \bigoplus_{\ga\in\La(a,d)}T(n,\ga)$$
be the iterated coproduct. Projecting onto the summand $T(n,\de)$ yields the algebra homomorphism 
$$
\nabla_\de:T\to T(n,\de).
$$
Using $\nabla_\de$, we can consider $T(n,\de)$ as a $(T(n,d),T(n,\de))$-bimodule, so that 
\begin{equation}\label{E090818_3}
U_1\otimes\dots\otimes U_a\cong 
T(n,\de)
\otimes_{T(n,\de)}(U_1\boxtimes\dots\boxtimes U_a)
\end{equation}
for $U_1\in\mod{T(n,d_1)},\,\dots,\, U_a\in\mod{T(n,d_a)}$.

Recall the idempotent 
$
\xi^N_n(d)\in T(N,d)
$
from (\ref{E130818}). 
The first result relates tensor product and truncation. 

\begin{Proposition}%\label{}%{\rm \cite{}}%{\bf ()}
Let $n\leq N$ and $V_k\in\mod{T(N,d_k)}$ for $k=1,\dots,a$. Then there is a functorial  isomorphism of $T(n,d)$-modules
$$
\xi^N_n(d)(V_1\otimes\dots\otimes  V_a)\cong (\xi^N_n(d_1) V_1)\otimes \dots\otimes (\xi^N_n(d_a)V_a).
$$
Similar statement holds for $S$ in place of $T$. 
\end{Proposition}
\begin{proof}
%We prove the result for $T^A_\a$; the proof for $S^A$ is similar. 
Note using Lemma~\ref{coprodchar} that 
\begin{align*}
\nabla_{\de}(\xi^N_n(d))= \sum_{\la\in\La^N_n(d)}\nabla_{\de}(\xi_\la)
%= \sum_{\la\in\La^N_n(d)} \sum_{\substack{\mu_1,\dots,\mu_a \in \La(N)%\\ \nu \in \La(N) 
%\\ \mu_1 + \dots+\mu_a = \la}} \xi_{\mu_1} \otimes\dots\otimes \xi_{\mu_a}\\
&= 
\sum_{ \mu_1 \in \La^N_n(d_1),\dots, \mu_a \in \La^N_n(d_a)} \xi_{\mu_1} \otimes\dots\otimes  \xi_{\mu_a}
\\
&=\xi^N_n(d_1)\otimes\dots\otimes  \xi^N_n(d_a).
\end{align*}
Therefore
$$
\xi^N_n(d)(V_1\otimes\dots\otimes  V_a)= (\xi^N_n(d_1) V_1)\otimes\dots\otimes  (\xi^N_n(d_a)V_a),
$$
and the result follows.
\end{proof}

In the rest of this subsection, we concentrate on $T(n,d)$, although similar results hold for $S(n,d)$. 
We now define certain induction operation and relate it to tensor product. Set
$$
m_k:=\sum_{r=1}^{k-1}n_r\qquad (k=1,\dots,a+1).
$$

Denote
$$
\La(\nu;\de)=\{\la\in\La(n,d)\mid \sum_{r=m_k+1}^{m_{k+1}}\la_{r}=d_k\ \text{for all $k=1,\dots,a$}\},
$$
and define the idempotent
$$
\xi(\nu;\de):=\sum_{\la\in\La(\nu;\de)}\xi_\la\in T(n,d).
$$

For $\br=r_1\cdots r_t\in\Z^t$ and $m\in\Z_{\geq 0}$, we define
$$
\br(+m):=(r_1+m)\cdots (r_t+m)\in\Z^t.
$$
Now let $\br^k\in [1,n_k]^{d_k}$ for $k=1,\dots,a$, and $\br:=\br^1\cdots\br^a\in[1,n]^d$. We define
$$
\br(+\nu):=\br^1(+m_1)\br^2(+m_2)\cdots\br^a(+m_a)\in [1,n]^{d}.
$$
If $(\bb^k,\br^k,\bs^k)\in\Seq^B(n_k,d_k)$ for $k=1,\dots,a$, and $\bb:=\bb^1\cdots\bb^a$, $\br:=\br^1\cdots\br^a$, $\bs:=\bs^1\cdots\bs^a$, then  $(\bb,\br(+\nu),\bs(+\nu))\in\Seq^B(n,d)$ is $\de$-separated, and so by Lemma~\ref{LSingleSep}, we have
\begin{equation}\label{E090818}
\eta^\bb_{\br(+\nu),\bs(+\nu)}=\eta^{\bb^1}_{\br(+m_1),\bs(+m_1)}
*\eta^{\bb^2}_{\br(+m_2),\bs(+m_2)}*\cdots*\eta^{\bb^a}_{\br(+m_a),\bs(+m_a)}.
\end{equation}
Similarly, $(\bb,\br,\bs(+\nu))\in\Seq^B(n,d)$ is $\de$-separated, and 
\begin{equation}\label{E090818_2}
\eta^\bb_{\br,\bs(+\nu)}=\eta^{\bb^1}_{\br,\bs(+m_1)}
*\eta^{\bb^2}_{\br,\bs(+m_2)}*\cdots*\eta^{\bb^a}_{\br,\bs(+m_a)}.
\end{equation}

\begin{Lemma} \label{TNuDeHom}%{\rm \cite{}}%{\bf ()}
The map
$$
T(\nu,\de)\to T(n,d),\ \eta^{\bb^1}_{\br^1,\bs^1}\otimes\dots\otimes \eta^{\bb^a}_{\br^a,\bs^a}\mapsto 
\eta^\bb_{\br(+\nu),\bs(+\nu)}
%\eta^{\bb}_{\br,\bs}*\eta^{\bc}_{\bt^{+n_1},\bu^{+n_1}}
$$
is an algebra homomorphism, mapping the identity element  of $T(\nu,\de)$  onto $\xi(\nu;\de)$. 
\end{Lemma}
\begin{proof}
This follows easily from (\ref{E090818}) and Lemma~\ref{Lstfo}. 
\end{proof}

In view of the lemma, we consider $$T(n,d)\xi(\nu,\de)$$ as a $(T(n,d),T(\nu,\de))$-bimodule. Given a $T(\nu,\de)$-module $V$, we now define
$$
I^{n,d}_{\nu,\de}V:=
T(n,d)\xi(\nu;\de)\otimes_{T(\nu,\de)} V.
$$
This yields the functor 
$$
I^{n,d}_{\nu,\de}:\mod{T(\nu,\de)}\to\mod{T(n,d)}.
$$
%We also have the functor
%\begin{align*}
%\mod{T(n_1,d_1)}\times\, \mod{T(n_2,d_2)}\to\mod{T(n,d)},\ (V_1,V_2)\mapsto I^{n,d}_{n_1,n_2;d_1,d_2}(V_1\boxtimes V_2),
%\end{align*}
%which we again denote $I^{n,d}_{n_1,n_2;d_1,d_2}$. 
The following proposition generalizes \cite[2.7]{BKlr}.

\begin{Proposition}%\label{}%{\rm \cite{}}%{\bf ()}
Suppose that  for all $k=1,\dots,a$ we have $d_k\leq n_k$ and let $V_k\in \mod{T(n,d_k)}$. Then we have a functorial isomorphism 
$$
V_1\otimes \dots\otimes V_a\cong 
I^{n,d}_{\nu;\de}\left(
(\xi^n_{n_1}V_1)\boxtimes\dots\boxtimes  (\xi^n_{n_a}V_a)\right).
$$
\end{Proposition}
\begin{proof}
In this proof $k$ always runs throgh $\{1,\dots,a\}$. 
Denote $T:=T(n,d)$, $T_k:=T(n,d_k)$, $T'_k:=T(n_k,d_k)$, so that $T(n,\de)=T_1\otimes \dots\otimes T_a$ and $T(\nu,\de)=T_1'\otimes \dots\otimes T_a'$.

Since $V_k\mapsto \xi^n_{n_k}V_k$ is an equivalence by Corollary~\ref{CEquivalence}, denoting $W_k:=\xi^n_{n_k} V_k$, we have $V_k\cong T_k\xi^N_{n_k} \otimes_{T_k'}W_k$, and it suffices to prove 
\begin{equation}\label{E080818}
(T_1\xi^n_{n_1} \otimes_{T_1'}W_1)\otimes\dots\otimes  (T_a\xi^n_{n_a} \otimes_{T_a'}W_a)\cong 
I^{n,d}_{\nu;\de}\left(
W_1\boxtimes\dots\boxtimes W_a\right).
\end{equation}
%Note that given $U_k\in\mod{T_k}$, the $T$-module $U_1\otimes \dots\otimes U_a$ is obtained by tensoring the $T(n,\nu)$-module $U_1\boxtimes\dots\boxtimes  U_a$ with the bimodule $M:=T_1\otimes\dots\otimes  T_a$, which is the right regular module over $T(n,\nu)$ and the left $T$-module structure comes from $\nabla_{\de}$. 
We now apply (\ref{E090818_3}) with $U_k:=T_k\xi^n_{n_k} \otimes_{T_k'}W_k$ to see that the left hand side of (\ref{E080818}) is obtained from $W_1\boxtimes \dots\boxtimes W_a$ by tensoring with the $(T,T(\nu,\de))$-bimodule
$$
M':=M\otimes_{T(n,\de)}(T_1\xi^n_{n_1} \otimes\dots\otimes  T_a\xi^n_{n_a})\cong T_1\xi^n_{n_1} \otimes\dots\otimes  T_a\xi^n_{n_a}.
$$
On the other hand, the right hand side of (\ref{E080818}) is obtained from $W_1\boxtimes \dots\boxtimes W_a$ by tensoring with the $(T,T(\nu,\de))$-bimodule $T\xi(\nu;\de)$. So we just need to prove that the $(T,T(\nu,\de))$-bimodules $M'$ and $T\xi(\nu;\de)$ are isomorphic.

Define
\begin{align*}
\Seq^B((n,n_k),d_k):= \{ (\bb, \br, \bs) \in \Seq^B(n,d_k) \mid \bs \in [1,n_k]^{d_k}\}.
\end{align*}
Then, for all \((\bb,\br,\bs) \in \Seq^B(n,d_k)\), we have that 
\begin{align*}
\eta^{\bb}_{\br,\bs} \xi^n_{n_k} = 
\begin{cases}
\eta^{\bb}_{\br,\bs} & \textup{if } (\bb,\br,\bs) \in \Seq^B((n,n_k),d_k) \\
0 & \textup{otherwise}.
\end{cases}
\end{align*}
Therefore
\(
\{\eta^\bb_{\br,\bs} \mid [\bb,\br,\bs] \in \Seq^B((n,n_k),d_k)/\Si_{d_k} \}
\)
is a basis for \(T_k\xi_{n_k}^n\), and we may define a \(\k\)-linear map
\begin{align*}
\varphi: M' \to T\xi(\nu;\de),
\qquad
\eta^{\bb^1}_{\br^1, \bs^1} \otimes \cdots \otimes \eta^{\bb^a}_{\br^a, \bs^a}
\mapsto
\eta^{\bb}_{\br, \bs(+\nu)}\xi(\nu;\de) = \eta^{\bb}_{\br, \bs(+\nu)},
\end{align*}
where \([\bb^k,\br^k,\bs^k] \in \Seq^B((n,n_k),d_k)/\Si_{d_k}\) for all \(k\), \(\bb = \bb^1 \cdots \bb^a\), \(\br= \br^1 \cdots \br^a\), and \(\bs=\bs^1 \cdots \bs^a\). It follows from (\ref{E090818_2}) and Lemmas~\ref{Lstfo} and ~\ref{TNuDeHom}  that \(\varphi\) is a map of \((T,T(\nu,\de))\)-bimodules, and it remains to prove that $\phi$ is an isomorphism. 

For \(\bs \in [1,n]^d\), we define \(\beta(\bs) \in \La(n,d)\) via
$
\beta(\bs)_t:=\#\{u \in [1,d] \mid s_u = t\},
$
for all \(t \in [1,n]\).
Then for \((\bb,\br,\bs) \in \Seq^B(n,d)\) we have 
\begin{align*}
\eta^\bb_{\br, \bs} \xi(\nu; \de) = 
\begin{cases}
\eta^\bb_{\br, \bs} & \textup{if } \beta(\bs) \in \La(\nu; \de);\\
0 & \textup{otherwise}.
\end{cases}
\end{align*}
Thus, setting
\begin{align*}
\Seq^B(\nu;\de):= \{ (\bb, \br, \bs) \in \Seq^B(n,d) \mid \beta(\bs) \in \La(\nu;\de)\},
\end{align*}
we have that
\(
\{ \eta^\bb_{\br, \bs} \mid [\bb, \br, \bs] \in \Seq^B(\nu;\de)/\Si_d\}
\)
is a basis for \(T \xi(\nu; \de)\). It is straightforward to check that the map
\begin{align*}
\prod_{k=1}^a (\Seq^B((n,n_k),d_k)/\Si_{d_k}) &\to \Seq^B(\nu;\de)/\Si_d,\\
([\bb^1,\br^1,\bs^1], \ldots, [\bb^a,\br^a,\bs^a]) &\mapsto [\bb,\br,\bs(+\nu)],
\end{align*}
where \(\bb = \bb^1 \cdots \bb^a\), \(\br =\br^1 \cdots \br^a\), \(\bs=\bs^1 \cdots \bs^a\), is a well-defined bijection. Therefore the \(\varphi\) restricts to a bijection (up to signs) of bases, and so \(\varphi\) is an isomorphism.
\end{proof}

\subsection{Examples}\label{SSExamples} We finish this section with some examples.

\subsubsection{Schur superalgebras.}  Let \(A=M_{p|q}(\k)\). For \(r,s \in [1,p+q]\), let \(E_{r,s}\) be the matrix with \(1\) in the \((r,s)\)-th component, and zeros elsewhere. We have
\begin{align*}
\overline{E}_{r,s} := 
\begin{cases}
\bar 0 & \textup{if }r,s \leq p \textup{ or } r,s >p,\\
\bar 1 & \textup{otherwise}.
\end{cases}
\end{align*}
%Then \(B:=\{E_{r,s} \mid r,s \in [1,n+m]\}\) is a homogeneous basis for \(M_{n|m}(\k)\).
It follows from \cite[Theorem 2]{MZ} that 
the Schur superalgebra $S(p|q,d)$ is not quasi-hereditary (over $\k$) unless $q=0$ or $p=q=d=1$. 
Choosing $\a:=\spa({E}_{r,s}\mid r,s \leq p)$ we get a non-unital good pair $(A,\a)$ and the corresponding non-unital generalized Schur superalgebra  $T^A_\a(n,d)$. We will prove in \cite{KM} that $T^A_\a(n,d)$ is quasihereditary if $d\leq n$.

\subsubsection{Trivial extension algebras}\label{SSTE}
Let $C$ be a unital superalgebra which is free of finite rank as a $\k$-supermodule. The dual $C^*:=\Hom_\k (V,\k)$ %A
is a $\k$-supermodule in a natural way. 
We have the pairing $\langle \cdot, \cdot \rangle$ between $C$ and $C^*$ with 
$
\lan  a,\al\ran=\lan \al, a\ran :=\al(a)
$
for $a\in C$ and $\al\in C^*$. 
We consider $C^*$ as a $C$-bimodule with respect to the dual regular actions  given by
\begin{equation*}\label{EActions}
\langle \al\cdot a,b\rangle=\langle \al,a b\rangle,\ \langle b,a\cdot \al\rangle=\langle ba,\al\rangle\qquad(a,b\in C,\ \al\in C^*).
\end{equation*}
%We refer to this bimodule as the {\em dual regular superbimodule}. 
The {\em trivial extension superalgebra} $\TE(C)$ of $C$ is 
$\TE(C)=C\oplus C^*$ as a $\k$-supermodule, with multiplication 
\begin{equation*}\label{ETrivInt}
(a,\al)(b,\be)=(ab,a\cdot \be+\al\cdot b)\qquad (a,b\in C,\ \al,\be\in C^*).
\end{equation*}
Note that $C_\0$ is a unital subalgebra of  $\TE(C)_\0$. The pair $(\TE(C),C_\0)$ is an example of a unital good pair $(A,\a)$. In this case it is natural to take $\c=C^*_0$. The basis $B_\a$ is a basis of $C_\0$, the basis $B_\1$ is a basis of $\TE(C)_\1=C_\1\oplus C_\1^*$, and the basis $B_\c$ is a basis of $C_\0^*$. 
%The basis elements $\eta^\bb_{\br,\bs}=[\bb,\br,  \bs]^!_{\c}\, \xi^\bb_{\br,  \bs}$ of the Schur algebra $T^{\TE(C)}_{C_\0}(n,d)$ 

Let $n\in\Z_{>0}$. 
For $\al\in C^*$ and $1\leq r,s\leq n$, we have the element $x_{r,s}^{\al}\in M_n(C)^*$ defined from
\begin{equation*}\label{EXRSAl}
\langle x_{r,s}^{\al},\xi_{t,u}^a\rangle=\de_{r,t}\de_{s,u}\langle \al,a\rangle\qquad (1\leq t,u\leq n,\ a\in C).
\end{equation*}
It is pointed out in \cite[Lemma 3.21]{EK1} that there is an isomorphism of superalgebras
\begin{equation}%\label{EIso}
M_n(\TE(C))\iso \TE(M_n(C)),\ \xi_{r,s}^{(a,\al)}\mapsto (\xi_{r,s}^a, x_{s,r}^\al).
\end{equation}

By \cite[Theorem 4.27]{EK1}, we have an explicit isomorphism 
\begin{equation}\label{EPrimIso}
S^{\TE(C)}(n,d)\cong {}'D^C(n,d),
\end{equation}
where ${}'D^C(n,d)$ is the {\em divided power Turner's double algebra} defined in \cite[\S\S4,6]{EK1}. It now follows from Theorem~\ref{TGen} and \cite[Theorem 4.31]{EK1} that under the isomorphism (\ref{EPrimIso}) the subalgebra $T^{\TE(C)}_{C_\0}(n,d)\subseteq S^{\TE(C)}(n,d)$ gets identified with the {\em Turner double} subalgebra $D^C(n,d)\subseteq {}'D^C(n,d)$ of \cite[\S\S4,6]{EK1}:
\begin{equation}\label{EIso}
T^{\TE(C)}_{C_\0}(n,d)\cong D^C(n,d).
\end{equation}
%for all $1\leq r,s\leq n$, $a\in C$ and $\al\in C^*$. 

\subsubsection{Zigzag algebras}\label{SSSZ}
Fix $\ell\geq 1$ and $\Gamma$ be the quiver with vertex set $I:=\{0,1,\dots,\ell\}$ and arrows $\{a_{j,j-1},a_{j-1,j}\mid j=1,\dots,\ell\}$ as in the picture: 
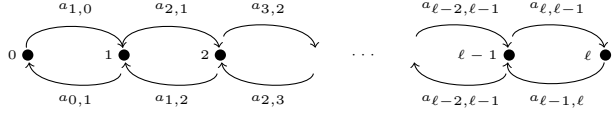
\begin{align*}
\begin{braid}\tikzset{baseline=3mm}
\coordinate (0) at (-4,0);
\coordinate (1) at (0,0);
\coordinate (2) at (4,0);
\coordinate (3) at (8,0);
%\coordinate (4) at (12,0);
\coordinate (6) at (12,0);
\coordinate (L1) at (16,0);
\coordinate (L) at (20,0);
\draw [thin, black,->,shorten <= 0.1cm, shorten >= 0.1cm]   (0) to[distance=1.5cm,out=100, in=100] (1);
\draw [thin,black,->,shorten <= 0.25cm, shorten >= 0.1cm]   (1) to[distance=1.5cm,out=-100, in=-80] (0);
\draw [thin, black,->,shorten <= 0.1cm, shorten >= 0.1cm]   (1) to[distance=1.5cm,out=100, in=100] (2);
\draw [thin,black,->,shorten <= 0.25cm, shorten >= 0.1cm]   (2) to[distance=1.5cm,out=-100, in=-80] (1);
\draw [thin,black,->,shorten <= 0.25cm, shorten >= 0.1cm]   (2) to[distance=1.5cm,out=80, in=100] (3);
\draw [thin,black,->,shorten <= 0.25cm, shorten >= 0.1cm]   (3) to[distance=1.5cm,out=-100, in=-80] (2);
%\draw [thin,black,->,shorten <= 0.25cm, shorten >= 0.1cm]   (3) to[distance=1.5cm,out=80, in=100] (4);
%\draw [thin,black,->,shorten <= 0.25cm, shorten >= 0.1cm]   (4) to[distance=1.5cm,out=-100, in=-80] (3);
\draw [thin,black,->,shorten <= 0.25cm, shorten >= 0.1cm]   (6) to[distance=1.5cm,out=80, in=100] (L1);
\draw [thin,black,->,shorten <= 0.25cm, shorten >= 0.1cm]   (L1) to[distance=1.5cm,out=-100, in=-80] (6);
\draw [thin,black,->,shorten <= 0.25cm, shorten >= 0.1cm]   (L1) to[distance=1.5cm,out=80, in=100] (L);
\draw [thin,black,->,shorten <= 0.1cm, shorten >= 0.1cm]   (L) to[distance=1.5cm,out=-100, in=-100] (L1);
\blackdot(-4,0);
\blackdot(0,0);
\blackdot(4,0);
%\blackdot(8,0);
\blackdot(16,0);
\blackdot(20,0);
\draw(-4,0) node[left]{$0$};
\draw(0,0) node[left]{$1$};
\draw(4,0) node[left]{$2$};
%\draw(8,0) node[left]{$3$};
\draw(10,0) node {$\cdots$};
\draw(13.4,0) node[right]{$\ell-1$};
\draw(18.65,0) node[right]{$\ell$};
\draw(-2,1.2) node[above]{$ a_{1,0}$};
\draw(2,1.2) node[above]{$ a_{2,1}$};
\draw(6,1.2) node[above]{$ a_{3,2}$};
%\draw(10,1.2) node[above]{$ a_{4,3}$};
\draw(14,1.2) node[above]{$ a_{\ell-2,\ell-1}$};
\draw(18,1.2) node[above]{$ a_{\ell,\ell-1}$};
\draw(-2,-1.2) node[below]{$ a_{0,1}$};
\draw(2,-1.2) node[below]{$ a_{1,2}$};
\draw(6,-1.2) node[below]{$ a_{2,3}$};
%\draw(10,-1.2) node[below]{$ a_{3,4}$};
\draw(14,-1.2) node[below]{$ a_{\ell-2,\ell-1}$};
\draw(18,-1.2) node[below]{$ a_{\ell-1,\ell}$};
\end{braid}
\end{align*}
The {\em extended zigzag algebra $\EZig$} is the path algebra $\k\Gamma$ modulo the following relations:
\begin{enumerate}
\item All paths of length three or greater are zero.
\item All paths of length two that are not cycles are zero.
\item All length-two cycles based at the same vertex are equivalent.
\item $ a_{\ell,\ell-1} a_{\ell-1,\ell}=0$.
\end{enumerate}

Length zero paths yield the idempotents $\{ e_0,\dots,e_\ell\}$ with $ e_i  a_{i,j} e_j= a_{i,j}$ for all admissible $i,j$. The algebra $\EZig$ is graded by the path length: 
$\EZig=\EZig^0\oplus \EZig^1\oplus \EZig^2.
$ 
We consider $\EZig$ as a superalgebra with 
$\EZig_\0=\EZig^0\oplus \EZig^2\ \text{and}\ \EZig_\1=\EZig^1.
$ 
Let $\z:=\spa(e_0,\dots,e_\ell)$. 
Define also $J:=\{0,1,\dots,\ell-1\}$ and for all $j\in J$, set 
$
 c_j:= a_{j,j+1} a_{j+1,j}.
$
We have a basis $B=B_\z\sqcup B_\c\sqcup B_\1$ of $Z$ as in (\ref{AaBasis}), with 
$$B_\1=\{ a_{j,j+1}, a_{j+1,j}\mid j\in J\},\ \ B_{\z}=\{ e_i\mid i\in I\},\ \ B_{\c}:=\{ c_j\mid j\in J\}.
$$ 

Let \(e:= e_0 + \cdots + e_{\ell-1} \in Z\). The {\em zigzag algebra} is $\overline{Z}:=e Z e \subset Z$. We also have  $\bar \z:=e\z e=\spa(e_0,\dots,e_{\ell-1})$. For $n\geq d$, the generalized Schur algebra $T^{\bar Z}_{\bar\z}(n,d)$ in this case is Morita equivalent to weight $d$ RoCK blocks of symmetric groups (and the corresponding Hecke algebras), as conjectured by Turner \cite{T} and proved in \cite{EK2}. 
This was our motivating example. 

In \cite{KM} we construct an explicit cellular basis of $T^{\bar Z}_{\bar\z}(n,d)$, while no such basis is known for $S^{\bar Z}(n,d)$ (and it probably does not exist in general). 
Moreover, in \cite{KM} we prove that $T^Z_\z(n,d)$ is quasi-hereditary, while $S^Z(n,d)$ in general is not.

\section{Symmetricity}\label{Ssym}
\subsection{Central and symmetrizing forms}
Assume in this section that \(A\) has finite rank as a \(\k\)-supermodule. We say an even \(\k\)-linear map \(\tr:A\to \k\) is a {\em central form} for \(A\) provided \(\tr(ab)=\tr(ba)\) for all \(a,b \in A\). In this case we have an associative symmetric bilinear form $(\cdot,\cdot)_\tr$ with 
$
(a,b)_\tr=\tr(ab)
$
and a homomorphism 
$$A \to A^*:=\Hom_\k(A,\k),\ a\mapsto (a,\cdot)_\tr$$  
of \((A,A)\)-superbimodules. We say that the central form \(\tr\) is a {\em symmetrizing form} if this homomorphism is an isomorphism, i.e. if $(\cdot,\cdot)_\tr$ is a perfect pairing. 

If \(A\) is equipped with a symmetrizing form, we say \(A\) is {\em symmetric.} We want to show  that if \(A\) is symmetric 
then so is \(T^A_\a(n,d)\). We will do this under some natural assumptions. The following lemma is easily checked.
%and \(\tr\) satisfies a certain additional assumption, then the algebra \(T^A_\a(n,d)\) is symmetric as well.

\begin{Lemma}\label{trM}
Assume that $\tr$ is a central form on $A$. Then the algebra \(M_n(A)^{\otimes d}\) has a central form \(\tr^M: M_n(A)^{\otimes d} \to \k\) given by 
\begin{align*}
\tr^M(\xi^{a_1}_{r_1,s_1} \otimes \cdots \otimes \xi^{a_d}_{r_d,s_d}) = \delta_{\br, \bs}\tr(x_1) \cdots \tr(a_d),
\end{align*}
for all \((\ba, \br, \bs)  \in A^d \times [1,n]^d \times [1,n]^d\). 
\end{Lemma}

\begin{Lemma}\label{L150518}
Assume that $\tr$ is a central form on $A$. Then the algebra \(S^A(n,d)\) has a central form \(\tr^S: S^A(n,d) \to \k\) given by 
\begin{align*}
\tr^S(\xi^{\ba}_{\br, \bs}):=\frac{d!}{[\ba, \br, \bs]^!} \delta_{\br, \bs} \tr(a_1) \cdots \tr(a_d).
\end{align*}
for all \((\ba, \br, \bs)  \in \Seq^H(n,d)\). Moreover, $\tr^M|_{S^A(n,d)}=\tr^S$. 
\end{Lemma}
\begin{proof}
Let \((\ba, \br, \bs) \in \Seq^H(n,d)\).
Since \(\tr(x) = 0\) whenever \(x \in A_\1\), we have 
\begin{align*}
\tr^M((\xi^{a_1}_{r_1,s_1} \otimes \cdots \otimes \xi^{a_d}_{r_d,s_d})^\sigma)= \tr^M(\xi^{a_1}_{r_1,s_1} \otimes \cdots \otimes \xi^{a_d}_{r_d,s_d}),
\end{align*}
for all \(\sigma \in \Si_d\).

Then for all \((\ba,\br, \bs) \in \Seq^H(n,d)\), we have
\begin{align*}
\tr^M(\xi^\ba_{\br, \bs}) &= \tr^M\left( \sum_{\sigma \in {}^{\bx, \br, \bs}\D} (
\xi^{a_1}_{r_1,s_1} \otimes \cdots \otimes \xi^{a_d}_{r_d,s_d}
)^\sigma\right)\\
&=  \sum_{\sigma \in {}^{\bx, \br, \bs}\D} \tr^M((\xi^{a_1}_{r_1,s_1} \otimes \cdots \otimes \xi^{a_d}_{r_d,s_d})^\sigma)\\
&=\sum_{\sigma \in {}^{\bx, \br, \bs}\D} \tr^M(\xi^{a_1}_{r_1,s_1} \otimes \cdots \otimes \xi^{a_d}_{r_d,s_d})\\
&=
|{}^{\ba, \br, \bs}\D|\tr^M(\xi^{a_1}_{r_1,s_1} \otimes \cdots \otimes \xi^{a_d}_{r_d,s_d})\\
&=
 \frac{d!}{[\ba,\br,\bs]^!}\delta_{\br, \bs} \tr(x_1) \cdots \tr(x_d),
\end{align*}
giving the result.
\end{proof}

If $k$ is a field of characteristic $0$ or greater than $d$, one can check that the from $\tr^S$ is  symmetrizing. But this is certainly false over fields of positive characteristics less than $d$. In fact, for such fields even the classical Schur algera $S^\k(n,d)$ is not symmetric. 

\subsection{Symmetricity of $T^A_\a(n,d)$}
Throughout the subsection we assume that \((A,\a)\) is a unital good pair, and that \(A\) is symmetric, with symmetrizing form \(\tr\). We say that $\tr$ is {\em $(A,\a)$-symmetrizing} if $(\a,\a)_\tr=0$ and the $\k$-complement $\c$ of $\a$ in $A_\0$ can be chosen so that the restriction of $(\cdot,\cdot)_\tr$ to $\a\times\c$ is a perfect pairing. 
%We identify $\a^*$ with the $\k$-submodule of $A^*$ consisting of all $\k$-linear maps which are zero on $\c\oplus A_\1$. We similarly identify $\c^*$ and $A_1^*$ as $\k$-submodules of $A^*$. Then we can write  \(A^* = \a^* \oplus \c^* \oplus A_\1^*\). We say that $\tr$ is {\em $(A,\a)$-symmetrizing} if \(\hat{\tr}\) maps \(\a\subset A\) isomorphically onto \(\c^* \subset A^*\). 
If $\tr$ is an {\em $(A,\a)$-symmetrizing} form, we always assume that the complement $\c$ has this property. 

Let $\tr$ be an {\em $(A,\a)$-symmetrizing} form. 
For $a\in\a$, we have $\tr(a)=(a,1)_\tr=0$ so $\tr(\a)=0$. 
There exists an \((A,\a)\)-basis \(B=B_\a \cup B_\c \cup B_\1\) of \(A\) such that the dual basis \(B^*= \{b^* \mid b \in B\}\) of $A$  with respect to $(\cdot,\cdot)_\tr$ satisfies the following property:  setting 
\begin{align*}
B_\a^*:=\{b^* \mid b \in B_\a\},
\quad
B_\c^*:=\{b^* \mid b \in B_\c\},
\quad
B_\1^*:=\{b^* \mid b \in B_\1\}.
\end{align*}
we have that \(B_\a^*=B_\c\), \(B_\c^*=B_\a\), and \(B_\1^*\) is a basis for \(A_\1\). Then \(B^*:=B_\c^* \cup B_\a^* \cup B_{\1}^*\) is an \((A,\a)\)-basis as well. If $\tr$ is an {\em $(A,\a)$-symmetrizing} form, we always assume that $B$ has been chosen to satisfy these properties. 

%For \((\bx, \br, \bs) \in A^d \times [1,n]^d \times [1,n]^d\), define \begin{align*} [\bx, \br, \bs]^!_{A\backslash\a}:= \frac{[\bx, \br, \bs]^!}{[\bx,\br, \bs]_\a^!},\end{align*} and write \begin{align*} \eta^\bx_{\br, \bs}:= [\bx, \br, \bs]^!_{A\backslash \a} \cdot \xi^\bx_{\br, \bs}. \end{align*} Since \([\bx, \br, \bs]^!_{A\backslash \a} = [\bx, \br, \bs]^!_{\c}\) whenever \(\bx \in B^d\), this generalizes our previous definition of \(\eta\)'s, and we have by Lemma \ref{TGen} that \begin{align*} T_\a^A(n,d) = \spa\{\eta^\bx_{\br, \bs} \mid (\bx, \br, \bs) \in A^d \times [1,n]^d \times [1,n]^d\}.\end{align*}

\begin{Lemma}\label{trT}
Let $\tr$ be an {\em $(A,\a)$-symmetrizing} form. 
Then the algebra \(T^A_\a(n,d)\) has a central form \(\tr^T: T^A_\a(n,d) \to \k\) given by 
\begin{align*}
\tr^T(\eta^{\bb}_{\br, \bs}):= \delta_{\br, \bs} \tr(b_1) \cdots \tr(b_d).
\end{align*}
for all \((\bb, \br, \bs)  \in \Seq^B(n,d)\). Moreover, $\tr^S|_{T^A_\a(n,d)}=d!\,\tr^T$. 
\end{Lemma}
\begin{proof}
Recall $\tr^S$ from Lemma~\ref{L150518}. 
For \((\bb, \br, \bs) \in \Seq^B(n,d)\) we have
\begin{align*}
\tr^S(\eta^\bb_{\br, \bs}) = \tr^S\left([\bb,\br,\bs]^!_\c \xi^\bb_{\br, \bs}\right)
=
 [\bb,\br,\bs]^!_\c \cdot  \frac{d!}{[\bb,\br,\bs]^!}\delta_{\br, \bs} \tr(b_1) \cdots \tr(b_d).
\end{align*}
But, since \(\tr(b) = 0\) whenever \(b \in B_\a \cup B_\1\), we have \([\bb,\br, \bs]^!_\c = [\bb, \br, \bs]^!\) whenever \(\tr^S(\eta^\bb_{\br, \bs}) \neq 0\). So for all $(\bb, \br, \bs) \in \Seq^B(n,d)$ we have 
\begin{align*}
\tr^S(\eta^\bb_{\br, \bs}) =d!\,\delta_{\br, \bs} \tr(b_1) \cdots \tr(b_d)
=d!\, \tr^T(\eta^{\bb}_{\br, \bs}).
\end{align*}
As \(\k\) is a characteristic zero domain and $\tr^S$ is central by Lemma~\ref{L150518}, $\tr^T$ is also central.
\end{proof}

\begin{Lemma}\label{trstar}
For all \(x \in T^A_\a(n,d_1)\), \(y \in T^A_\a(n,d_2)\), we have \(\tr^T(x * y) = \tr^T(x)\tr^T(y)\).
\end{Lemma}
\begin{proof}
Take \((\bb, \br, \bs) \in \Seq^B(n,d_1)\) and \((\bc,\bt,\bu) \in \Seq^B(n,d_2)\). We have
\begin{align*}
\tr^T(\eta^\bb_{\br, \bs})\tr^T(\eta^\bc_{\bt,\bu}) = \delta_{\br,\bs}\delta_{\bt, \bu} \tr(b_1) \cdots \tr(b_{d_1})\tr(c_1) \cdots \tr(c_{d_2}).
\end{align*}
On the other hand, by Lemma~\ref{xistarproducts}(iii), we have
\begin{align*}
\tr^T\left(\eta^\bb_{\br, \bs}*\eta^{\bc}_{\bt,\bu}\right)&=\frac{[\bb\bc, \br\bt, \bs\bu]_\a^!}{[\bb, \br, \bs]_\a^! [\bc, \bt, \bu]_\a^!}\tr^T\left(\eta_{\br\bt, \bs\bu}^{\bb\bc}\right)\\
&=
\frac{[\bb\bc, \br\bt, \bs\bu]_\a^!}{[\bb, \br, \bs]_\a^! [\bc, \bt, \bu]_\a^!}\delta_{\br\bt,\bs\bu}\tr(b_1) \cdots \tr(b_{d_1})\tr(c_1) \cdots \tr(c_{d_2}).
\end{align*}
Since \(\tr(b)=0\) whenever \(b \in \a\), we have the result.
\end{proof}

\begin{Lemma}\label{trBstar}
Let \((b^d, r^d, s^d) \in \Seq^B(n,d)\). Let \((\bc,\bt,\bu) \in \Seq^{B^*}(n,d)\). Then
\begin{align*}
\tr^T(\eta^{b^d}_{r^d,s^d} \eta^\bc_{\bt, \bu}) = 
\begin{cases}
\pm1 & \textup{if }\bc=(b^*)^d, \bt= s^d, \bu=r^d,\\
0 & \textup{otherwise.}
\end{cases}
\end{align*}
\end{Lemma}
\begin{proof}
Recall the central forms \(\tr^M\) and \(\tr^S\) from Lemmas \ref{trM} and \ref{L150518}. We have 
\begin{align*}
d!\,\tr^T(\eta^{b^d}_{r^d,s^d}  \eta^\bc_{\bt, \bu}) &=
\tr^S\left( [b^d, r^d, s^d]^!_\c [\bc, \bt, \bu]^!_\c\, \xi^{b^d}_{r^d,s^d}   \xi^\bc_{\bt, \bu}\right)\\
&=[b^d, r^d, s^d]^!_\c   [\bc, \bt, \bu]^!_\c \,\tr^M \left( 
\xi_{r^d,s^d}^{b^d}   \sum_{\sigma \in {}^{\bc,\bt, \bu}\D} (\xi_{t_1,u_1}^{c_1} \otimes \cdots \otimes \xi_{t_d,u_d}^{c_d})^\sigma
\right)\\
&=[b^d, r^d, s^d]^!_\c   [\bc, \bt, \bu]^!_\c \,\tr^M \left( 
\sum_{\sigma \in {}^{\bc,\bt, \bu}\D} \pm \delta_{s^d,\bt \sigma}(\xi_{r,u_{\sigma1}}^{bc_{\sigma 1}} \otimes \cdots \otimes \xi_{r,u_{\sigma d}}^{bc_{\sigma d}})
\right)\\
&=[b^d, r^d, s^d]^!_\c   [\bc, \bt, \bu]^!_\c \sum_{\sigma \in {}^{\bc,\bt, \bu}\D} \pm  \delta_{s^d,\bt \sigma}\delta_{r^d,\bu \sigma} \tr(bc_{\sigma_1}) \cdots \tr(bc_{\sigma d}).
\end{align*}
Since \(\tr(bc) = 0\) for all \(c \in B^*\backslash\{b^*\}\), we have that \(\tr^T(\eta^{b^d}_{r^d,s^d}   \eta^\bc_{\bt, \bu})=0\) unless \(\bc = (b^*)^d\), \(\bt = s^d\), and \(\bu =r^d\). Assume now that \(\bc = (b^*)^d\), \(\bt = s^d\), and \(\bu =r^d\). Then \({}^{\bc, \bt, \bu}\D = \{1\}\), so the above simplifies to
\begin{align*}
d!\,\tr^T(\eta^{b^d}_{r^d,s^d}   \eta^\bc_{\bt, \bu}) = \pm [b^d, r^d, s^d]^!_\c   [(b^*)^d, s^d, r^d]^!_\c.
\end{align*}
If \(b \in \a\), then \(b^* \in \c\), so we have 
\begin{align*}
[b^d, r^d, s^d]_{\c} = 0, \quad  [(b^*)^d,s^d, r^d]_\c = d.
\end{align*}
Conversely, if \(b \in \c\), then \(b^* \in \a\), so we have
\begin{align*}
[b^d, r^d, s^d]_{\c}  = d, \quad [(b^*)^d,s^d, r^d]_\c=0.
\end{align*}
If \(b \in A_\1\), then \(b^* \in A_\1\), and \((b^d, r^d, s^d) \in \Seq^B(n,d)\) implies \(d=1\), so we have
\begin{align*}
[b^d, r^d, s^d]_{\c} = [(b^*)^d,s^d, r^d]_\c = 0.
\end{align*}
Thus, in any case, we have 
$
d!\,\tr^T(\eta^{b^d}_{r^d,s^d}   \eta^\bc_{\bt, \bu}) = \pm d!,
$ 
so the result follows since \(\k\) is a characteristic zero domain.
\end{proof}

Now we upgrade this lemma. For \(\bb \in B^d\), write \(\bb^*:=(b_1^*, \ldots, b_d^*) \in (B^*)^d\).

\begin{Lemma}\label{etadual}
Let \((\bb, \br, \bs) \in \Seq^B(n,d)\) and \((\bc,\bt, \bu) \in \Seq^{B^*}(n,d)\). Then
\begin{align*}
\tr^T(\eta^\bb_{\br,\bs} \eta^\bc_{\bt, \bu}) = 
\begin{cases}
\pm 1 & \textup{if }(\bc, \bt, \bu) \sim (\bb^*, \bs, \br)\\
0 & \textup{otherwise}.
\end{cases}
\end{align*}
\end{Lemma}
\begin{proof}
We go by induction on \(d\). The base case follows from Lemma \ref{trBstar}. Let \(d>1\). We have by Lemma \ref{LSingleSep} that either
$%\begin{align*}
\eta^\bb_{\br, \bs} = \eta^{b^d}_{r^d,s^d}
$ %\end{align*}
for some \(b \in B\), \(r,s \in [1,n]\), or else \(\eta^\bb_{\br, \bs}\) may be written as 
\begin{align*}
\eta^\bb_{\br, \bs} =  \pm \eta^{\bb^{(1)}}_{\br^{(1)}, \bs^{(1)}} * \eta^{\bb^{(2)}}_{\br^{(2)},\bs^{(2)}},
\end{align*}
for some \(\delta\)-separated \((\bb^{(1)}, \br^{(1)}, \bs^{(1)}) \in \Seq^B(n,d_1)\), \((\bb^{(2)}, \br^{(2)}, \bs^{(2)}) \in \Seq^B(n,d_2)\), with \(d_1, d_2 >0\) and \(d_1 + d_2 = d\). In the former situation, the claim follows from Lemma \ref{trBstar}, so assume we are in the latter situation.

Recalling the notation of \S\ref{SSCoproduct}, we assume without loss of generality that \((\bc, \bt, \bu) \in \Seq_0^{B^*}(n,d)\). We have that
\begin{align*}
\tr^T(\eta^{\bb}_{\br, \bs} \eta^{\bc}_{\bt,\bu}) &= \tr^T((\pm \eta^{\bb^{(1)}}_{\br^{(1)}, \bs^{(1)}} *\eta^{\bb^{(2)}}_{\br^{(2)}, \bs^{(2)}})   \eta^\bc_{\bt, \bu})\\
&=\tr^T\left( 
\sum_{(\Triple^1,\Triple^2) \in \textup{Spl}(\bc, \bt, \bu)}  
{\small \frac{\pm[\bc, \bt, \bu]^!_{\c}}{[\Triple^1]^!_{\c}[\Triple^2]^!_{\c}}}
(\eta^{\bb^{(1)}}_{\br^{(1)},\bs^{(1)}}    \eta_{\Triple^1}) * 
(\eta^{\bb^{(2)}}_{\br^{(2)},\bs^{(2)}}    \eta_{\Triple^2})
\right)\\
&=
\sum_{(\Triple^1,\Triple^2) \in \textup{Spl}(\bc, \bt, \bu)}  
{\small \frac{\pm[\bc, \bt, \bu]^!_{\c}}{[\Triple^1]^!_{\c}[\Triple^2]^!_{\c}}}
\tr^T(\eta^{\bb^{(1)}}_{\br^{(1)},\bs^{(1)}}    \eta_{\Triple^1})   
\tr^T(\eta^{\bb^{(2)}}_{\br^{(2)},\bs^{(2)}}    \eta_{\Triple^2}),
\end{align*}
applying Corollary \ref{coprodeta} and Lemma \ref{Lstfo} for the second equality and Lemma \ref{trstar} for the third equality. If this is nonzero, then by the induction assumption we must have \(\Triple^1 \sim ((\bb^{(1)})^*, \bs^{(1)}, \br^{(1)})\) and \(\Triple^2 \sim ((\bb^{(2)})^*, \bs^{(2)}, \br^{(2)})\) for some \((\Triple^1, \Triple^2) \in  \textup{Spl}(\bc, \bt, \bu) \), which implies \((\bb^*, \bs, \br) \sim (\bc, \bt, \bu)\).

On the other hand, assume \((\bb^*, \bs, \br) \sim (\bc, \bt, \bu)\). Then there is exactly one \((\Triple^1, \Triple^2) \in  \textup{Spl}(\bc, \bt, \bu) \) such that \(\Triple^1 \sim ((\bb^{(1)})^*, \bs^{(1)}, \br^{(1)})\) and \(\Triple^2 \sim ((\bb^{(2)})^*, \bs^{(2)}, \br^{(2)})\). Then by the above we have
\begin{align*}
\tr^T(\eta^{\bb}_{\br, \bs} \eta^{\bc}_{\bt,\bu}) &= \pm {\small \frac{[\bc, \bt, \bu]^!_{\c}}{[\Triple^1]^!_{\c}[\Triple^2]^!_{\c}}} \tr^T(\eta^{\bb^{(1)}}_{\br^{(1)},\bs^{(1)}}    \eta_{\Triple^1})   
\tr^T(\eta^{\bb^{(2)}}_{\br^{(2)},\bs^{(2)}}    \eta_{\Triple^2})
=  \pm {\small \frac{[\bc, \bt, \bu]^!_{\c}}{[\Triple^1]^!_{\c}[\Triple^2]^!_{\c}}},
\end{align*}
using the induction assumption. But note, since \((\bb^{(1)}, \br^{(1)}, \bs^{(1)})\) and \((\bb^{(2)}, \br^{(2)}, \bs^{(2)})\) are \(\delta\)-separated, we have
\(
[\bc, \bt, \bu]_\c^! = [\Triple^1]_\c^! [\Triple^2]_\c^!,
\)
which completes the proof.
\end{proof}

\begin{Corollary}\label{Tsym}
If $\tr$ is an $(A,\a)$-symmetrizing form on $A$, then the algebra $T^A_\a(n,d)$ is symmetric, with symmetrizing form $\tr^T$.
\end{Corollary}
\begin{proof}
The form \(\tr^T\) is central by Lemma \ref{trT}. Moreover, by Lemma \ref{etadual}, we have that $(\cdot,\cdot)_{\tr^T}$ is a perfect pairing. 
\end{proof}

\begin{Remark}
Given a superalgebra \(C\), this result recovers the symmetricity  property of the Turner double algebra \(D^C(n,d)\) described in \S\ref{SSTE}. In particular, \(E(C)= C \oplus C^*\) has a symmetrizing form given by \(\tr(a,\alpha)=\alpha(a)\), so the symmetricity of \(D^C(n,d) \cong T^{E(C)}_{C_\0}(n,d)\) follows from Corollary \ref{Tsym}.
\end{Remark}

\section{Double centralizer property}\label{SDCP}
Throughout the subsection we assume that $\k$ is a principal ideal domain (as usual of characteristic $0$). 
All modules and algebras are assumed to be free of finite rank as $\k$-modules. 
An element $v$ of a $\k$-module $V$ is called {\em divisible} if there is $w\in V$ and a non-unit $m\in \k$ with $v=mw$. Otherwise $v$ is called {\em indivisible}. 
We let $\K$ to be the field of fractions of $\k$. 

\subsection{Double centralizer idempotents}

Let $S$ be a $\k$-algebra and  $e\in S$ be an idempotent. Considering $Se$ as a right $eSe$-module, we have an algebra homomorphism
$$
\la:S \to \End_{eSe}(Se),\ s\mapsto \la_s
$$
where $\la_s(s'e)=ss'e$ for all $s,s'\in S$. We say that $e$ is a {\em double centralizer idempotent for $S$} if $\la$ is an isomorphism. 
We say that $e$ is a {\em sound idempotent for $S$} if $\la$ sends indivisible elements of $S$ to indivisible elements of $\End_{eSe}(Se)$. 
Clearly, a double centralizer idempotent  is sound. 

Set 
$S_\K:=S\otimes_\k \K$ and use the map $s \mapsto s\otimes 1_\K$ to identify $S$ as a subset of $S_\K$. Then  $Se\otimes_\k\K=S_\K e$. Using the Universal coefficient theorem, we identify  $\End_{eSe}(Se)\otimes_\k\K=\End_{eS_\K e}(S_\K e)$, so that the $\la_\K:=\la\otimes\id_\K$ is the map
$$
\la_\K:S_\K\to \End_{eS_\K e}(S_\K e),\ s\mapsto \la_s
$$
where $\la_s$ is the left multiplication by $s$. Clearly, if $e$ is a double centralizer idempotent for $S$ then $e$ is a double centralizer idempotent for $S_\K$, but not vice versa in general.

\begin{Lemma} \label{LK} %{\rm \cite{}}%{\bf ()}
Let  $e\in S$ be an idempotent. Then $e$ is a double centralizer idempotent for $S$ if and only if $e$ is a double centralizer idempotent for $S_\K$ and 
$e$ is sound for $S$.  
\end{Lemma}
\begin{proof}
The `only-if' part is immediate. For the `if' part, the assumption that $e$ is a double centralizer idempotent for $S_\K$ implies that $\la:S\to \End_{eS e}(Se)$ is injective. Moreover, 
$$
\rank_\k S=\dim_\K S_\K=\dim  \End_{eS_\K e}(S_\K e)=\rank_\k \End_{eS e}(Se).
$$ 
So $\la$ is a full rank embedding. As $e$ is sound, it now follows that $\la$ is surjective. 
\end{proof}

\subsection{Double centralizer property for $S^A(n,d)$} Let $(A,\a)$ be a unital good pair, $e\in \a$ be an idempotent, and $\xi^e\in T^A_\a(n,d)$ be the idempotent of (\ref{E070318_2}). Suppose that $e$ is a double centralizer idempotent for $A$. Although it is not true in general that $\xi^e$ is then a double centralizer for $T^A_\a(n,d)$, this is true in some interesting situations. 

In view of Lemma~\ref{LK}, to verify that $\xi^e$ is a double centralizer for $T^A_\a(n,d)$, it suffices to check that  $\xi^e$ is a double centralizer idempotent for $S^A(n,d)$ and that $\xi^e$ is sound for $T^A_\a(n,d)_\K = S^A(n,d)_\K$. It turns out that the first condition is always true provided $d\leq n$. This will follow from the following stronger theorem: 

\begin{Theorem} \label{Tk} %{\rm \cite{}}%{\bf ()}
If $e\in A$ is a double centralizer idempotent for $A$ and $d\leq n$, then $\xi^e$  is a double centralizer idempotent for $S^A(n,d)$. 
\end{Theorem}
\begin{proof}
As usual, we write $\bar A:=eAe$.  
First we show
\begin{align}\label{DCP1}
M_n(A) \cong  \End_{M_n(\bar A)}(M_n(Ae)).
\end{align}
We have an algebra homomorphism \(\la:M_n(A) \to \End_{M_n(\bar A)}(M_n(Ae))\) given by left multiplication. On the other hand, let \(\phi \in \End_{M_n(\bar A)}(M_n(Ae))\). Since \(\phi(E_{i,j}^{ae})=\phi(E_{i,t}^{ae}E_{t,j}^{e})=\phi(E_{i,t}^{ae})E_{t,j}^{e}\) for any \(a \in A\) and \(i,t,j \in [1,n]\), it follows that there exist functions \(\phi_{k,i}: Ae \to Ae\) such that 
\begin{align*}
\phi(E_{i,j}^{ae})= \sum_{k} E_{k,j}^{\phi_{k,i}(ae)}.
\end{align*}
Moreover, each \(\phi_{k,i} \in \End_{\bar A}(Ae)\), since for all \(a \in A, b \in \bar A\), we have
\begin{align*}
\sum_{k}E_{k,j}^{\phi_{k,i}(aeb)}&=
\phi(E^{aeb}_{i,j})=
\phi(E^{ae}_{i,j}E^{b}_{j,j}) = \phi(E^{ae}_{i,j})E^{b}_{j,j}\\
&= \sum_{k}E^{\phi_{k,i}(ae)}_{k,j}E^{b}_{j,j} = \sum_{k}E^{\phi_{k,i}(ae)b}_{k,j}.
\end{align*}
implies that \(\phi_{k,i}(aeb)=\phi_{k,i}(ae)b\). Thus, since \(A \cong \End_{\bar A}(Ae)\), we have that \(\phi_{k,i}\) is given by left multiplication by a unique \(x_{k,i} \in A\). Then, since 
\begin{align*}
\left( \sum_{k,i} E^{x_{k,i}}_{k,i} \right)E^{ae}_{r,s} = \sum_{k}E_{k,s}^{x_{k,r}ae} = \sum_{k}E^{\phi_{r,k}(ae)}_{k,s} = \varphi(E_{r,s}^{ae}),
\end{align*}
we have a well-defined map \(\rho: \End_{M_n(\bar A)}(M_n(Ae)) \to M_n(A)\) given by \(\phi \mapsto \sum_{k,i}E_{k,i}^{x_{k,i}}\). It is clear that \(\la\) and \(\rho\) are mutual inverses, proving (\ref{DCP1}).

Let us now write 
\[ \begin{array}{lll}
S^{A}:= S^A(n,d) \;\;\;\;\;\;\;\;& S^{Ae}:=(M_n(Ae)^{\otimes d})^{\Si_d}\;\;\;\;\;\;\;\;& S^{\bar A}:=S^{\bar A}(n,d) \\
M^A:= M_n(A)^{\otimes d}& M^{Ae}:=M_n(Ae)^{\otimes d} & M^{\bar A}:=M_n(\bar A)^{\otimes d}
 \end{array}\] 
 For \(\ba = (a_1, \ldots, a_d) \in A^d\), \(\br, \bs \in [1,n]^d\), we will write
 \begin{align*}
 E^{\ba}_{\br, \bs}:= \xi^{a_1}_{r_1, s_1} \otimes \cdots \otimes \xi^{a_d}_{r_d,s_d}.
 \end{align*}
Our next task is to show
\begin{align}\label{DCP3}
\End_{S^{\bar A}}(S^{Ae}) \cong \End_{M^{\bar A}}(M^{Ae})^{\Si_d}.
\end{align}
%\begin{align*}\End_{(M_n(\bar A)^{\otimes d})^{\Si_d}}((M_n(Ae)^{\otimes d})^{\Si_d})\cong\End_{M_n(\bar A)^{\otimes d}}(M_n(Ae)^{\otimes d})^{\Si_d}.\end{align*}
The \(\Si_d\)-action on \(\End_{M^{\bar A}}(M^{Ae})^{\Si_d}\) is given by \(f^{\sigma}(x):= f(x^{\sigma^{-1}})^{\sigma}\). Let 
$$f \in\End_{M^{\bar A}}(M^{Ae})^{\Si_d}.
$$ 
Consider the restriction \(f_{\textup{res}}\) of \(f\) to the \(\k\)-submodule \(S^{Ae}\). For any \(\al \in S^{Ae}\), we have
\begin{align*}
f_{\res}(\al)^{\si} = f(\al)^{\si} = f(\al^{\si^{-1}})^{\si} =f^\si(\al) = f(\al) = f_{\res}(\al),
\end{align*}
so \(f_{\res} \in\End_{S^{\bar A}}(S^{Ae}) \) and \(\res: \End_{S^{\bar A}}(S^{Ae}) \to \End_{M^{\bar A}}(M^{Ae})^{\Si_d} \) is an algebra homomorphism. 

Now, let \(g \in \End_{S^{\bar A}}(S^{Ae})\). Write \(\bbe:=(e,\ldots, e) \in A^d\), and \(\omega = (1, \ldots, d) \in [1,n]^d\). We inflate \(g\) to a linear map \(g_{\textup{infl}}:M^{Ae} \to M^{Ae}\) via
\begin{align*}
g_{\textup{infl}}(E^\bx_{\br, \bs}):= g(\xi^{\bx}_{\br, \omega })E^{\bbe}_{\omega, \bs}.
\end{align*}
Note that for any \(\si \in \Si_d\), we also have
\begin{align*}
g_{\textup{infl}}(E^\bx_{\br, \bs})&=
g(\xi^\bx_{\br, \omega})E^\bbe_{\omega, \bs} = g(\xi^\bx_{\br, \omega^\si} \xi^\bbe_{\omega^{\si},\omega})E^\bbe_{\omega, \bs}\\
&=
g(\xi^\bx_{\br, \omega^\si}) \xi^\bbe_{\omega^{\si},\omega}E^\bbe_{\omega, \bs}
=
g(\xi^\bx_{\br, \omega^\si})E^\bbe_{\omega^\si, \bs}.
\end{align*}
We will show that \(g_{\textup{infl}} \in \End_{M^{\bar A}}(M^{Ae})^{\Si_d}\). First we check symmetry:
\begin{align*}
g_{\textup{infl}}^\si(E^\bx_{\br, \bs}) &= g_{\textup{infl}}((E^\bx_{\br, \bs})^{\si^{-1}})^{\si} = 
(-1)^{\langle \bx, \sigma^{-1} \rangle} g_{\textup{infl}}(E^{\bx^{\si^{-1}}}_{\br^{\si^{-1}},\bs^{\si^{-1}}})^{\si}\\
&=(-1)^{\langle \bx, \sigma^{-1} \rangle} [g(\xi^{\bx^{\si^{-1}}}_{\br^{\si^{-1}},\omega})E_{\omega, \bs^{\si^{-1}}}^{\bbe}]^{\si}\\
&=(-1)^{\langle \bx, \sigma^{-1} \rangle} g(\xi^{\bx^{\si^{-1}}}_{\br^{\si^{-1}},\omega})^{\si}(E_{\omega, \bs^{\si^{-1}}}^{\bbe})^{\si}\\
&=(-1)^{\langle \bx, \sigma^{-1} \rangle} g(\xi^{\bx^{\si^{-1}}}_{\br^{\si^{-1}},\omega})E_{\omega^{\si}, \bs}^{\bbe}\\
&=(-1)^{\langle \bx, \sigma^{-1} \rangle} (-1)^{\langle \bx^{\si^{-1}}, \si \rangle} g(\xi^\bx_{\br, \omega^\si})E^\bbe_{w^\si,\bs}
=g_\textup{infl}(E^\bx_{\br, \bs}).
\end{align*}
Now we check that \(g_\textup{infl}\) is an \(M^{\bar A}\)-homomorphism. Let \(\bx \in (Ae)^d\), \(\bb \in \bar{A}^d\), and \(\br, \bs, \bt, \bu \in [1,n]^d\).
\begin{align*}
g_{\textup{infl}}(E^\bx_{\br, \bs} E^\bb_{\bt, \bu}) &= \delta_{\bs, \bt}g_{\textup{infl}}(E^{\bx \bb}_{\br, \bu}) = \delta_{\bs, \bt}g(\xi^{\bx\bb}_{\br, \omega})E^\bbe_{\omega, \bu}
= \delta_{\bs, \bt}g(\xi^{\bx}_{\br, \omega}\xi^{\bb}_{\omega, \omega})E^\bbe_{\omega, \bu}\\
&=\delta_{\bs, \bt}g(\xi^{\bx}_{\br, \omega})\xi^{\bb}_{\omega, \omega}E^\bbe_{\omega, \bu}
=\delta_{\bs, \bt}g(\xi^{\bx}_{\br, \omega})\xi^{\bb}_{\omega, \omega}E^\bbe_{\omega, \bu}
= \delta_{\bs, \bt}g(\xi^{\bx}_{\br, \omega})E^{\bb}_{\omega, \bu}\\
&= g(\xi^{\bx}_{\br, \omega})E^{\bbe}_{\omega, \bs}E^{\bb}_{\bt, \bu} = g_{\textup{infl}}(E^\bx_{\br, \bs}) E^\bb_{\bt, \bu}. 
\end{align*}
Therefore \(g_{\textup{infl}} \in \End_{M^{\bar A}}(M^{Ae})^{\Si_d}\). Now we show that \(\res\) and \(\textup{infl}\) are mutual inverses. Let \(\bx \in (Ae)^d\), and \(\br, \bs \in [1,n]^d\). For \(f \in \End_{M^{\bar A}}(M^{Ae})^{\Si_d}\) we have
\begin{align*}
(f_{\res})_{\infl}(E^{\bx}_{\br, \bs}) = f_{\res}(\xi^{\bx}_{\br, \omega})E^{\bbe}_{\omega, \bs}=f(\xi^{\bx}_{\br, \omega})E^{\bbe}_{\omega, \bs} = f(\xi^\bx_{\br,\omega}E^\bbe_{\omega, \bs})=f(E^{\bx}_{\br, \bs}),
\end{align*}
so \((f_{\res})_{\infl} = f\).

Now, let \(g \in \End_{S^{\bar A}}(S^{Ae})\), and write \(g':=(g_{\textup{infl}})_{\res}\). We have \begin{align*}
g'(\xi^{\bx}_{\br, \omega}) &= g_{\textup{infl}}(\xi^{\bx}_{\br, \omega}) = g_{\textup{infl}}\left( \sum_{\si \in \Si_d} (E^\bx_{\br, \omega})^\si\right) = g_{\textup{infl}}\left(\sum_{\si \in \Si_d} (-1)^{\langle \bx, \si \rangle }E^{\bx^\si}_{\br^\si, \omega^\si}\right)\\
&= \sum_{\si \in \Si_d} (-1)^{\langle \bx, \si \rangle} g_{\textup{infl}}(E^{\bx^\si}_{\br^\si, \omega^\si})
=  \sum_{\si \in \Si_d} (-1)^{\langle \bx, \si \rangle} g(\xi^{\bx^\si}_{\br^\si, \omega^\si})E^\bbe_{\omega^\si, \omega^\si}\\
&= \sum_{\si \in \Si_d} (-1)^{\langle \bx, \si \rangle} (-1)^{\langle \bx^\si, \si^{-1} \rangle}g(\xi^{\bx}_{\br, \omega})E^\bbe_{\omega^\si, \omega^\si}
= \sum_{\si \in \Si_d} g(\xi^{\bx}_{\br, \omega})E^\bbe_{\omega^\si, \omega^\si}\\
&=g(\xi^{\bx}_{\br, \omega}) \sum_{\si \in \Si_d} E^{\bbe}_{w^\si, w^\si} = g(\xi^{\bx}_{\br, \omega})\xi^\bbe_{\omega, \omega} = g(\xi^{\bx}_{\br, \omega}\xi^\bbe_{\omega, \omega})
=g(\xi^{\bx}_{\br, \omega}),
\end{align*}
so \(g'(\xi^{\bx}_{\br, \omega}) = g(\xi^{\bx}_{\br, \omega})\). We also have
\begin{align*}
\xi^{\bx}_{\br, \omega} \xi^{e^d}_{\omega, \bs} = \xi^{x_1}_{r_1,s_1} * \cdots * \xi^{x_d}_{r_d,s_d} = |\Si_{\bx, \br, \bs}|\cdot \xi^{\bx}_{\br, \bs}.
\end{align*}
Thus, writing \(m:=  |\Si_{\bx, \br, \bs}| \in \k\), we have
\begin{align*}
mg'(\xi^{\bx}_{\br, \bs}) &= g'(m \xi^{\bx}_{\br, \bs}) = g'(\xi^{\bx}_{\br, \omega} \xi^{e^d}_{\omega, \bs}) = g'(\xi^{\bx}_{\br, \omega}) \xi^{e^d}_{\omega, \bs}\\ &= g(\xi^{\bx}_{\br, \omega}) \xi^{e^d}_{\omega, \bs} = g(\xi^{\bx}_{\br, \omega} \xi^{e^d}_{\omega, \bs}) = g(m \xi^{\bx}_{\br, \bs}) = mg(\xi^{\bx}_{\br, \bs}).
\end{align*}
As \(m \neq 0\) and \(S^{Ae}\) is free over \(\k\), this implies that \(g'(\xi^{\bx}_{\br, \bs}) = g(\xi^{\bx}_{\br, \bs})\). Thus \(\res\) and \(\textup{infl}\) are mutual inverses, proving (\ref{DCP3}).

Now, note that the action of \(\Si_d\) intertwines the isomorphisms
\begin{align*}\label{DCP2}
M_n(A)^{\otimes d} \cong \End_{M_n(\bar A)}(M_n(Ae))^{\otimes d} \cong
\End_{M^{\bar A}}(M^{Ae}),
\end{align*}
so we have 
\begin{align*}
\End_{\xi^e S^A \xi^e}(S^A \xi^e) = \End_{S^{\bar A}}(S^{Ae}) \cong \End_{M^{\bar A}}(M^{Ae})^{\Si_d} \cong (M_n(A)^{\otimes d})^{\Si_d} = S^A,
\end{align*}
as desired.
\end{proof}

\begin{Corollary} \label{TK} %{\rm \cite{}}%{\bf ()}
If $e\in A$ is a double centralizer idempotent for $A$ and $d\leq n$, then $\xi^e$  is a double centralizer idempotent for $S^A(n,d)_\K=T^A_\a(n,d)_\K$. 
\end{Corollary}

\subsection{Computations in extended zigzag Schur algebras}
Recall the notation of \S\ref{SSSZ}. In particular, we have 
 the extended zigzag algebra $Z$ for a fixed $\ell$ and 
the idempotent $e:=e_0+\dots+e_{\ell-1}\in Z$. We will use the standard   basis $B=B_\z\sqcup B_\c\sqcup B_\1$ of $Z$. %Since all $c_j\in B_\c$, we have from Lemma~\ref{xistarproducts}(iii):

For $\br\in[1,n]^d$, set 
$
\Si_\br:=\{\si\in\Si_d\mid \br\si=\br\},
$ 
and  denote by ${}^\br\D$ the set of the shortest coset representatives for $\Si_\br\backslash\Si_d$. 

%\begin{Lemma} %\label{}%{\rm \cite{}}%{\bf ()} For $\br,\bs\in[1,n]^d$ and $\bt,\bu\in[1,n]^f$, we have $\eta^{c_1^d}_{\br,\bs}*\eta^{c_1^f}_{\bt,\bu}=\eta^{c_1^{d+f}}_{\br\bt, \bs\bu}$. \end{Lemma}

%From definitions we easily get:

\begin{Lemma} \label{L280318_3} %{\rm \cite{}}%{\bf ()}
Let $t\in [1,n]$ and $\br,\bs,\bu\in[1,n]^d$. Suppose that  $r_a\neq r_b$ and $s_a\neq s_b$ for all $1\leq a\neq b\leq d$. Then
\begin{enumerate}
\item[{\rm (i)}] $
\eta^{a_{\ell-1,\ell}^d}_{\br,t^d}\eta^{a_{\ell,\ell-1}^d}_{t^d,\bs}=
\pm
\sum_{\si\in\Si_d}(\sgn\, \si)\eta^{c_{\ell-1}^d}_{\br\si,\bs}.
$
\item[{\rm (ii)}] $
\eta^{e_{\ell}^d}_{\bu,t^d}\eta^{a_{\ell,\ell-1}^d}_{t^d,\bs}=\sum_{\si\in{}^\bu\D}\eta^{a_{\ell,\ell-1}^d}_{\bu\si,\bs}.
$
\end{enumerate}
\end{Lemma}
%\iffalse{
\begin{proof}
(i) We have
\begin{align*}
\eta^{a_{\ell-1,\ell}^d}_{\br,t^d}\eta^{a_{\ell,\ell-1}^d}_{t^d,\bs}&=\left(\sum_{\si\in\Si_d}%(\sgn \si)
(\xi^{a_{\ell-1,\ell}}_{r_{1},t}\otimes\dots\otimes\xi^{a_{\ell-1,\ell}}_{r_{ d},t})^\si\right)\left(\sum_{\tau\in\Si_d}
%(\sgn \tau)
(\xi^{a_{\ell,\ell-1}}_{t,s_{ 1}}\otimes\dots\otimes\xi^{a_{\ell,\ell-1}}_{t,s_{ d}})^\tau\right)
\\
&=\pm
\sum_{\si,\tau\in\Si_d}(\sgn\, \si)(\sgn\, \tau)\xi^{c_{\ell-1}}_{r_{\si 1},s_{\tau 1}}\otimes\dots\otimes \xi^{c_{\ell-1}}_{r_{\si d},s_{\tau d}}
\\
&=
\pm
\sum_{\si,\tau\in\Si_d}(\sgn\, \si)(\sgn\, \tau)\xi^{c_{\ell-1}}_{r_{\si \tau^{-1}\tau 1},s_{\tau 1}}\otimes\dots\otimes \xi^{c_{\ell-1}}_{r_{\si \tau^{-1}\tau d},s_{\tau d}}
\\
&=
\pm
\sum_{\si,\tau\in\Si_d}(\sgn\, \si\tau^{-1})(\xi^{c_{\ell-1}}_{r_{\si \tau^{-1} 1},s_{1}}\otimes\dots\otimes \xi^{c_{\ell-1}}_{r_{\si \tau^{-1} d},s_{d}})^\si,
\end{align*}
which equals the right hand side of the equation in (i). 

(ii) We have
\begin{align*}
\eta^{e_\ell^d}_{\bu,t^d}\eta^{a_{\ell,\ell-1}^d}_{t^d,\bs}&=\left(\sum_{\si\in{}^\bu\D}%(\sgn \si)
(\xi^{e_\ell}_{u_{1},t}\otimes\dots\otimes\xi^{e_\ell}_{u_{ d},t})^\si\right)\left(\sum_{\tau\in\Si_d}
%(\sgn \tau)
(\xi^{a_{\ell,\ell-1}}_{t,s_{ 1}}\otimes\dots\otimes\xi^{a_{\ell,\ell-1}}_{t,s_{ d}})^\tau\right)
\\
&=\sum_{ \si\in{}^\bu\D,\,\tau\in\Si_d}(\sgn \tau)\,\xi^{a_{\ell,\ell-1}}_{u_{\si 1},s_{\tau 1}}\otimes\dots\otimes \xi^{a_{\ell,\ell-1}}_{u_{\si d},s_{\tau d}}
\\
&=
\sum_{ \si\in{}^\bu\D,\,\tau\in\Si_d}(\sgn \tau)\,\xi^{a_{\ell,\ell-1}}_{u_{\si\tau^{-1}\tau 1},s_{\tau 1}}\otimes\dots\otimes \xi^{a_{\ell,\ell-1}}_{u_{\si \tau^{-1}\tau d},s_{\tau d}}
\\
&=
\sum_{ \si\in{}^\bu\D,\,\tau\in\Si_d}
(\xi^{a_{\ell,\ell-1}}_{u_{\si\tau^{-1} 1},s_{1}}\otimes\dots\otimes \xi^{a_{\ell,\ell-1}}_{u_{\si \tau^{-1} d},s_{ d}})^\tau,
\end{align*}
which equals the right hand side of the equation in (ii).
\end{proof}
%}\fi

We set $$P_d:=[1,n]^d\times [1,n]^d,$$ 
i.e. elements of $P_d$ are pairs $(\br,\bs)$ of words $\br=r_1\cdots r_d,\ \bs=s_1\cdots s_d$ in $[1,n]^d$. We also define
$$
P_d':=\{(\br,\bs)\in P_d\mid (r_a,s_a)\neq (r_b,s_b)\ \text{for all $1\leq a\neq b\leq c$}\}.
$$
For $b\in B_\0$, the triple $(b^d,\br,\bs)$ belongs to $\Seq^B(n,d)$ for all $(\br,\bs)\in P_d$, while for $b\in B_\1$, we have $(b^d,\br,\bs)\in\Seq^B(n,d)$ if and only if $(\br,\bs)\in P_d'$. 

Given $\la\in\La(n,d)$, we define 
$$
\br^\la:=1^{\la_1}\cdots n^{\la_n}\in[1,n]^d.
$$
We refer to such tuples as {\em leading tuples}. 
Then 
$$
\Si_\la:=\Si_{\br^\la}=\Si_{\la_1}\times\dots\times\Si_{\la_n}.
$$
For $(b^d,\br,\bs)\in\Seq^B(n,d)$, the corresponding $\Si_d$-orbit $[b^d,\br,\bs]$ has a representative of the form $[b^d,\bt,\br^\la]$ for some $\la\in \La(n,d)$ and a representative of the form $[b^d,\br^\mu,\bu]$ for some $\mu\in \La(n,d)$. 
So while working with elements of the form $\eta^{b^d}_{\br,\bs}\in T^Z_\z(n,d)$ we will often assume that $\bs$ or $\br$ is a leading tuple when convenient.

\begin{Lemma} \label{L280318_4} %{\rm \cite{}}%{\bf ()}
Let $\la\in\La(n,d)$, and $\bs,\bt,\bu\in[1,n]^d$ be such that $(\bs,\br^\la),(\br^\la,\bt)\in P_d'$. Then
\begin{enumerate}
\item[{\rm (i)}] $
\eta^{a_{\ell-1,\ell}^{d}}_{\bs,\br^\la}\eta^{a_{\ell,\ell-1}^{d}}_{\br^\la,\bt}=\pm \sum_{\si\in\Si_{\la}}(\sgn\, \si)\eta^{c_{\ell-1}^d}_{\bs\si,\bt}.
$
\item[{\rm (ii)}] $
\eta^{e_\ell^{d}}_{\bu,\br^\la}\eta^{a_{\ell,\ell-1}^{d}}_{\br^\la,\bt}= \sum_{\si\in{}^\bu\Si_{\la}}\eta^{a_{\ell,\ell-1}^d}_{\bs\si,\bt},
$
where ${}^\bu\Si_{\la}$ is the set of the shortest coset representatives for $(\Si_\bu\cap\Si_\la)\backslash\Si_\la$. 

\end{enumerate}
\end{Lemma}
\begin{proof}
For $u=1,\dots,n$, there exist words $\bs^u,\bt^u\in[1,n]^{\la_u}$ such that $\bs=\bs^1\cdots\bs^n,\, \bt=\bt^1\cdots\bt^n$. 
We have 
\begin{align*}
\eta^{a_{\ell-1,\ell}^{d}}_{\bs,\br^\la}\eta^{a_{\ell,\ell-1}^{d}}_{\br^\la,\bt}
&=\pm(\eta^{a_{\ell-1,\ell}^{\la_1}}_{\bs^1,1^{\la_1}}\eta^{a_{\ell,\ell-1}^{\la_1}}_{1^{\la_1},\bt^1})*\dots*(\eta^{a_{\ell-1,\ell}^{\la_n}}_{\bs^n,n^{\la_n}}\eta^{a_{\ell,\ell-1}^{\la_n}}_{n^{\la_n},\bt^n})
\\
&=\pm \left(\sum_{\si^1\in\Si_{\la_1}}(\sgn\, \si^1)\eta^{c_{\ell-1}^{\la_1}}_{\bs^1\si^1,\bt^1}\right)*\dots*\left(\sum_{\si^n\in\Si_{\la_n}}(\sgn\, \si^n)\eta^{c_{\ell-1}^{\la_n}}_{\bs^n\si^n,\bt^n}\right)
\\
&=\pm \sum_{\si\in\Si_{\la}}(\sgn\, \si)\eta^{c_{\ell-1}^d}_{\bs\si,\bt},
\end{align*}
where we have used Lemma~\ref{LSep} for the first equality, Lemma~\ref{L280318_3}(i) for the second equality and Lemma~\ref{xistarproducts}(iii) for the last equality. This proves (i). The proof of (ii) is similar but uses Lemma~\ref{L280318_3}(ii) instead of Lemma~\ref{L280318_3}(i). 
\end{proof}

Let $\la\in \La(n,d)$. For $r\in[1,n]$, let $x_r:=1+\sum_{s=1}^{r-1}\la_s$ and $y_r:=\sum_{s=1}^{r}\la_s$ so that 
$$
[1,d]=[x_1,y_1]\sqcup\dots\sqcup[x_n,y_n]
$$
and if $\br^\la=r_1\cdots r_d$ then $r_s=r$ if and only if $s\in [x_r,y_r]$. Let $0\leq c\leq d$ and $\mu\in\La(n,c)$, $\nu\in\La(n,d-c)$ satisfy $\mu+\nu=\la$. We denote 
\begin{align*}
\Om^{\mu,\nu}:=\{U\subseteq [1,d]\mid |U\cap [x_r,y_r]|=\mu_r\ \text{for all $r=1,\dots,n$}\}
\end{align*}
Let $\bt=t_1\cdots t_d\in[1,n]^d$. 
% be such that $(\br^{\la},\bq)\in P_d'$ so that $\eta_{\br^{\la},\bq}^{a_{\ell,\ell-1}^{d}}\in\T^Z_\z(n,d)$ makes sense. 
For $U=\{u_1<\dots<u_c\}\in\Om^{\mu,\nu}$, we set 
$$U'=\{v_1<\dots<v_{d-c}\}:=[1,d]\setminus U$$ and 
$$
\bt^U:=t_{u_1}\cdots t_{u_c},\quad \bt^{U'}:=t_{v_1}\cdots t_{v_{d-c}}. 
$$
With this notation, Corollary~\ref{coprodeta} yields:

\begin{Lemma} \label{L280318} %{\rm \cite{}}%{\bf ()}
Let $\la\in \La(n,d)$ and $\bt\in[1,n]^d$. If $(\br^{\la},\bt)\in P_d'$ then
$$
\nabla(\eta_{\br^{\la},\bt}^{a_{\ell,\ell-1}^{d}})=\sum_{c=0}^d\,\sum_{\substack{\mu\in\La(n,c),\,\nu\in\La(n,d-c)\\ \mu+\nu=\la}}\,\sum_{U\in\Om^{\mu,\nu}} 
\pm \eta_{\br^{\mu},\bt^U}^{a_{\ell,\ell-1}^{c}}
\otimes \eta_{\br^{\nu},\bt^{U'}}^{a_{\ell,\ell-1}^{d-c}}.
$$
\end{Lemma}

\begin{Lemma} \label{L280318_6}%{\rm \cite{}}%{\bf ()}
Let $0\leq c\leq d$, $\mu\in\La(n,c)$, $\nu\in\La(n,d-c)$, $\br\in[1,n]^c$, $\bs\in[1,n]^{d-c}$ and $\bt\in[1,n]^d$. Suppose that  
$(\br,\br^\mu)\in P'_c$, and $(\br^{\mu+\nu},\bt)\in P_d'$. Then 
$$(\eta^{a_{\ell-1,\ell}^{c}}_{\br,\br^\mu}*\eta^{e_\ell^{d-c}}_{\bs,\br^\nu})\eta_{\br^{\mu+\nu},\bt}^{a_{\ell,\ell-1}^{d}}
=
\sum_{U\in\Om^{\mu,\nu}}\sum_{\si\in\Si_{\mu},\, \tau\in{}^\bs\Si_{\nu}}\pm \eta^{c_{\ell-1}^ca_{\ell,{\ell-1}}^{d-c}}_{(\br\si)(\bs\tau),\bt^U\bt^{U'}}.
$$
\end{Lemma}
\begin{proof}
Using Lemmas~\ref{Lstfo}, \ref{L280318} and \ref{L280318_4}, we have 
\begin{align*}
(\eta^{a_{{\ell-1},\ell}^{c}}_{\br,\br^\mu}*\eta^{e_\ell^{d-c}}_{\bs,\br^\nu})\eta_{\br^{\mu+\nu},\bt}^{a_{\ell,{\ell-1}}^{d}}
&=\sum_{U\in\Om^{\mu,\nu}}
\pm (\eta^{a_{{\ell-1},\ell}^{c}}_{\br,\br^\mu} \eta_{\br^{\mu},\bt^U}^{a_{\ell,{\ell-1}}^{c}})*(\eta^{e_\ell^{d-c}}_{\bs,\br^\nu}\eta_{\br^{\nu},\bt^{U'}}^{a_{\ell,{\ell-1}}^{d-c}})
\\
&=\sum_{U\in\Om^{\mu,\nu}}\left(\sum_{\si\in\Si_{\mu}}\pm \eta^{c_{\ell-1}^c}_{\br\si,\bt^U}\right)*
\left(
\sum_{\tau\in{}^\bs\Si_{\nu}}\eta^{a_{\ell,{\ell-1}}^{d-c}}_{\bs\tau,\bt^{U'}}
\right),
\end{align*}
which equals the right hand side of the equality in the claim by Lemma~\ref{LSingleSep}. 
\end{proof}

\begin{Remark} \label{R1} %{\rm \cite{}}%{\bf ()}
{\rm 
Let $(\bb,\br,\bs), (\bb',\br',\bs')\in\Seq^B(n,d)$. 
By (\ref{E080717}) and Lemmas~\ref{LBasis'},\,\ref{LXiZero}, we have that 
$\eta^{\bb}_{\br,\bs}$ and $\eta^{\bb'}_{\br',\bs'}$ are proportional if and only if $(\bb,\br,\bs)\sim (\bb',\br',\bs')$. It will be important later on that all the basis elements appearing in the right hand side of Lemma~\ref{L280318_6} are linearly independent.  
}
\end{Remark}

Let
$$
K_d:=\bigsqcup_{c=0}^d\La(4,c)= \{\ka=(\ka_1,\ka_2,\ka_3,\ka_4)\in\La(4)\mid |\ka|\leq d\}.
$$ 
For $\ka\in K_d$, we set  
\begin{align*}
P_{\ka}&:=P_{\ka_1}'\times P_{\ka_2}\times P_{\ka_3}'\times P_{\kappa_4}.
\end{align*}
Let  
$
\big((\br^1,\bs^1),(\br^2,\bs^2),(\br^3,\bs^3),(\br^4,\bs^4)\big)\in P_{\ka}.
$
We often denote 
$$\br:=\br^1\br^2\br^3\br^4,\ \bs:=\bs^1\bs^2\bs^3\bs^4\in[1,n]^{|\ka|},
$$
and, abusing notation, write $(\br,\bs)\in P_\ka$. Given $(\br,\bs)\in P_\ka$, we define the following element of $T^Z_\z(n,|\ka|)$: 
$$
\eta^{\ka}_{\br,\bs}:=
\eta^{a_{{\ell-1},\ell}^{\ka_1}e_\ell^{\ka_2}a_{\ell,{\ell-1}}^{\ka_3}c_{\ell-1}^{\ka_4}}_{\br,\bs}=
\eta^{a_{{\ell-1},\ell}^{\ka_1}}_{\br^1,\bs^1}*\eta^{e_\ell^{\ka_2}}_{\br^2,\bs^2}*\eta^{a_{\ell,{\ell-1}}^{\ka_3}}_{\br^3,\bs^3}*\eta_{\br^4,\bs^4}^{c_{\ell-1}^{\ka_4}}.
$$
We have the equivalence relation $\sim$ on $P_\ka$ given by 
$$
(\br,\bs)\sim (\bt,\bu)\quad\text{if and only if}\quad  (\br^h,\bs^h)\sim(\bt^h,\bu^h)\ \text{for $h=1,2,3,4$}. 
$$

Let $B':=B\setminus\{a_{{\ell-1},\ell},e_\ell,a_{\ell,{\ell-1}},c_{\ell-1}\}$. 
Recall the equivalence relation $\sim$ on $\Seq^{B'}(n,d-|\ka|)$ from \S\ref{SComb}. By Lemmas~\ref{LBasis} and \ref{LSingleSep}, we have that 
\begin{equation}\label{E280318_8}
\{\eta^{\ka}_{\br,\bs}*\eta^\bb_{\bp,\bq}
\mid \ka\in K_d,\ (\br,\bs)\in P^{\ka}/\sim,\ 
(\bb,\bp,\bq)\in\Seq^{B'}(n,d-|\ka|)/\sim\}
\end{equation}
is a basis of $T^Z_\z(n,d)$.

\begin{Lemma} \label{LProdZ} %{\rm \cite{}}%{\bf ()}
Let $\ka\in K_d$, $(\br,\bs)\in P^\ka$ with $\bs^1=\br^\mu,\, \bs^2=\br^\nu$, and $(\bb,\bp,\bq)\in\Seq^{B'}(n,d-|\ka|)$.  Let  
%$\eta\in T'(n,d-|\ka|)$, 
$k=\ka_1+\ka_2$ and $(\br^{\mu+\nu},\bt)\in P_k'$. Then we have:
$$
(\eta^\ka_{\br,\bs}
*\eta^\bb_{\bp,\bq})(\eta_{\br^{\mu+\nu},\bt}^{a_{\ell,\ell-1}^{k}}*\xi^{e^{d-k}})=
\sum_{U\in\Om^{\mu,\nu}}\sum_{\si\in\Si_{\mu},\, \tau\in{}^{\br^2}\Si_{\nu}}\pm 
\eta^{c_{\ell-1}^{\ka_1+\ka_4}a_{\ell,{\ell-1}}^{\ka_2+\ka_3}}_{(\br^1\si)\br^4(\br^2\tau)\br^3,\bt^U\bs^4\bt^{U'}\bs^3}*\eta^\bb_{\bp,\bq}.
%*\eta^{a_{\ell,{\ell-1}}^{\ka_2+\ka_3}}_{(\br^2\tau)\br^3,\bt^{[1,d]\setminus U}\bs^3}
$$ 
\end{Lemma}
\begin{proof}
By Lemmas~\ref{LSep} and \ref{L280318_6}, the left hand side equals
\begin{align*}
&
\big((\eta^{a_{{\ell-1},\ell}^{\ka_1}}_{\br^1,\br^\mu}*\eta^{e_\ell^{\ka_2}}_{\br^2,\br^\nu})\eta_{\br^{\mu+\nu},\bt}^{a_{\ell,{\ell-1}}^{k}}\big)*\big((\eta^{a_{\ell,{\ell-1}}^{\ka_3}}_{\br^3,\bs^3}*\eta_{\br^4,\bs^4}^{c_{\ell-1}^{\ka_4}}*\eta^\bb_{\bp,\bq})\xi^{e^{d-k}}\big)
\\
=&
\left(\sum_{U\in\Om^{\mu,\nu}}\sum_{\si\in\Si_{\mu},\, \tau\in{}^{\br^2}\Si_{\nu}}\pm 
\eta^{c_{\ell-1}^{\ka_1}}_{\br^1\si,\bt^U}
*\eta^{a_{\ell,{\ell-1}}^{\ka_2}}_{\br^2\tau,\bt^{U'}}
\right) *\big(\eta^{a_{\ell,{\ell-1}}^{\ka_3}}_{\br^3,\bs^3}*\eta_{\br^4,\bs^4}^{c_{\ell-1}^{\ka_4}}*\eta^\bb_{\bp,\bq}\big),
\end{align*}
which equals the right hand side of the equality in the claim by Lemma~\ref{xistarproducts}(iii) and Lemma~\ref{LSingleSep}. 
\end{proof}

\begin{Remark} \label{R2} %{\rm \cite{}}%{\bf ()}
{\rm 
Suppose that in the assumption of Lemma~\ref{LProdZ}, we have additionally that $\bt=t_1,\dots,t_k$ satisfies $t_a\neq t_b$ for all $1\leq a\neq b\leq k$ and $\bt$ shares no letters in common with the words $\bs^3$ and $\bs^4$. 
Taking into account Remark~\ref{R1}, one can see that all the basis elements appearing in the right hand side of Lemma~\ref{LProdZ} are linearly independent.  
}
\end{Remark}

Recall the idempotent $\xi^e\in T^Z_\z(n,d)$. In this subsection we sometimes write $\xi^{e^d}:=\xi^e$, so that we also have idempotents $\xi^{e^c}\in T^Z_\z(n,c)$ for all $c\in\Z_{\geq 0}$. 
Recalling (\ref{Ea}), we also introduce idempotents 
$$
\xi^{(c,d-c)}:=\big((E^{e_\ell})^{\otimes c}\big)*\big((E^{e})^{\otimes d-c}\big)\in T^Z_\z(n,d)\qquad(0\leq c\leq d).
$$
Note that $\xi^{(0,d)}=\xi^e$ and $\xi^{(d,0)}=\xi^{e_\ell}$. Moreover,
\begin{equation}\label{}
\xi^{(c,d-c)}\xi^{(b,d-b)}=\de_{b,c}\xi^{(c,d-c)}. 
\end{equation}

The following is easily checked using Lemma~\ref{Lstfo}: 

\begin{Lemma} \label{L270318} %{\rm \cite{}}%{\bf ()}
Let $\ka\in K_d$, $(\br,\bs)\in P_\ka$, $(\bb,\bp,\bq)\in\Seq^{B'}(n,d-|\ka|)$, $0\leq k\leq d$, $\la\in\La(n,k)$, and suppose that $(\br^\la,\bt)\in P_k'$. Then
\begin{align*}
\xi^{(\ka_2+\ka_3,d-\ka_2-\ka_3)}(\eta^{\ka}_{\br,\bs}*\eta^\bb_{\bp,\bq})\xi^{(\ka_1+\ka_2,d-\ka_1-\ka_2)}&=\eta^{\ka}_{\br,\bs}*\eta^\bb_{\bp,\bq},
\\
\xi^{(k,d-k)}(\eta_{\br^{\la},\bt}^{a_{\ell,\ell-1}^{k}}*\xi^{e^{d-k}})\xi^e&=\eta_{\br^{\la},\bt}^{a_{\ell,\ell-1}^{k}}*\xi^{e^{d-k}}.
\end{align*}
In particular, 
$
(\eta^{\ka}_{\br,\bs}*\eta^\bb_{\bp,\bq})(\eta_{\br^{\la},\bt}^{a_{\ell,\ell-1}^{k}}*\xi^{e^{d-k}})=0,
$ unless $\ka_1+\ka_2=k$ and $\bs^1\bs^2\sim \br^\la$.
\end{Lemma}

\subsection{Double centralizer property for zigzag Schur algebras}
Recall that $\bar Z:=eZe$ is the zigzag algebra for a fixed $\ell\geq 1$.  

\begin{Lemma} %\label{}%{\rm \cite{}}%{\bf ()}
We have that $e$ is a double centralizer idempotent for $Z$. 
\end{Lemma}
\begin{proof}
As a right \(\bar Z\)-module, \(Ze\) decomposes as
\begin{align*}
Ze = e_\ell Ze \oplus e_{\ell-1}Ze \oplus \cdots \oplus e_0 Ze, 
\end{align*}
so it is enough to check that the algebra map \(\la: Z \to \End_{\bar Z}(Ze)\) restricts to an isomorphism \(e_j Z e_i \to \Hom_{\bar Z}(e_i Z e, e_j Ze)\), for all \(i,j \in I\).

Let \(i,j \in I\) with \(i \neq \ell\). The map \(\la|_{e_j Z e_i}: e_j Z e_i \to \Hom_{\bar Z}(e_i Z e, e_j Ze)\) is injective, since \(\la_x(e_i) = x\) for all \(x \in e_jZe_i\). Now let \(f \in \Hom_{\bar Z}(e_i Z e, e_j Ze)\). As \(f\) is a right \(\bar Z\)-module homomorphism, \(f\) is determined by the image of \(e_i\). Moreover, since \(f(e_i) = f(e_i)e_i\), we have \(f(e_i) \in e_jZe_i\), and thus \(f = \la_{f(e_i)}\), so \(\la|_{e_j Z e_i}\) is surjective, and thus an isomorphism.

Now let \(j \in I\), and consider the map \(\la|_{e_j Z e_\ell}: e_j Z e_\ell \to \Hom_{\bar Z}(e_\ell Z e, e_j Ze)\). Note that \(e_\ell Z e = \spa (a_{\ell,\ell-1})\). First we show \(\la|_{e_j Z e_\ell}\) is injective. If \(j\neq \ell,\ell-1\) then \(e_j Z e_\ell = 0\), and this is trivially true. If \(j=\ell-1\) then \(e_j Z e_\ell = \spa (a_{\ell-1,\ell})\), and \(\la_{a_{\ell-1,\ell}}(a_{\ell,\ell-1}) = c_{\ell-1} \neq 0\). If \(j=\ell\), then \(e_j Z e_\ell = \spa (e_\ell)\), and \(\la_{e_\ell}(a_{\ell,\ell-1}) = a_{\ell,\ell-1}  \neq 0\). Thus, in any case \(\la|_{e_j Z e_\ell}\) is injective.

Now we show \(\la|_{e_j Z e_\ell}\) is surjective. Let \(f \in \Hom_{\bar Z}(e_\ell Z e, e_j Ze)\). Since \(f(a_{\ell,\ell-1}) = f(a_{\ell,\ell-1})e_{\ell-1}\), we have that \(f(a_{\ell,\ell-1}) \in e_j Z e_{\ell-1}\). Since \(e_j Z e_{\ell-1} = 0\) for \(j\neq \ell,\ell-1,\ell-2\) we may assume \(j \in \{\ell,\ell-1,\ell-2\}\). If \(j =\ell-2\), then \(f(a_{\ell,\ell-1}) = \alpha a_{\ell-2,\ell-1}\) for some \(\alpha \in \k\). But then \(0 = f(a_{\ell,\ell-1} a_{\ell-1,\ell-2}) = f(a_{\ell,\ell-1})a_{\ell-1,\ell-2} = \alpha c_{\ell-2}\), so \(f=0 = \la_0\). If \(j = \ell-1\), then \(f(a_{\ell,\ell-1}) = \alpha e_{\ell-1} + \beta c_{\ell-1}\) for some \(\alpha, \beta \in \k\). But then 
\(
0=f(a_{\ell,\ell-1}c_{\ell-1}) = f(a_{\ell,\ell-1})c_{\ell-1} = \alpha c_{\ell-1}
\)
implies that \(\alpha =0\). Thus \(f(a_{\ell,{\ell-1}}) = \beta c_{\ell-1} = \beta a_{{\ell-1},\ell}a_{\ell,{\ell-1}}\) and \(f=\la_{\beta a_{{\ell-1},\ell}}\). Finally, assume \(j = \ell\). Then \(f(a_{\ell,{\ell-1}}) = \alpha a_{\ell,{\ell-1}} = \alpha e_\ell a_{\ell,{\ell-1}}\) for some \(\alpha \in \k\), so \(f = \la_{\alpha e_\ell}\). Thus in any case \(\la|_{e_j Z e_\ell}\) is surjective, and thereby and isomorphism, completing the proof.
\end{proof}

In view of Lemma~\ref{LTruncation}, we have $\xi^e  T^Z_\z(n,d)\xi^e = T^{\bar Z}_{\bar \z}(n,d)$.  
The main result of this subsection is

\begin{Theorem} \label{TDCPZ} %{\rm \cite{}}%{\bf ()}
Let $d\leq n$. Then $\xi^e$ is a double centralizer idempotent for $T^Z_\z(n,d)$. In particular, $T^Z_\z(n,d)\cong \End_{T^{\bar Z}_{\bar \z}(n,d)}(T^Z_\z(n,d)\xi^e)$.
\end{Theorem}

Theorem~\ref{TDCPZ} follows immediately from Lemma~\ref{LK}, Corollary~\ref{TK}, and the following proposition:

\begin{Proposition} %\label{}%{\rm \cite{}}%{\bf ()}
Let $d\leq n$. Then $\xi^e$ is a sound idempotent for $T^Z_\z(n,d)$. 
\end{Proposition}
\begin{proof} Set $T:=T^Z_\z(n,d)$. We will use the basis (\ref{E280318_8}) of $T$. 

Suppose for a contradiction that there exists an indivisible element 

$$x:=\sum_{\substack{\ka\in K_d,\,(\br,\bs)\in P_{\ka}/\sim,\\ 
(\bb,\bp,\bq)\in\Seq^{B'}(n,d-|\ka|)/\sim}}q^{\ka,\bb}_{\br,\bs;\bp,\bq}(\eta^{\ka}_{\br,\bs}*\eta^\bb_{\bp,\bq})\in T\qquad(q^{\ka,\bb}_{\br,\bs;\bp,\bq}\in\k)$$  such that $\la_x: T\xi^e\to T\xi^e$ is divisible, i.e. there exists a  non-zero non-unit $m\in\k$ such that $x\eta\xi^e\in mT$ for all $\eta\in T$. By the remarks preceding Lemma~\ref{L280318_4}, we may assume that all $\bs^1$ and $\bs^2$ are leading tuples. If $\bs^1=\br^\mu$ and $\bs^2=\br^\nu$, we write $(\br,\bs)\in P_\ka^{\mu,\nu}$.

We may also assume that among all indivisible elements $x$ as above, our $x$ has the smallest possible number of non-zero coefficients $q^{\ka,\bb}_{\br,\bs;\bp,\bq}$. Then $m$ does not divide $q^{\ka,\bb}_{\br,\bs;\bp,\bq}$ whenever $q^{\ka,\bb}_{\br,\bs;\bp,\bq}\neq 0$. 

Let $k\in\Z_{\geq 0}$ be such that some coefficient $q^{\ka,\bb}_{\br,\bs;\bp,\bq}$ with $\ka_1+\ka_2=k$ is non-zero. We assume that $(\br,\bs)\in P_\ka^{\mu,\nu}$ for some such non-zero coefficient. 
We now pick $\bt=(t_1,\dots,t_k)\in[1,n]^d$ be such that:
\begin{enumerate}
\item[{\rm (1)}] $t_r\neq t_s$ for all $1\leq r\neq s\leq k$;
\item[{\rm (2)}] the words $\bs^4$ and $\bs^3$ have no letters of the form $t_r$ for $1\leq r\leq k$. 
\end{enumerate}
Such $\bt$ exists by the assumption that $d\leq n$. 

We have $(\br^{\mu+\nu},\bt)\in P_k'$. 
By Lemma~\ref{L270318}, we have $\eta_{\br^{\mu+\nu},\bt}^{a_{\ell,\ell-1}^{k}}*\xi^{e^{d-k}}\in T\xi^e$, so by the assumptions made, we must have  $x(\eta_{\br^{\mu+\nu},\bt}^{a_{\ell,\ell-1}^{k}}*\xi^{e^{d-k}})\in mT$. On the other hand, by Lemmas~\ref{L270318} and \ref{LProdZ}, we have that $x(\eta_{\br^{\mu+\nu},\bt}^{a_{\ell,\ell-1}^{k}}*\xi^{e^{d-k}})$ equals 
\begin{align*}
&\sum_{\substack{\ka\in K_d,\,(\br,\bs)\in P_{\ka}^{\mu,\nu}/\sim,\\ 
(\bb,\bp,\bq)\in\Seq^{B'}(n,d-|\ka|)/\sim}}q^{\ka,\bb}_{\br,\bs;\bp,\bq}
(\eta^{\ka}_{\br,\bs}*\eta^\bb_{\bp,\bq})
(\eta_{\br^{\mu+\nu},\bt}^{a_{\ell,\ell-1}^{k}}*\xi^{e^{d-k}})
\\
=&\sum_{\substack{\ka\in K_d,\,(\br,\bs)\in P_{\ka}^{\mu,\nu}/\sim,\\ 
(\bb,\bp,\bq)\in\Seq^{B'}(n,d-|\ka|)/\sim}}
q^{\ka,\bb}_{\br,\bs;\bp,\bq}
\sum\pm 
\eta^{c_{\ell-1}^{\ka_1+\ka_4}a_{\ell,{\ell-1}}^{\ka_2+\ka_3}}_{(\br^1\si)\br^4(\br^2\tau)\br^3,\bt^U\bs^4\bt^{U'}\bs^3}*\eta^\bb_{\bp,\bq},
\end{align*}
where the second sum is over all $U\in\Om^{\mu,\nu}, \si\in\Si_{\mu}$ and $\tau\in{}^{\br^2}\Si_{\nu}$. By Remark~\ref{R2}, all the basis  elements appearing in the whole sum above are linearly independent.  By our assumptions this implies that all of the non-zero coefficients $q^{\ka,\bb}_{\br,\bs;\bp,\bq}$ appearing there are divisible by $m$, which is a contradiction. 
\end{proof}

\begin{Remark} \label{RCounterExample} %{\rm \cite{}}%{\bf ()}
{\rm 
For an arbitrary algebra \(A\) with double centralizer idempotent \(e\), it is not the case that $\xi^e$ is in general a double centralizer idempotent for \(T^A_\a(n,d)\). For example, take arbitrary \(n \in \Z_{>0}\), and consider the case \(A = A_{\bar 0} = M_2(\k)\), \(\a = \spa ( E_{11}, E_{22})\), \(\c = \spa (E_{12}, E_{21})\), and \(e = E_{11}\). Then \(e\) is clearly a double centralizer idempotent for \(A\).

The element \(\eta^{E_{12}, E_{12}}_{11,11} = 2\xi^{E_{12}, E_{12}}_{11,11}\) is indivisible in \(T^A_\a(n,2)\), and \(\{E_{11}, E_{21}\}\) is a basis for \(Ae\). For \(\bb \in \{E_{11}, E_{21}\}^2\) and \(\br, \bs \in [1,n]^2\), we have that \(\eta^{E_{12}, E_{12}}_{11,11} \cdot \eta^{\bb}_{\br, \bs}=0\) unless \(\bb = (E_{21},E_{21})\) and \(\br = (1,1)\). Then, if \(s_1 = s_2\) we have
\begin{align*}
\eta^{E_{12}, E_{12}}_{11,11} \cdot \eta^{E_{21}, E_{21}}_{11, \bs} = 
(2 \xi^{E_{12}, E_{12}}_{11,11}) \cdot (2  \xi^{E_{21}, E_{21}}_{11, \bs} )
= 4 \xi^{E_{11}, E_{11}}_{11, \bs}
= 4 \eta^{E_{11}, E_{11}}_{11, \bs}.
\end{align*}
If \(s_1 \neq s_2\) we have
\begin{align*}
\eta^{E_{12}, E_{12}}_{11,11} \cdot \eta^{E_{21}, E_{21}}_{11, \bs} = 
(2 \xi^{E_{12}, E_{12}}_{11,11}) \cdot (  \xi^{E_{21}, E_{21}}_{11, \bs} )
= 2 \xi^{E_{11}, E_{11}}_{11, \bs}
= 2 \eta^{E_{11}, E_{11}}_{11, \bs}.
\end{align*}
This implies that \(\la_{\eta^{E_{12}, E_{12}}_{11,11}}\) is divisible in \(\End_{T^{\bar A}_{\bar \a}(n,2)}(T^A_\a(n,2)\xi^e)\), so \(\xi^e\) is not sound, and thus not a double centralizer idempotent by Lemma \ref{LK}.
}
\end{Remark}

\end{document}